\documentclass[10pt, a4paper]{article}

\pdfoutput=1

\usepackage[utf8]{inputenc}
\usepackage[T1]{fontenc}
\usepackage[colorlinks]{hyperref}
\usepackage{microtype}
\usepackage{xcolor}

\usepackage{enumitem}

\hypersetup{
  linkcolor=[rgb]{0.3,0.3,0.6},
  citecolor=[rgb]{0.2, 0.6, 0.2},
  urlcolor=[rgb]{0.6, 0.2, 0.2}
}

\usepackage{mathtools, amsthm, amsfonts, amssymb}

\newcommand{\parag}[1]{\textbf{#1.}}

\usepackage{mathrsfs}
\usepackage{pbox}

\usepackage{booktabs}
\usepackage{braket}
\usepackage{commath}
\usepackage{caption}
\usepackage{titlesec}
\titlelabel{\thetitle. }
\titleformat*{\section}{\large\bfseries}

\usepackage[capitalise]{cleveref}
\usepackage{accents}

\usepackage{tikz}
\usetikzlibrary{positioning}

\usepackage{tikz-cd}

\theoremstyle{plain}
\newtheorem{theorem}{Theorem}[section]
\newtheorem*{theorem*}{Theorem}
\newtheorem{proposition}[theorem]{Proposition}
\newtheorem{lemma}[theorem]{Lemma}
\newtheorem{corollary}[theorem]{Corollary}

\theoremstyle{definition}
\newtheorem{definition}[theorem]{Definition}

\newtheorem*{problem*}{Problem}

\newtheorem{remark}[theorem]{Remark}
\newtheorem{example}[theorem]{Example}

\DeclarePairedDelimiter\floor{\lfloor}{\rfloor}

\newcommand{\regularize}[1]{\underaccent{\wtilde}{#1}}

\DeclareMathOperator{\supp}{supp}

\DeclareMathOperator{\rank}{R}
\DeclareMathOperator{\subrank}{Q}
\DeclareMathOperator{\bordersubrank}{\underline{Q}}
\DeclareMathAccent{\wtilde}{\mathord}{largesymbols}{"65}
\DeclareMathOperator{\asymprank}{\underaccent{\wtilde}{R}}
\DeclareMathOperator{\asympsubrank}{\underaccent{\wtilde}{Q}}
\DeclareMathOperator{\borderrank}{\underline{R}}

\newcommand{\FF}{\mathbb{F}}
\newcommand{\CC}{\mathbb{C}}

\newcommand{\NN}{\mathbb{N}}
\newcommand{\GL}{\mathrm{GL}}
\newcommand{\End}{\mathrm{End}}

\newcommand{\id}{\mathrm{Id}}

\newcommand{\ZZ}{\mathbb{Z}}
\newcommand{\RR}{\mathbb{R}}
\newcommand{\bases}{\mathcal{C}}
\newcommand{\filt}{\mathcal{F}}

\DeclareMathOperator{\type}{type}

\DeclareMathOperator{\Span}{Span}

\DeclareMathOperator{\Hom}{Hom}

\DeclareMathOperator{\diag}{diag}

\newcommand{\ketbra}[2]{\ket{#1}\!\!\bra{#2}}

\DeclareMathOperator{\Id}{Id}

\newcommand{\defin}[1]{\emph{#1}}

\newcommand{\CW}{\mathrm{CW}}

\newcommand{\asympcombsubrank}{\underaccent{\wtilde}{\mathrm{Q}}}

\newcommand{\degengeq}{\unrhd}
\newcommand{\degenleq}{\unlhd}

\newcommand{\asympgeq}{\gtrsim}
\newcommand{\asympleq}{\lesssim}

\newcommand{\eps}{\varepsilon}

\DeclareMathOperator{\slicerank}{slicerank}
\DeclareMathOperator{\multislicerank}{multislicerank}
\newcommand{\s}{\mathrm{s}}

\newcommand{\prob}{\mathcal{P}}

\DeclareMathOperator{\hook}{hook}

\DeclareMathOperator{\downset}{\downarrow\hspace{-0.1em}}

\newcommand{\efftens}{\mathcal{T}}

\DeclareMathOperator{\flatten}{flatten}

\newcommand{\HH}{\mathcal{H}}
\newcommand{\states}{\mathcal{S}}
\newcommand{\nc}{\mathrm{nc}}
\DeclareMathOperator{\Tr}{Tr}

\newcommand{\SSS}{\mathbb{S}}

\newcommand{\domleq}{\preccurlyeq}

\newcommand{\complete}{\mathrm{c}} %

\newcommand{\cont}{\mathcal{C}}

\newcommand{\semiring}[1]{\mathcal{#1}}

\newcommand{\conv}{\mathrm{conv}}

\newcommand{\msr}{\mathrm{MSR}^{\sim}}
\newcommand{\sr}{\mathrm{SR}^{\sim}}

\DeclareMathOperator{\SL}{SL}
\DeclareMathOperator{\SU}{SU}

\DeclareMathOperator{\instab}{instab}

\begin{document}

\vspace*{1em}
\begin{center}
\Large\textbf{Universal points in the\\[1ex] asymptotic spectrum of tensors}\par
\vspace{1.5em}
\large Matthias Christandl, Péter Vrana and Jeroen Zuiddam\par
\end{center}
\vspace{0.5em}

\begin{abstract}
The \emph{asymptotic restriction problem} for tensors is to decide,
given tensors $s$ and $t$, whether the $n$th tensor power of $s$ can be obtained from the $(n+o(n))$th tensor power of $t$ by applying linear maps to the tensor legs (this we call restriction), when~$n$ goes to infinity.
In this context, Volker Strassen, striving to understand the complexity of matrix multiplication, introduced in 1986 
the asymptotic spectrum of tensors.
Essentially, the asymptotic restriction problem for a family of tensors~$\semiring{X}$, closed under direct sum and tensor product, reduces to finding all maps from~$\semiring{X}$ to the nonnegative reals that are monotone under restriction,  normalised on diagonal tensors, additive under direct sum and multiplicative under tensor product, which Strassen named spectral points. 
Spectral points are by definition an upper bound on asymptotic subrank and a lower bound on asymptotic rank.
Strassen created the support functionals, which are spectral points for oblique tensors, a strict subfamily of all tensors.

\emph{Universal spectral points} are spectral points for the family of all tensors.
The construction of nontrivial universal spectral points has been an open problem for more than thirty years.
We construct for the first time a family of nontrivial universal spectral points over the complex numbers, using the theory of quantum entropy and covariants: the \emph{quantum functionals}. In the process we connect the asymptotic spectrum of all tensors to the quantum marginal problem and to the entanglement polytope. 
In en\-tang\-lement theory, our results amount to the first construction of additive entanglement monotones for the class of stochastic local operations and classical communication.

To demonstrate the asymptotic spectrum, we reprove (in hindsight) recent results on the \emph{cap set problem} by reducing this problem to computing the lowest point in the asymptotic spectrum of the reduced polynomial multiplication tensor, a prime example of Strassen.
A better understanding of our universal spectral points construction may lead to further progress on related combinatorial questions.  
We additionally show that the quantum functionals characterise asymptotic slice rank for complex tensors. %
\end{abstract}

\vspace{1em}
{\small\noindent
\parag{Keywords} asymptotic restriction, asymptotic spectrum, tensors, fast matrix multiplication, cap set problem, reduced polynomial multiplication, stochastic local operations and classical communication (slocc), entanglement monotones, quantum entropy, moment polytope\\[1ex]
\noindent
\parag{MSC} 
15A69, %
14L24, %
68Q17 %
}
\newpage
\tableofcontents
\newpage
\section{Introduction}

\subsection{The asymptotic restriction problem}
We study the asymptotic restriction problem, following the pioneering work of Volker Strassen \cite{Strassen:1986:AST:1382439.1382931, strassen1987relative, strassen1988asymptotic, strassen1991degeneration}. The asymptotic restriction problem is a problem about multilinear maps
$f : \FF^{n_1} \times \cdots \times \FF^{n_k} \to \FF$
over an arbitrary field $\FF$. Letting $(e_1, \ldots, e_{n_i})$ be the standard basis of $\FF^{n_i}$, one may equivalently think of~$f$ as the $k$-tensor $t \in \FF^{n_1} \otimes \cdots \otimes \FF^{n_k}$ defined~by 
$t = \sum f(e_{a_1}, \ldots, e_{a_k}) \,e_{a_1} \otimes \cdots \otimes e_{a_k}$
where $a_i$ goes over $\{1, \ldots, n_i\}$. %
To state the asymptotic restriction problem we need the concepts restriction and tensor product.  %
Let $f : \FF^{n_1} \times \cdots \times \FF^{n_k} \to \FF$ and ${g : \FF^{m_1} \times \cdots \times \FF^{m_k} \to \FF}$ be multilinear maps. 
We say \defin{$f$ restricts to~$g$}, and write~$f\geq g$, if there are linear maps $A_i : \FF^{m_i} \to \FF^{n_i}$ such that $g = f \circ (A_1, \ldots, A_k)$ where $\circ$ denotes composition. 
We naturally define the \defin{tensor product}
$f \otimes g$ as the multilinear map 
$(\FF^{n_1} \otimes \FF^{m_1}) \times \cdots \times (\FF^{n_k} \otimes \FF^{m_k}) \to \FF$ defined by
$(v_1 \otimes w_1, \ldots, v_k \otimes w_k) \mapsto f(v_1, \ldots, v_k) g(w_1, \ldots, w_k)$.
We say~$f$ \defin{restricts asymptotically to} $g$, written~$f \asympgeq g$, if there is a sequence of natural numbers $a(n) \in o(n)$ %
such that 
\[
f^{\otimes n + a(n)} \geq g^{\otimes n} \quad\textnormal{when}\quad n \to \infty.
\]
The asymptotic restriction problem is:  %
given $f$ and $g$, decide whether $f\asympgeq g$.

Applications of the asymptotic restriction problem include computing the computational complexity of matrix multiplication in algebraic complexity theory \cite{ALDER1981201, blaser20015, Burgisser:2011:GCT:1993636.1993704,landsberg2014new, strassen1969gaussian, MR1056627,stothers2010complexity,MR2961552,le2014powers, cohn2003group, cohn2005group} (see also \cite{burgisser1997algebraic, landsberg2012tensors, blaser2013fast, landsberg2017}), deciding the feasibility of an asymptotic transformation between pure quantum states via stochastic local operations and classical communication (slocc) in quantum information theory \cite{bennett2000exact, MR1804183, MR1910235,MR2515619}, bounding the size of combinatorial structures like cap sets and tri-colored sum-free sets in additive combinatorics \cite{MR2031694, tao2008structure, Alon2013, MR3583357,MR3583358,tao,MR3631613,kleinberg2016growth,sawin}, and bounding the query complexity of certain properties in algebraic property testing \cite{Kaufman:2008:APT:1374376.1374434, bhattacharyya2010testing, Shapira:2009:GCT:1536414.1536438,  Bhattacharyya2015, Haviv2017, fu_et_al:LIPIcs:2014:4730}.
There are naturally two directions in the asymptotic restriction problem, namely finding
(1) \emph{constructions}, i.e.~matrices that carry out~$f\asympgeq g$, and
(2)~\emph{obstructions}, i.e.~certificates that prohibit $f\asympgeq g$.
For constructions one should think of fast matrix multiplication algorithms or efficient quantum protocols.
For obstructions one should think of lower bounds in the sense of computational complexity theory.
Strassen introduced in 1986 the theory of asymptotic spectra of tensors to understand the asymptotic restriction problem \cite{Strassen:1986:AST:1382439.1382931, strassen1988asymptotic}. Deferring the details to the next subsection, this can be viewed as the theory of obstructions in the above sense.
A remarkable result of this theory is that the asymptotic restriction problem for a family of tensors $\semiring{X}$ that is closed under direct sum and tensor product and contains the diagonal tensors $\langle n\rangle$, reduces to finding all maps $\semiring{X} \to \RR_{\geq 0}$ that are
\begin{enumerate}[label=(\alph*)]
\item monotone under restriction $\geq$\label{a}
\item multiplicative under tensor product~$\otimes$\label{b}
\item additive under direct sum $\oplus$\label{c}
\item normalised to have value~$n$ at the unit tensor $\langle n \rangle$.\label{d}
\end{enumerate} %
Such maps are called \defin{spectral points}.
Here the \defin{direct sum} $f \oplus g$ is defined naturally as the multilinear map $(\FF^{n_1} \oplus \FF^{m_1}) \times \cdots \times (\FF^{n_k} \oplus \FF^{m_k}) \to \FF$ such that
$(v_1 + w_1, \ldots, v_k + w_k) \mapsto f(v_1, \ldots, v_k) + f(w_1, \ldots, w_k)$
where $v_i \in \FF^{n_i}, w_i \in \FF^{m_i}$, %
and  the \defin{unit tensor} $\langle n \rangle$ for $n\in \NN$ is defined as the multilinear map $(\FF^n)^{\times k} \to \FF$ that maps $(e_{i_1}, \ldots, e_{i_k})$ to~1 if $i_1 = \cdots = i_k$ and to 0 otherwise. 

Properties \ref{a} and \ref{b} are natural properties to obtain an obstruction. Namely, suppose $\xi$ is such a map, and let $f,g\in \semiring{X}$. If $f \asympgeq g$, then by definition $f^{\otimes n + o(n)} \geq g^{\otimes n}$, and \ref{a} and \ref{b} imply $\xi(f)^{n + o(n)} = \xi(f^{n + o(n)}) \geq \xi(g^n) = \xi(g)^n$, which implies $\xi(f) \geq \xi(g)$. Turning this around, if $\xi(f) < \xi(g)$ then \emph{not} $f\asympgeq g$, so~$\xi$ yields an obstruction to $f\asympgeq g$.
Strassen in~\cite{strassen1991degeneration} created a family of spectral points.
Let $\Theta$ be the set of all probability distributions on $\{1, 2, \ldots, k\}$.
Strassen defined a family of maps $\zeta^\theta : \{\textnormal{$k$-tensors}\} \to \RR_{\geq 0}$ parametrised by $\theta \in \Theta$, named the support functionals, and he proved that the $\zeta^\theta$ are spectral points for the family $\semiring{X}$ of oblique tensors,  tensors whose support in some basis is an antichain, a strict (and nongeneric) subfamily of all tensors.
In \cite{MR2138544}, for such tensors, $\zeta^\theta$ has been given a formulation in terms of moment polytopes.

\parag{Our main result}
\defin{Universal spectral points} are spectral points for the family of \emph{all} tensors. The construction of nontrivial universal spectral points has been an open problem for more than thirty years.
We introduce maps $F_\theta : \{\textnormal{complex $k$-tensors}\}\to \RR_{\geq 0}$ called the \defin{quantum functionals} and we prove that they are universal spectral points.
The quantum functionals are defined as follows (we will carefully define the quantum concepts later). For any $\theta \in \Theta$, define
\begin{align*}
F_\theta(t) &= 2^{E_{\theta}(t)}\\
E_\theta(t) &= \sup_{g_i} \sum_{i=1}^k \theta(i)\, H\bigl( \Tr_i\, (g_1, \ldots, g_k) \cdot t \bigr)
\end{align*}
where the supremum goes over invertible maps $g_i \in \GL(\CC^{n_i})$, $H$ denotes the quantum entropy i.e.~von Neumann entropy, and $\Tr_i\, (g_1, \ldots, g_k) \cdot t$ denotes a partial trace of $(g_1, \ldots, g_k) \cdot t$ interpreted as a pure quantum state. To prove properties \ref{a}--\ref{d} we draw from the theory on invariants, quantum entropy, entanglement polytopes, Kronecker coefficients and Littlewood--Richardson coefficients. %
Let us briefly sketch the connection to moment polytopes, which in this context are called  entanglement polytopes. Given a tensor~$s$, let $r_1(s), \ldots, r_k(s)$ be the single-system quantum marginal entropies of the tensor $s$ viewed as a quantum state. 
Let $G$ be the group $\GL(\CC^{n_1}) \times \cdots \times \GL(\CC^{n_k})$.
For a tensor~$t$,  define the set $\Delta_t = \{(r_1(s), \ldots, r_k(s)) \mid s \in \overline{G\cdot t} \}$, where $\overline{G \cdot t}$ denotes the Euclidean closure (or equivalently Zariski closure) of the orbit $G\cdot t$. It is a nontrivial fact that $\Delta_t$ is a polytope, named the entanglement polytope.
Then $E_\theta(t)$ equals the following optimisation over the entanglement polytope,
\[
E_\theta(t) = \sup \Bigl\{ \sum_{i=1}^k \theta(i)\, H(\lambda^{(i)}) \,\Big|\, (\lambda^{(1)}, \ldots, \lambda^{(k)}) \in \Delta_t \Bigr\},
\]
which is a convex optimisation problem.

The definition of $F_\theta$ can in fact be extended to probability distributions $\theta$ on \emph{subsets} of $\{1,2, \ldots, k\}$ in several ways. We explore these extensions and study the properties \ref{a}--\ref{d}. One such extension, the lower quantum functional, satisfies \ref{a} and \ref{d} and is super-additive and super-multiplicative. Another extension, the upper quantum functional, satisfies \ref{a} and \ref{d} and is sub-additive and sub-multiplicative for $\theta$ that are what we call noncrossing. (The Strassen support functionals in fact similarly come in an upper and a lower version.)

\parag{Asymptotic rank and asymptotic subrank}
In applications (e.g.~the complexity of matrix multiplication) one is often interested in asymptotic restriction to or from a unit tensor. One commonly defines the \defin{tensor rank}
$\rank(f) = \min \{ r \in \NN \mid f \leq \langle r \rangle \}$, and the \defin{subrank} $\subrank(f) = \max \{ s \in \NN \mid \langle s \rangle \leq f \}$,
and asymptotically the \defin{asymptotic rank} $\asymprank(f) = \lim_{n \to \infty} \rank(f^{\otimes n})^{1/n}$ and the \defin{asymptotic subrank} $\asympsubrank(f) = \lim_{n\to \infty} \subrank(f^{\otimes n})^{1/n}$.
Clearly, maps that satisfy~\ref{a},~\ref{d} and super-multiplicativity are lower bounds on asymptotic rank, and maps that satisfy \ref{a},~\ref{d} and sub-multiplicativity are upper bounds on asymptotic subrank. Spectral points (maps satisfying \ref{a}--\ref{d}) are thus between asymptotic subrank and asymptotic rank. 
The defining expression of asymptotic rank does not suggest any algorithm for computing its value other than computing the rank for high tensor powers. In information theory, such an expression is called a multi-letter formula, and maps satisfying \ref{a}--\ref{d} are called single-letter formulas.

\parag{Cap sets and slice rank}
To demonstrate an application of the asymptotic spectrum we go on a brief combinatorial excursion to the cap set problem. (Full details are in \cref{capsetsec}.)
A subset $A \subseteq (\ZZ/3\ZZ)^n$ is called a \defin{cap set} if any line in $A$ is a point, a line being a triple of points of the form $(u, u+v, u+2v)$. The \defin{cap set problem} is to decide whether the maximal size of a cap set in $(\ZZ/3\ZZ)^n$ grows like $3^{n - o(n)}$ or like~$c^{n+o(n)}$ for some $c < 3$ when $n\to \infty$. Gijswijt and Ellenberg in \cite{MR3583358}, inspired by the work of Croot, Lev and Pach in~\cite{MR3583357}, settled this problem, showing that $c\leq 3 (207 + 33 \sqrt{33})^{1/3} /8 \approx 2.755$.
Tao realised in \cite{tao} that the cap set problem may naturally be phrased as the problem of computing the size of the largest main diagonal in powers of the \defin{cap set tensor} $\sum_{\alpha} e_{\alpha_1} \otimes e_{\alpha_2} \otimes e_{\alpha_3}$
where the sum is over $\alpha_1, \alpha_2, \alpha_3 \in \ZZ/3\ZZ$ with $\alpha_1 + \alpha_2 + \alpha_3 = 0$. Here main diagonal refers to a subset $A$ of the basis elements such that restricting to $A\times A \times A$ gives the tensor $\sum_{v \in A} v \otimes v \otimes v$. A main diagonal is essentially a unit tensor, so it is sufficient to upper bound the asymptotic subrank of the cap set tensor interpreted as a tensor over $\FF_3$.
We show (in hindsight) that the cap set tensor is in the $\GL_3(\FF_3)^{\times 3}$-orbit of the structure tensor of the algebra~$\FF_3[x]/(x^3)$. %
This implies that the asymptotic spectrum of the cap set tensor and the asymptotic spectrum of $\FF_3[x]/(x^3)$ coincide. The tensor $\FF_3[x]/(x^3)$ is the prime example in \cite{strassen1991degeneration} for which Strassen computes the whole asymptotic spectrum. The minimal point in this asymptotic spectrum, which corresponds to the asymptotic subrank, is computed by Strassen to be $3 (207 + 33 \sqrt{33})^{1/3} /8 \approx 2.755$ (see~\cite[Table 1]{strassen1991degeneration}). This reproves the bound by Ellenberg--Gijswijt.

Although our universal spectral points, the quantum functionals, do not play a role for cap sets (we cannot work over $\CC$ for this problem), we think that a better understanding of our universal spectral points construction may lead to further progress on related combinatorial questions.

In the study of the cap set problem, slice rank and multi-slice rank were introduced as upper bounds on subrank \cite{tao, naslund2017multi}. 
We show that asymptotically the (upper) quantum functionals and the Strassen (upper) support functionals are an upper bound on slice rank and multi-slice rank. As a consequence we prove that for the family of tight 3-tensors asymptotic subrank and asymptotic slice rank coincide. For complex tensors we characterise the asymptotic slice rank in terms of the quantum functionals.

\subsection{Asymptotic spectra for tensors}\label{asympspec}

To understand why we focus on maps satisfying the properties \ref{a}--\ref{d} we have to give a brief introduction to Strassen's theory of asymptotic spectra for tensors.

We begin by putting an equivalence relation on tensors to get rid of trivialities.
Restriction $\geq$ and asymptotic restriction $\asympgeq$ are both preorders (reflexive and transitive).
We say \defin{$f$ is isomorphic to $g$}, and write $f\cong g$, if there are bijective linear maps $A_i : \FF^{m_i} \to \FF^{n_i}$ such that $g = f \circ (A_1, \ldots, A_k)$. We say $f$ and $g$ are \defin{equivalent}, and write $f \sim g$, if there are null maps $f_0 : \FF^{a_1} \times \cdots \times \FF^{a_k} \to \FF : x \mapsto 0$ and $g_0 : \FF^{b_1} \times \cdots \times \FF^{b_k} \to \FF : x \mapsto 0$ such that $f \oplus f_0 \cong g \oplus g_0$. The equivalence relation~$\sim$ is in fact the equivalence relation generated by the restriction preorder~$\geq$.
Let $\efftens$ be the set of $\sim$-equivalence classes of $\FF$-multilinear maps of order~$k$. Direct sum and tensor product naturally carry over to $\efftens$\!, and $\efftens$ becomes a semiring with additive unit $\langle 0\rangle$ and multiplicative unit $\langle 1\rangle$ (more precisely, the equivalence classes of those tensors).
Restriction $\geq$ induces a partial order on $\efftens$ (reflexive, antisymmetric, transitive), and asymptotic restriction~$\asympgeq$ induces a preorder on $\efftens$. Both behave well with respect to the semiring operations, and naturally $n\geq m$ if and only if $\langle n\rangle \geq \langle m\rangle$.

The theory of asymptotic spectra revolves around the following theorem proved by Strassen \cite{Strassen:1986:AST:1382439.1382931, strassen1988asymptotic}. Given a topological space $\Delta$ we denote by~$\cont^+(\Delta)$ the semiring of continuous maps $\Delta \to \RR_{\geq 0}$.

\begin{theorem}[Spectral theorem]\label{spectraltheorem}
For any semiring $\semiring{X} \subseteq \efftens$ there are
\begin{enumerate}
\item a compact space $\Delta$
\item a homomorphism $\phi: \semiring{X} \to \cont^+(\Delta)$ of semirings
\end{enumerate}
such that $\phi(\semiring{X})$ separates points and such that
$a \asympleq b$ if and only if $\phi(a) \leq \phi(b)$ pointwise on $\Delta$, 
and the pair $(\Delta, \phi)$ is essentially unique. We call $(\Delta,\phi)$ an asymptotic spectrum for $\semiring{X}$. Explicitly, $(\Delta, \phi)$ may be taken as follows:
\begin{enumerate}
\item $\Delta = \{ \textnormal{$\geq$-monotone homomorphisms $\semiring{X} \to \RR_{\geq 0}$}\}$.
\item $\phi: \semiring{X} \to \cont^+(\Delta) : a \mapsto \hat{a}$ where $\hat{a}  : \Delta \to \RR_{\geq 0} : \xi \mapsto \xi(a)$.
\end{enumerate}
We call this pair $(\Delta,\phi)$ the asymptotic spectrum of $\semiring{X}$ and refer to it as $\Delta(\semiring{X})$.
\end{theorem}

\cref{spectraltheorem} is proved by a nontrivial reduction to the structure theory of Stone--Kadison--Dubois for a certain class of ordered rings (see \cite{MR707730}). We will not go into the details of this proof here, nor do we elaborate on how $(\Delta, \phi)$ is unique. We have taken the name spectral theorem from the talk \cite{strassentalk}.

\begin{remark}\label{degenremark}
We note that $\Delta(\semiring{X})$ may equivalently be defined with degeneration~$\degengeq$ instead of restriction $\geq$.
Over~$\CC$, we say $f$ \defin{degenerates to} $g$, written, $f \degengeq g$, if $g$ is in the Euclidean closure (or equivalently Zariski closure) of the orbit $\GL(\CC^{n_1})\times \cdots \times \GL(\CC^{n_k}) \cdot f$. It is a nontrivial fact from algebraic geometry (see \cite[Lemma~III.2.3.1]{kraft1984geometrische} or \cite{burgisser1997algebraic}) that there is a degeneration $f\degengeq g$ if and only if there are linear maps $A_i \in \GL(\CC(\eps)^{n_i})$ such that $(A_1, \ldots, A_k) \cdot f = g + \eps^1 g_1 + \cdots + \eps^e g_e$ for some elements $g_1, \ldots, g_e$, where $\CC(\eps)$ is the field $\CC$ extended with the formal variable $\eps$. The latter definition of degeneration is valid when $\CC$ is replaced by an arbitrary field $\FF$ and that is how degeneration is defined for an arbitrary field. 
Degeneration is weaker than restriction: $f\geq g$ implies $f\degengeq g$. Asymptotically, however, the notions coincide: $f \asympgeq g$ if and only if $f^{\otimes n + o(n)} \degengeq g^{\otimes n}$.
We mention that, analogous to restriction, degeneration gives rise to \defin{border rank} and \defin{border subrank}, $\borderrank(f) = \min\{r \in \NN \mid f \degenleq \langle r \rangle \}$, $\bordersubrank(f) = \max \{s \in \NN \mid \langle s \rangle \degenleq f\}$.
\end{remark}
\begin{remark}
Let $\semiring{X}$ be a family of tensors over $\FF$. Let $\FF \subseteq K$ be a field extension. We may view $\semiring{X}$ as a family of tensors $\semiring{X}^K$ over $K$. The asymptotic spectra of $\semiring{X}$ and $\semiring{X}^K$ are equal \cite[Theorem 3.10]{strassen1988asymptotic}.
In particular, if $t \in \{0,1\}^{n_1} \times \cdots \times \{0,1\}^{n_k}$, then $t$ can be viewed as a tensor $t^\FF$ in $\FF^{n_1} \otimes \cdots \otimes \FF^{n_k}$ over any field $\FF$, and by the above, if $K \subseteq \FF$ is the prime subfield of $\FF$, then $\Delta(t^K) = \Delta(t^\FF)$, and e.g.~$\asymprank(t^K) = \asymprank(t^\FF)$. A well-known example is that the exponent $\omega_\FF$ of matrix multiplication over $\FF$ depends only on the characteristic of~$\FF$.
\end{remark}

We reduced the asymptotic restriction problem $a \asympleq b$ to the problem of computing the asymptotic spectrum of $\efftens$ or a subsemiring $\semiring{X}\subseteq\efftens$.  
An element $\xi \in \Delta(\semiring{X})$ is called a \defin{spectral point}.
It is clear that constructing explicit elements in the asymptotic spectrum~$\Delta(\efftens)$ of all tensors $\efftens$ is the holy grail in the development of the theory of asymptotic spectra of tensors. Elements of~$\Delta(\efftens)$ are called \defin{universal spectral points}, and they correspond precisely to maps $\xi : \{\textnormal{$k$-tensors over $\FF$}\} \to \RR_{\geq 0}$ satisfying $\xi(s \oplus t) = \xi(s) + \xi(t)$, $\xi(s\otimes t) = \xi(s) \xi(t)$, $\xi(\langle 1 \rangle) = 1$, and $\xi(s) \leq \xi(t)$ whenever~$s \leq t$ (i.e.~properties \ref{a}--\ref{d}). 

\begin{example}[Gauge points]
In \cite[Equation 3.10]{strassen1988asymptotic} $k$ universal spectral points are given. Namely, given a multilinear map $f : \FF^{n_1} \times \cdots \times \FF^{n_k} \to \FF$, let $V_1 = \{ v \in \FF^{n_1} \mid f(v, \FF^{n_2}, \ldots, \FF^{n_k}) \neq \{0\} \}$
and similarly define $V_i$ for $i \geq 2$. Define $\zeta_{(i)} : \efftens \to \RR_{\geq 0} : f \mapsto \dim V_i$.
The maps $\zeta_{(i)}$ are universal spectral points and are named \defin{gauge points}.
\end{example}

In some applications it is useful to think of the asymptotic spectrum as a compact subspace of $\RR^X$.

\begin{definition}
Let $X\subseteq \efftens$ be a set. Let $\mathcal{X} \subseteq \efftens$ be the semiring generated by $X$. A~homomorphism $\xi: \mathcal{X} \to \RR_{\geq 0}$ is determined by the values $(\xi(x) : x\in X) \in \RR_{\geq 0}^X$. We may thus identify the asymptotic spectrum of the semiring generated by $X$ with a subspace of $\RR_{\geq0}^X$. 
\begin{enumerate}
\item $\Delta(X) = \{\textnormal{$\xi \in \RR^X$ generating a $\geq$-monotone homomorphism $\semiring{X} \to \RR_{\geq 0}$}\}$.
\item $\phi(a)(\xi) = \xi(a)$.
\end{enumerate}
We call $(\Delta(X), \phi)$ \defin{the asymptotic spectrum of $X$}, and write $\Delta(X)$.
When $X = \{a\}$ we will denote $\Delta(X)$ by $\Delta(a)$.
\end{definition}

We finish by stating the important relationship between the asymptotic spectrum and the asymptotic (sub)rank, which follows from \cref{spectraltheorem} (see \cite[Theorem 3.8]{strassen1988asymptotic}).

\begin{proposition}%
\label{minmax}
Let $\semiring{X} \subseteq \efftens$ be a semiring.
Let $(\Delta, \phi)$ be an asymptotic spectrum of $\semiring{X}$.
Let~$a \in \semiring{X}$.
Then
\[
\asympsubrank(a) = \min_{\xi \in \Delta} \phi(a)(\xi)
\qquad
\asymprank(a) = \max_{\xi \in \Delta} \phi(a)(\xi).
\]
\end{proposition}

\begin{remark}
Obviously $\asymprank$ and $\asympsubrank$ are $\geq$-monotones and have value $n$ on $\langle n \rangle$. They are not universal spectral points however. Namely, the asymptotic rank of $e_1 \otimes e_1 \otimes 1 + e_2 \otimes e_2 \otimes 1$, $e_1 \otimes 1 \otimes e_1 + e_2 \otimes 1 \otimes e_2$ and $1 \otimes e_1 \otimes e_1 + 1 \otimes e_2 \otimes e_2$ is 2, whereas the tensor product equals the matrix multiplication tensor whose asymptotic rank is strictly smaller than $2^3$. With the same tensors one shows that asymptotic subrank is not multiplicative.
\end{remark}

\subsection{This paper}

\parag{Main results}
The main results of this paper are summarised as follows.
\begin{enumerate}
\item We construct for the first time a nontrivial family of maps
\[
F_\theta : \{\textnormal{$k$-tensors}\}\to\RR_{\geq 0},
\]
named quantum functionals,
that are nonincreasing under restriction $\geq$, normalised on the unit tensor~$\langle r \rangle$, additive under direct sum $\oplus$ and multiplicative under tensor product~$\otimes$, i.e.~universal spectral points, advancing Strassen's theory of asymptotic spectra.
\item We connect Strassen's asymptotic spectra to entanglement polytopes and the quantum marginal problem. This is to our knowledge the first time that information about asymptotic transformations is obtained from entanglement polytopes, as opposed to information about single-copy transformations.
\item We put recent progress on the cap set problem in the framework of Strassen, which may prove useful in solving variations on the cap set problem. We characterise asymptotic slice rank in terms of the quantum functionals %
and show that asymptotic slice rank coincides with asymptotic subrank for tight 3-tensors.
\end{enumerate}

\parag{Paper overview}
In \cref{strassenfunc} we set the scene by reviewing the theory of the Strassen upper and lower support functionals~$\zeta^\theta$ and~$\zeta_\theta$ introduced in \cite{strassen1991degeneration}. 
This is important background material for our work.
The support functionals allow us to compute elements in the asymptotic spectrum for a class of tensors called oblique tensors. Our exposition follows closely the exposition in \cite{strassen1991degeneration} with the exception that we consider multilinear maps $V_1 \times \cdots \times V_k \to \FF$ with $k\geq 3$ instead of bilinear maps $V_1 \times V_2 \to V_3$. Strassen already observed that his results generalise to the multilinear regime, but we will make this explicit for the sake of understanding and comparison to our results.

\cref{quantum} contains our main result. Here we introduce two new families of functionals over the complex numbers. These functionals are denoted by $F^{\theta}$ and~$F_\theta$ and called the \emph{upper and lower quantum functionals}. Both functionals are $\geq$-monotone. The upper quantum functional $F^{\theta}$ is sub-additive and sub-multiplicative when $\theta$ is what we call noncrossing. The lower quantum functional~$F_\theta$ is super-additive and super-multiplicative. For singleton $\theta$ the two functionals coincide and thus yield an additive, multiplicative $\degengeq$-monotone, a universal spectral point. 

In \cref{tightfreegeneric} we consider several families of tensors. %
We introduce the family of free tensors. We show that for free tensors, the Strassen upper support functional coincides with the quantum functionals for singleton $\theta$. %
This is useful computationally, since the upper support functional is defined as a minimisation while the quantum functionals are defined as a maximisation. 
We next compute generic values of the quantum functionals for tensor formats in which a quantum state with completely mixed marginals exists. Our results extend over $\CC$ the results by Verena Tobler~\cite{tobler} on the generic value of the upper support functional.
Finally, we reprove recent results on the cap set problem, by reducing this problem to a result of Strassen on reduced polynomial multiplication. %

In \cref{slicesec} we show that asymptotically the Strassen upper support functional and the upper quantum functional are upper bounds for slice rank and multi-slice rank, recently introduced $\geq$-monotones that are neither super-multiplicative nor sub-multiplicative. As a consequence, we find that slice rank coincides asymptotically with subrank for tight 3-tensors.
We additionally show that the quantum functionals characterise asymptotic slice rank for complex tensors.

\section{Strassen support functionals}\label{strassenfunc}

In this section the field $\FF$ is arbitrary.  After having introduced the concept of the asymptotic spectrum of tensors in \cite{strassen1988asymptotic}, Strassen constructed a nontrivial family of spectral points in the asymptotic spectrum of \emph{oblique} tensors in~\cite{strassen1991degeneration}. Oblique tensors are tensors for which the support is an antichain in some basis. In plain English, Strassen constructed a family of functions from $k$-tensors to~$\RR$ that are $\geq$-monotone, are normalised to attain value $r$ at $\langle r\rangle$, and, when restricted to oblique tensors, are additive under $\oplus$ and multiplicative under~$\otimes$. He calls his functions the \emph{support functionals}, because the support of a tensor plays an important role in the definition. Not all tensors are oblique, and obliqueness is not a generic property. Many tensors that are of interest in the field of algebraic complexity theory, however, turn out to be oblique, notably the structure tensor of the algebra of $n\times n$ matrices. 
This section is devoted to explaining the construction of these spectral points for oblique tensors. We do this not only to set the scene and provide benchmarks for our new functionals, but also because the Strassen support functionals remain relevant today, for example in the context of combinatorial problems like the cap set problem, as we will make clear in \cref{capset}.

The construction goes in four steps. First we define for probability distributions $\theta$ on $\{1, 2, \ldots, k\}$ the \emph{upper support functional} $\zeta^\theta$, which is $\geq$-monotone, normalised at $\langle r\rangle$, sub-additive under $\oplus$ and sub-multiplicative under $\otimes$. Second we define the \emph{lower support functional} $\zeta_\theta$, which is $\geq$-monotone, normalised at~$\langle r\rangle$, super-additive under $\oplus$ and super-multiplicative under $\otimes$. Third we show that $\zeta^\theta(t) \geq \zeta_\theta(t)$ for any $k$-tensor $t$. Fourth we show that for oblique tensors $t$ the upper support functional $\zeta^\theta(t)$ and lower support functional $\zeta_\theta(t)$ coincide, which gives a spectral point.

\parag{Notation} We use the following standard notation.
For any natural number~$n$ we write~$[n]$ for the set $\{1, 2, \ldots, n\}$. 
For any finite set $X$, let $\prob(X)$ be the set of all probability distributions on~$X$. For any probability distribution $P\in \prob(X)$ the \defin{Shannon entropy} of $P$ is defined as $H(P) = -\sum_{x\in X} P(x)\log_2 P(x)$ with $0 \log_2 0$ understood as~0. The \defin{support} of $P$ is $\supp P = \{x \in X \mid P(x)\neq 0\}$. Given finite sets $I_1,\ldots, I_k$ and a probability distribution $P \in \prob(I_1 \times \cdots \times I_k)$ on the product set $I_1 \times \cdots \times I_k$ we denote the \defin{marginal distribution} of $P$ on~$I_i$ by $P_i$, that is, $P_i(a) = \sum_{x : x_i = a} P(x)$ for any $a \in I_i$. 

We set some specific notation for this text. Let $k\geq3$. For $i\in [k]$, let $V_i$ be a vector space over $\FF$ of dimension $n_i$. Let $I_i$ be a finite set. Let $f : V_1 \times \cdots \times V_k \to \FF$ be a multilinear map. We order the set $[n_i]$ naturally by $1< 2< \cdots < n_i$. When the sets $I_1, \ldots, I_k$ are ordered, we give $I_1 \times \cdots \times I_k$ the \defin{product order}, defined by $x \leq y$ iff for all $i\in [k]$: $x_i \leq y_i$. %

\subsection{Upper support functional}

We begin with introducing the upper support functional. We naturally define the notion of the support of a multilinear map $V_1 \times \cdots \times V_k \to \FF$.

\begin{definition} 
Let $f: V_1 \times \cdots \times V_k \to \FF$ be a multilinear map. %
Let~$\bases(f)$ denote the set of $k$-tuples of bases for $V_1, \ldots, V_k$.
That is, each element $C \in \mathcal{C}(f)$ is a $k$-tuple %
$((v_{1, 1}, v_{1, 2}, \ldots, v_{1, n_1}), \ldots,  (v_{k, 1}, v_{k, 2}, \ldots, v_{k, n_k}))$ in which $(v_{i, 1}, v_{i, 2}, \ldots, v_{i, n_i})$ is a basis of $V_i$. Define the \defin{support of $f$ with respect to $C$} as
\[
\supp_C f \coloneqq \bigl\{ (\alpha_1, \ldots, \alpha_k) \in [n_1]\times \cdots \times [n_k] \,\big|\, f(v_{1, \alpha_1}, \ldots, v_{k, \alpha_k}) \neq 0 \bigr\}.
\]
\end{definition}

\begin{definition} Let $\theta \in \prob([k])$. Let $P \in \prob(I_1\times \cdots \times I_k)$ be a probability distribution. Define $H_\theta(P)$ as the $\theta$-weighted average of the Shannon entropies of the marginal distributions~$P_i$ of~$P$,
\[
H_\theta(P) = \sum_{i\in [k]} \theta(i) H(P_i).
\]
Let $\Phi \subseteq I_1 \times \cdots \times I_k$ be a nonempty subset. Define~$H_\theta(\Phi)$ as the maximum~$H_\theta(P)$ over all probability distributions $P \in \prob(\Phi)$,
\begin{equation}\label{convprog}
H_\theta(\Phi) = \max_{P \in \prob(\Phi)} H_\theta(P). %
\end{equation}
\end{definition}

\begin{definition}[Strassen upper support functional]\label{upper}
Let $\theta \in \prob([k])$. Let $f$ be a nonzero multilinear map. Define $\rho^\theta(f)$ by
\begin{align*}
\rho^\theta(f) &= \min_{\smash{C\in \bases(f)}} H_\theta(\supp_C f)
\intertext{and define $\zeta^\theta$ by}
\zeta^\theta(f) &= 2^{\rho^\theta(f)}.
\end{align*}
We call $\zeta^\theta$ the \defin{Strassen upper support functional} and we call $\rho^\theta$ the \defin{logarithmic Strassen upper support functional}. When $f$ is 0 we naturally define $\zeta^\theta(f) = 0$ and $\rho^\theta(f) = -\infty$.
\end{definition}

The key properties of the upper support functional are as follows.

\begin{theorem}[\cite{strassen1991degeneration}]\label{us} Let $f : V_1 \times \cdots \times V_k\to\FF$ and $g : W_1 \times \cdots \times W_k\to\FF$ be multilinear maps. Let~$\theta \in \prob([k])$. The Strassen upper support functional $\zeta^\theta$ has the following properties.
\begin{enumerate}
\item $\zeta^\theta\bigl(\langle r \rangle\bigr) = r$ for $r \in \NN$.
\item $\zeta^\theta(f \oplus g) = \zeta^\theta(f) + \zeta^\theta(g)$.
\item $\zeta^\theta(f \otimes g) \leq \zeta^\theta(f) \zeta^\theta(g)$.
\item If $f \geq g$, then $\zeta^\theta(f) \geq \zeta^\theta(g)$.
\item $0 \leq \zeta^\theta(f) \leq (\dim V_1)^{\theta(1)} \cdots (\dim V_k)^{\theta(k)}$.
\end{enumerate}
\end{theorem}

(Note that $\zeta^\theta$ is not just sub-additive, but even additive.)
One verifies directly that statement~1 and 5 of \cref{us} are true. 
Statement 2--4 can be proved with generalisations of the arguments in \cite[Section 2]{strassen1991degeneration}.

In the rest of this section we discuss an important alternative characterisation of the upper support functional in terms of filtrations. This alternative definition we need later to relate the upper and lower support functionals.

\begin{definition}
A \defin{proper filtration} or \defin{complete flag} of a vector space $U$ is a sequence $(U_\alpha)_{1 \leq \alpha \leq a}$ of subspaces of $U$ with $U = U_1 \supseteq \cdots \supseteq U_a = \{0\}$ and $\dim U_\alpha / U_{\alpha+1} = 1$ for $1 \leq \alpha < a$. (In this text all flags will be decreasing.)
We denote by $\filt_{\complete}(f)$ the set of $k$-tuples of complete flags $F = ((V_{1, \alpha_1})_{\alpha_1}, \ldots, (V_{k, \alpha_k})_{\alpha_k})$ of $V_1, \ldots, V_k$. For $F \in \filt_\complete(f)$ we define the support of $f$ with respect to $F$ as
\[
\supp_F f = \{ (\alpha_1, \ldots, \alpha_k) \mid f(V_{1, \alpha_1} \times \cdots \times V_{k, \alpha_k}) \neq 0\} \subseteq [n_1] \times \cdots \times [n_k].
\]
\end{definition}

The alternative characterisation of the upper support functional in terms of filtrations is as follows.
We refer to \cite[Section 2]{strassen1991degeneration} for the proof.

\begin{proposition}\label{upperfilt}
$\rho^\theta(f) = \min_{F \in \filt_\complete(f)} H_{\theta}(\supp_F f)$. %
\end{proposition}

\begin{remark}
The characterisation of $\rho^\theta$ as a minimisation over filtrations can be thought of as an asymptotic version of the filtration method explained in \cite[Exercise 15.13]{burgisser1997algebraic}.
\end{remark}

\subsection{Lower support functional}

The upper support functional $\zeta^\theta$ has a companion called the lower support functional $\zeta_\theta$ which we introduce here. We shift focus from the support to the maximal points in the support. %

\begin{definition}
Let $\Phi \subseteq [n_1] \times \cdots \times [n_k]$ be a subset.
We define the \emph{maximal points} of $\Phi$ with respect to the product order of the natural orders on $[n_1],\ldots, [n_k]$ as
\[
\max \Phi = \bigl\{ \alpha \in \Phi \bigm| \nexists \beta \in \Phi : \alpha < \beta \bigr\}.
\]
\end{definition}

\begin{definition}\label{lower}
Let $f$ be a nonzero multilinear map. %
Given $C\in \bases(f)$ we 
set the notation $M_C f = \max \supp_C f$. %
Define $\rho_\theta(f)$ by
\[
\rho_\theta(f) = \max_{C \in \bases(f)} H_\theta(M_C f)
\]
and define $\zeta_\theta(f)$ by
\[
\zeta_\theta(f) = 2^{\rho_\theta(f)}.
\]
We call $\zeta_\theta$ the \defin{Strassen lower support functional} and we call $\rho_\theta$ the \defin{logarithmic Strassen lower support functional}. When $f$ is 0 we naturally define $\zeta_\theta(f) = 0$ and $\rho_\theta(f) = -\infty$.
\end{definition}

\begin{theorem}[\cite{strassen1991degeneration}]\label{ls}  Let $f : V_1 \times \cdots \times V_k\to\FF$ and $g : W_1 \times \cdots \times W_k\to \FF$. Let $\theta \in \prob([k])$. The Strassen lower support functional $\zeta_\theta$ has the following properties.
\begin{enumerate}
\item $\zeta_\theta\bigl(\langle r \rangle\bigr) = r$ for $r \in \NN$.
\item $\zeta_{\theta}(f \oplus g) \geq \zeta_\theta(f) + \zeta_\theta(g)$.
\item $\zeta_{\theta}(f\otimes g) \geq \zeta_\theta(f) \zeta_\theta(g)$.
\item If $f \geq g$, then $\zeta_\theta(f) \geq \zeta_\theta(g)$.
\item $0 \leq \zeta_\theta(f) \leq (\dim V_1)^{\theta(1)} \cdots (\dim V_k)^{\theta(k)}$.
\end{enumerate}
\end{theorem}

\begin{remark}\label{notequal}
Regarding statement 2 in \cref{ls}, 
Bürgisser \cite{burg} shows that the lower support functional $\zeta_\theta$ is not in general additive under the direct sum when $\theta_i>0$ for all $i$. See also~\cite[Comment (iii)]{strassen1991degeneration}. In particular, this implies that the upper support functional $\zeta^\theta(f)$ and the lower support functional~$\zeta_\theta(f)$ are not equal in general, the upper support functional being additive. In fact, to show that the lower support functional is not additive, Bürgisser first shows that (with $\FF$ algebraically closed) the \defin{typical value} of $\zeta_\theta$ on $\FF^n \otimes \FF^n \otimes \FF^n$ equals $(1-\min_i \theta_i) \log_2 n + o(n)$; on the other hand, Tobler~\cite{tobler} shows that the typical value of $\zeta^\theta$ on $\FF^n \otimes \FF^n \otimes \FF^n$ equals $\log_2 n$. (So even generically $\zeta^\theta$ and $\zeta_\theta$ are different on $\FF^n \otimes \FF^n \otimes \FF^n$.) We discuss typical values in \cref{generic}.
\end{remark}

One verifies directly that statement 1 and 5 of \cref{ls} are true. For the proofs of statements~2--4 we refer to \cite[Section 3]{strassen1991degeneration}.

We finish this section by giving an alternative definition of $\rho_\theta$ in terms of complete flags, in the same spirit as \cref{upperfilt} for the upper support functional. Let us first make some general remarks about how (the maximal points in) the support with respect to a basis and the support with respect to a flag are related.

\begin{definition}\label{associatedflag}
Let $\Phi \subseteq [n_1] \times \cdots \times [n_k]$ be a subset. We define the \defin{downward closure} of $\Phi$ with respect to the product of the natural orders on $[n_1], \ldots, [n_k]$ as
\[
\downset\Phi =  \bigl\{ \alpha \in [n_1] \times \cdots \times [n_k] \bigm| \exists \beta \in \Phi : \alpha \leq \beta \bigr\}.
\]
To any $C = ((v_{1, \alpha_1}), \ldots, (v_{k, \alpha_k})) \in \bases(f)$ we can naturally associate a $k$-tuple of complete flags $F = ((V_{1, \alpha_1}), \ldots, (V_{k, \alpha_k})) \in \filt_\complete(f)$, by defining the subspace $V_{i, \alpha_i}$ as $\Span\{v_{i, \alpha_i}, v_{i, \alpha_i + 1}, \ldots \}$. %
\end{definition}

The following key properties of the downward closure and the maximal points follow directly from the definitions.

\begin{lemma}\label{downwardtrick}
Let $C \in \bases(f)$ and $F \in \filt_\complete(f)$ be associated. Then %
\begin{align*}
\downset \supp_C f &= \supp_F f,\\
\max \supp_C f &= \max \supp_F f.
\end{align*}
\end{lemma}

For $C \in \bases(f)$ we have set the notation $M_C f = \max \supp_C f$. For $F \in \filt(f)$ %
we similarly set the notation $M_F f = \max \supp_F f$. %
\begin{proposition}\label{lowerfilt}
$\rho_\theta(f) = \max_{F \in \filt_\complete(f)} H_\theta(M_F f)$ %
for $\theta\in \prob([k])$.  %
\end{proposition}
\begin{proof}
Let $C \in \bases(f)$ and $F \in \filt_\complete(f)$ be associated. Then $M_C f = M_F f$ by \cref{downwardtrick} and hence $H_\theta(M_C f) = H_\theta(M_F f)$. Since any $C$ is associated to some $F$ and vice versa we conclude that 
\[
\max_{C \in \bases(f)} H_\theta(M_C f) = \max_{F \in \filt_\complete(f)} H_\theta(M_F f).\qedhere
\]
\end{proof}

\subsection{Comparing the support functionals}

Strassen calls his functionals \emph{upper} and \emph{lower} because he proves that the upper support functional is at least the lower support functional  \cite[Corollary 4.3]{strassen1991degeneration}. 

\begin{theorem}%
\label{supportupperlower}
$\rho^\theta(f) \geq \rho_\theta(f)$.
\end{theorem}

This is a useful property when doing computations, since $\rho^\theta$ is defined as a minimisation and $\rho_\theta$ is defined as a maximisation. 

In fact, in \cite[Corollary 4.2]{strassen1991degeneration} Strassen first proves $\zeta^\theta(f \otimes g) \geq \zeta^\theta(f)\zeta_\theta(g)$ and then proves \cref{supportupperlower} as a corollary. The purpose of this section is to give a more direct proof of \cref{supportupperlower} for the benefit of the reader. The proof essentially comes down to the following proposition.

\begin{proposition}\label{short}
Let $f : V_1 \otimes \cdots \otimes V_k \to \FF$.
Let $F, G \in \filt_\complete(f)$. There are permutations $\phi_i : [n_i] \to [n_i]$ $(i \in [k])$ such that $(\phi_1 \times \cdots \times \phi_k) M_G f \subseteq \supp_F f$.
\end{proposition}

\begin{proof}
Let $V$ be some $n$-dimensional vector space and let $(V_\alpha)_\alpha$ and $(W_\beta)_\beta$ be complete flags of $V$ such that
\begin{align*}
V_\alpha &= \Span \{ v_\alpha, v_{\alpha+1}, \ldots \}\\
W_\beta &= \Span \{w_\beta, w_{\beta + 1}, \ldots \}. 
\end{align*}
Define the map
\begin{equation}\label{phidef}
\phi : [n] \to [n] : \beta \mapsto \max \bigl\{ \alpha \in [n] : V_\alpha \cap (w_\beta + W_{\beta + 1}) \neq \emptyset \bigr\}.
\end{equation}
The map $\phi$ is injective. Let $\beta, \gamma \in [n]$ with $\beta\leq\gamma$ and suppose $\alpha = \phi(\beta) = \phi(\gamma)$. Then \eqref{phidef} gives
\begin{align}
(\FF^\times v_\alpha + V_{\alpha + 1} )\cap (w_\beta + W_{\beta + 1}) &\neq \emptyset \label{precise1}\\
V_{\alpha + 1} \cap (w_\beta + W_{\beta + 1}) &= \emptyset \label{precise2}\\
(\FF^\times v_\alpha + V_{\alpha + 1}) \cap (w_\gamma + W_{\gamma + 1}) &\neq \emptyset \label{precise3}%
\end{align}
If $\beta<\gamma$, then using \eqref{precise1} and \eqref{precise3} we may obtain a contradiction to~\eqref{precise2}. We conclude that $\beta = \gamma$.

We turn to $F$ and $G$. Write
\begin{align*}
F &= ((V_{1, \alpha_1}), \ldots, (V_{k, \alpha_k}))\\
G &= ((W_{1, \beta_1}), \ldots, (W_{k, \beta_k}))
\end{align*}
and for each pair of complete flags $(V_{i, \alpha_i}), (W_{i, \beta_i})$ of $V_i$ we define the permutation $\phi_i : [n_i] \to [n_i]$ in the same way as above.
We prove $(\phi_1 \times \cdots \times \phi_k) M_G f \subseteq \supp_F f$. 
Let $(\beta_1, \ldots, \beta_k)$ be in $M_G f$. Let $(\alpha_1, \ldots, \alpha_k) = (\phi_1 \times \cdots \times \phi_k)(\beta_1, \ldots, \beta_k)$. Then by definition of the $\phi_i$ the intersection $V_{i, \alpha_i} \cap (w_{i, \beta_i} + W_{i, \beta_i + 1})$ is not empty. Choose %
\[
\overline{w_{i, \beta_i}} \in V_{i, \alpha_i} \cap (w_{i, \beta_i} + W_{i, \beta_i + 1}). %
\]
Since $f$ is multilinear we have for some $x_{\beta'} \in \FF$
\[
f(\overline{w_{1, \beta_1}}, \ldots, \overline{w_{k, \beta_k}}) = f(w_{1, \beta_1}, \ldots, w_{k, \beta_k}) + \sum_{\beta'> \beta} x_{\beta'}  f(w_{1, \beta'_1}, \ldots, w_{k, \beta'_k})
\]
with the sum over tuples $\beta' = (\beta'_1, \ldots, \beta'_k)$ that are strictly larger than $\beta$ in the product order. Since $\beta$ is a maximal element in $\supp_G f$ with respect to the product order, the sum over $\beta' > \beta$ equals zero.
We conclude that $f(\overline{w_{1, \beta_1}}, \ldots, \overline{w_{k, \beta_k}}) = f(w_{1, \beta_1}, \ldots, w_{k, \beta_k}) \neq 0$. Therefore, $f(V_{1, \alpha_1} \times \cdots \times V_{k, \alpha_k})$ is not zero and thus $(\alpha_1, \ldots, \alpha_k) \in \supp_F f$.
\end{proof}

The above proof of \cref{short} may more naturally be phrased in the language of Schubert cells, when $\FF = \CC$, as follows.
We use the notation and definitions of \cite{MR2143072} with the difference that our complete flags are decreasing instead of increasing so we need to use the group of lower unitriangular matrices~$U^-$ instead of the group of upper unitriangular matrices $U$.%

\begin{proof}
Let $U_n^- \subseteq \GL(\CC^{n})$ be the unipotent subgroup of lower triangular $n \times n$ matrices with ones on the diagonal. Let $Y$ be the variety of complete (decreasing) flags in~$\CC^n$.  The group $\GL(\CC^{n})$ naturally acts on $Y$. Let the symmetric group~$S_n$ act  on $Y$ via its natural action on the standard basis of $\CC^n$. 
It is well-known (see e.g.~\cite[Proposition~1.2.1]{MR2143072}) that the variety $Y$ is the disjoint union of the $U_n^{-}$-orbits $C^\phi \coloneqq U_n^- \phi\, x$ over all permutations~$\phi$ in the symmetric group $S_n$, where~$x$ is the standard flag $\Span \{e_1, \ldots, e_n \} \supset \Span \{ e_2, \ldots, e_n \} \supset \cdots \supset \Span \{ e_n \}$ in~$\CC^n$. These orbits are called the \emph{Schubert cells}. %

We turn to the $k$-tuples of flags $F$ and $G$. Write again
\begin{align*}
F &= ((V_{1, \alpha_1}), \ldots, (V_{k, \alpha_k}))\\
G &= ((W_{1, \beta_1}), \ldots, (W_{k, \beta_k})).
\end{align*}
Without loss of generality we may assume that each $(W_{i, \beta_i})$ is the standard flag in $\CC^{n_i}$. %
By the previous paragraph, for each~$i \in [k]$ there is a permutation $\phi_i \in S_{n_i}$ such that the flag $(V_{i, \phi_i(\beta_i)})$ is in the $U_{n_i}^-$-orbit of the flag $(W_{i, \beta_i})$, that is, there is a group element $g_i \in U_{n_i}^{-}$ such that $V_{i, \phi_i(\beta_i)} = g_i  W_{i, \beta_i}$ for every $\beta_i \in [n_i]$. Let  $(\beta_1, \ldots, \beta_k)$ be in $\supp_G f$, so
\[
f(W_{1,\beta_1} \times \cdots \times W_{k,\beta_k}) \neq \{0\}.
\]
Suppose $(\beta_1, \ldots, \beta_k)$ is maximal. Then since $g_i \in U_{n_i}^-$ and since $f$ is multilinear,
\[
f(W_{1,\beta_1} \times \cdots \times W_{k,\beta_k}) = f(g_1 W_{1,\beta_1} \times \cdots \times g_k W_{k,\beta_k}).
\]
By our choice of $g_i$,
\[
f(g_1 W_{1,\beta_1} \times \cdots \times g_k W_{k,\beta_k}) = f(V_{1, \phi_1(\beta_1)} \times \cdots \times V_{k, \phi_k(\beta_k)}).
\]
We conclude that $(\phi_1(\beta_1), \ldots, \phi_k(\beta_k))$ is in $\supp_F f$.
\end{proof}

\begin{proof}[\bfseries\upshape Proof of \cref{supportupperlower}]
By \cref{upperfilt} and \cref{lowerfilt} we have
\begin{align}
\rho^\theta(f) &= \min_{F\in \filt_\complete(f)} H_\theta(\supp_F f),\label{minx}\\
\rho_\theta(f) &= \max_{G \in \filt_\complete(f)} H_\theta(M_G f).\label{maxx}
\end{align}
Let $F \in \filt_\complete(f)$ minimise \eqref{minx} and let $G \in \filt_\complete(f)$ maximise \eqref{maxx}.
It suffices to show
\begin{equation}\label{coolthing}
H_\theta(M_G f) \leq H_\theta(\supp_F f).
\end{equation}
Let $P \in \prob(M_G f)$ such that $H_\theta(P) = H_\theta(M_G f)$.
\cref{short} gives permutations $\phi_i : [n_i] \to [n_i]$ such that
\[
(\phi_1 \times \cdots \times \phi_k) M_G f \subseteq \supp_F f.
\]
Let $Q \in \prob(\supp_F f)$ be defined by $Q(\alpha) = P((\phi_1^{-1} \times \cdots \times \phi_k^{-1})(\alpha))$.
Then $H_\theta(P) = H_\theta(Q)$. This proves \eqref{coolthing}.
\end{proof}

\subsection{Robust and oblique tensors; spectral points}

At this point we have the upper support functional $\zeta^\theta$, which is additive and sub-multiplicative, we have the lower support functional $\zeta_\theta$, which is super-additive and super-multiplicative, and we know that $\zeta^\theta(t) \geq \zeta_\theta(t)$ for any $t$. When the functionals $\zeta^\theta$ and~$\zeta_\theta$ coincide we have an additive and multiplicative $\leq$-monotone, a spectral point. In general (in fact generically for certain formats) the functionals $\zeta^\theta$ and $\zeta_\theta$ do \emph{not} coincide, however (\cref{notequal}). To get spectral points from the support functionals we should thus restrict to smaller families of tensors $X\subseteq \efftens$. The following definition just gives a name to those tensors for which the support functionals do coincide.

\begin{definition}[Robust] For $\theta \in \prob([k])$ we say $f$ is \defin{$\theta$-robust} if ${\zeta^\theta(f) = \zeta_\theta(f)}$. We say $f$ is \defin{robust} if $f$ is $\theta$-robust for all $\theta \in \prob([k])$.
\end{definition}

We try to understand what robust tensors look like.
Let $f$ be nonzero. Clearly~$f$ is $\theta$-robust if and only if 
\begin{equation}\label{thetarobust}
\zeta^\theta(f) \leq \zeta_\theta(f)
\end{equation}
Being $\theta$-robust is closed under $\oplus$ and $\otimes$, since $\zeta^\theta(f \oplus g) = \zeta^\theta(f) + \zeta^\theta(g) = \zeta_\theta(f) + \zeta_\theta(g) \leq \zeta_\theta(f \oplus g)$, and $\zeta^\theta(f \otimes g) \leq \zeta^\theta(f)\zeta^\theta(g) = \zeta_\theta(f)\zeta_\theta(g) \leq \zeta_\theta(f \otimes g)$.
Equation \eqref{thetarobust} means precisely that there exist  $C, D\in \bases(f)$ and $P \in \prob(M_D f)$ such that  
\begin{equation}\label{robustcrit}
H_\theta(\supp_{C} f) \leq H_\theta(P).
\end{equation}
In this case we have $\zeta_\theta(f) = \zeta^\theta(f) = 2^{H_\theta(P)}$.
In particular, $f$ is $\theta$-robust if there is a $C \in \bases(f)$ such that the maximisation $H_\theta(\supp_C f)$ is attained by a probability distribution $P \in \prob(M_C f)$. %
This criterion is automatically satisfied for all $\theta$ when in some basis the support of~$f$ equals the maximal points in the support of~$f$. %
Recall that a subset $\Phi$ of a partially ordered set is called an \defin{antichain} if all elements in $\Phi$ are pairwise incomparable.

\begin{definition}[Oblique]
We say $f$ is \defin{oblique} if there exists a $C \in \bases(f)$ such that $\supp_C f = M_C f$, that is, such that $\supp_C f$ is an antichain.
\end{definition}

Note that the family of oblique tensors is a semiring under $\oplus$ and $\otimes$.
Suppose~$f$ is oblique. Let $\supp_C f$ be an antichain. Then $f$ is robust and we have simply $\rho_\theta(f) = \rho^\theta(f) = H_\theta(\supp_C f)$. In the language of asymptotic spectra this means the following. 

\begin{theorem}[Strassen]\label{obliquetheorem}
Let~$\semiring{X}\subseteq\efftens$ be the family of all oblique $k$-tensors. Let $\theta \in \prob([k])$. Given $f\in \semiring{X}$, let $C_f \in \bases(f)$ be any bases in which $\supp_{C_f} f$ is an antichain. Then
\[
\zeta^\theta(f) = \zeta_\theta(f) = 2^{H_\theta(\supp_{C_f} f)}
\]
and thus $\zeta^\theta : \semiring{X} \to \RR_{\geq 0}$ is a restriction monotone, that is normalised on $\langle r\rangle$, additive under direct sum $\oplus$ and multiplicative under tensor product $\otimes$.
So~$\zeta^\theta$ is an element of the asymptotic spectrum~$\Delta(\semiring{X})$.
\end{theorem}

\begin{definition}
[Support simplex]\label{supportsimplex}
For any family of oblique tensors $\semiring{Y}$,
we call the image of $\prob([k]) \to \Delta(\semiring{Y}) : \theta \to \zeta_\theta$ the \defin{support simplex} in the asymptotic spectrum of $\semiring{Y}$. We denote the support simplex by $\zeta_{\prob([k])}(\semiring{Y})$.
\end{definition}

We stress again that the lower support functionals are not additive and can thus not be universal spectral points. 
The upper support functionals may be universal spectral points, but this can, however, not be shown with the help of the lower support functionals.

\section{Quantum functionals}\label{quantum}
In this section we let the base field $\FF$ be the complex numbers $\CC$.
The goal of this section is to construct an explicit family of universal spectral points in the asymptotic spectrum of tensors~$\Delta(\efftens)$. In other words, we construct maps $\{\textnormal{complex $k$-tensors}\} \to \RR_{\geq 0}$ that are $\geq$-monotone, multiplicative under $\otimes$, additive under $\oplus$ and normalised on the unit tensor $\langle n\rangle$. (In fact our maps will be monotone under degeneration $\degengeq$, cf.~\cref{degenremark}.)  

To achieve this we introduce two families of functionals. The \defin{upper quantum functional}~$F^{\theta}$ is defined in terms of isotypic projections; the \defin{lower quantum functional}~$F_\theta$ is defined in terms of quantum entropy. Recall that the Strassen support functionals $\zeta^\theta$ and~$\zeta_\theta$ are parametrised by the probability distributions~$\theta$ on~$[k]$, in other words by the probability distributions~$\theta$ on bipartitions of~$[k]$ of the form~$\{\{j\}, [k]\setminus \{j\}\}$.  
Our functionals~$F^{\theta}$ and $F_\theta$ are parametrised by the probability distributions $\theta$ on all bipartitions $\{S, [k]\setminus S\}$ of~$[k]$.  %
We show that the upper functional is always at least the lower functional, $F^{\theta}(t) \geq F_{\theta}(t)$. 
When $\theta$ is supported on bipartitions of the form $\{\{j\}, [k]\setminus\{j\}\}$ we show that~$F^{\theta}(t)$ and~$F_\theta(t)$ are equal and are thus universal spectral points, and moreover $\zeta^\theta(t) \geq F^{\theta}(t)$. 
Finally, we show that the regularised upper support functional equals the upper quantum functional, $\lim_{n\to \infty} \zeta^{\theta}(t^{\otimes n})^{1/n} = F^{\theta}(t)$, which implies that $F^\theta(t) \geq \zeta_\theta(t)$.

\parag{Notation} 
As always $[k]$ denotes the set $\{1, 2, \ldots, k\}$. For $S \subseteq [k]$ we define the complement $\overline{S} = [k]\setminus S$ and for $j \in [k]$ we define the complement $\overline{j} = [k] \setminus \{j\}$. As before, $V_1, \ldots, V_k$ are vector spaces over the base field, which is the complex numbers in this section. For a subset $S \subseteq [k]$ we define $V_S = \bigotimes_{\smash{j \in S}} V_j$.

A \defin{bipartition} of $[k]$ is an unordered pair $\{S, \overline{S}\}$ with $S\subseteq [k]$, $S \neq \emptyset$ and $\overline{S} \neq \emptyset$. We say that the bipartitions $\{S, \overline{S}\}$ and $\{T, \overline{T}\}$ are \defin{noncrossing} if $S\subseteq T$ or $T \subseteq S$ or $S\cap T = \emptyset$. The set of bipartitions of $[k]$ is denoted by $B$.
We will consider two subsets of the set $\prob(B)$ of probability distributions on $B$. Let~$\prob_\s(B)$ be the set of distributions supported on bipartitions of the form $\{\{j\}, \overline{j}\}$. We say a distribution $\theta \in \prob(B)$ is \defin{noncrossing} if $\theta$ is supported on pairwise noncrossing partitions. Let $\prob_\nc(B)$ be the set of noncrossing distributions. Note that the latter is not a convex set in general. Clearly $\prob_\s(B) \subseteq \prob_{\nc}(B) \subseteq \prob(B)$, with equality if and only if $k\leq 3$.

An \defin{integer partition} is a sequence $\lambda = (\lambda_1, \lambda_2, \ldots, \lambda_d)$ of nonnegative integers satisfying $\lambda_1\geq \lambda_2 \geq \cdots \geq \lambda_d$. We say that $\lambda$ is a partition of $n$ and write $\lambda \vdash n$ if $\lambda_1 + \cdots + \lambda_d = n$.
We denote the number of nonzero parts of $\lambda$ by $\ell(\lambda)$.
We write $\lambda \vdash_d n$ if $\lambda \vdash n$ and $\ell(\lambda) \leq d$.
Let~$\prob_n$ be the set of partitions of $n$.
For a partition $\lambda \vdash n$ let $\overline{\lambda}$ be the sequence $(\tfrac{\lambda_1}{n}, \ldots, \tfrac{\lambda_d}{n})$. For $e\geq d$ we may view $\overline{\lambda}$ as a probability distribution on $[e]$, and we can consider the Shannon entropy~$H(\overline{\lambda})$ of $\overline{\lambda}$. %

\subsection{Upper quantum functional}

To define the upper quantum functional we need concepts from representation theory.
Let~$V$ be a complex vector space of dimension $d$. Let the symmetric group~$S_n$ act on the tensor power $V^{\otimes n}$ by permuting the tensor legs and let the general linear group~$\GL(V)$ act on $V^{\otimes n}$ by acting on the $n$ tensor legs simultaneously. These actions commute, so $V^{\otimes n}$ is a $\GL(V) \times S_n$-module. It decomposes into irreducible modules as
\begin{equation}\label{sw}
V^{\otimes n} \cong \bigoplus_{\lambda \vdash n} \SSS_\lambda(V) \otimes [\lambda],
\end{equation}
where $[\lambda]$ is an irreducible $S_n$-module and $\SSS_\lambda(V)$ is an irreducible $\GL(V)$-module if $\ell(\lambda) \leq d$ and 0 otherwise. Decomposition \eqref{sw} is referred to as Schur--Weyl duality. Let $P_\lambda^V \in \End(V^{\otimes n})$ be the equivariant projection onto the subspace isomorphic to $\SSS_\lambda(V) \otimes [\lambda]$.
Recall that for $S\subseteq [k]$ we defined~$V_S$ as $\bigotimes_{j \in S} V_j$.
Generalising the above discussion, for any subset $S \subseteq [k]$, the group $\GL(V_S) \times S_n$ acts naturally on $(V_1 \otimes \cdots \otimes V_k)^{\otimes n}$, which decomposes as
\[
(V_1 \otimes \cdots \otimes V_k)^{\otimes n} \cong \bigoplus_{\lambda\vdash n} \SSS_\lambda(V_S) \otimes [\lambda] \otimes (V_{\overline{S}})^{\otimes n}
\]
with $\GL(V_S)$ acting trivially on $(V_{\overline{S}})^{\otimes n}$. We define $P_\lambda^{V_S} \in \End((V_{[k]})^{\otimes n})$ to be the equivariant projection  onto the subspace isomorphic to $\SSS_\lambda(V_S) \otimes [\lambda] \otimes (V_{\overline{S}})^{\otimes n}$.

We are interested in powers $t^{\otimes n}$ so we want to restrict $(V_{[k]})^{\otimes n}$ to the symmetric subspace $\SSS_{(n)}(V_{[k]})$. The following observation is well-known.

\begin{lemma}\label{SSbar}
Let $S \subseteq [k]$. Let $\lambda \vdash n$.
\begin{enumerate}
\item The projections $P_{(n)}^{V_{[k]}}$ and $P_\lambda^{V_S}$ commute.
\item The projections $P_\lambda^{V_S} P_{(n)}^{V_{[k]}}$ and $P_\lambda^{V_{\overline{S}}} P_{(n)}^{V_{[k]}}$ are equal.
\end{enumerate}
\end{lemma}
\begin{proof}
For the first statement, the projector $P_{\smash{(n)}}^{V_{[k]}}$ is just the symmetriser $\tfrac{1}{n!} \sum_{\pi \in S_n} \pi$, which commutes with the $S_n$-equivariant projector $P_\lambda^{\smash{V_S}}$.
We now prove the second statement. By Schur--Weyl duality,
\[
(V_S)^{\otimes n} \otimes (V_{\overline{S}})^{\otimes n} \cong \bigoplus_{\lambda, \mu} \SSS_\lambda(V_S) \otimes [\lambda] \otimes \SSS_\mu(V_{\overline{S}}) \otimes [\mu].
\]
Considering the symmetric subspace gives
\begin{align*}
\bigl((V_S \otimes V_{\overline{S}})^{\otimes n}\bigr)^{S_n} &\cong \bigoplus_{\lambda, \mu} \SSS_\lambda(V_S)  \otimes \SSS_\mu(V_{\overline{S}}) \otimes ([\lambda] \otimes [\mu])^{S_n}\\
&\cong \bigoplus_{\lambda} \SSS_\lambda(V_S)  \otimes \SSS_\lambda(V_{\overline{S}})
\end{align*}
since $([\lambda] \otimes [\mu])^{S_n} \cong ([\lambda]^* \otimes [\mu])^{S_n} \cong \Hom_{S_n}([\lambda], [\mu])$ is one-dimensional if $\lambda = \mu$ and the zero space otherwise. 
The image of $P_\lambda^{V_S} P_{(n)}^{V_{[k]}}$ and the image of $P_\lambda^{V_S} P_{(n)}^{V_{[k]}}$ are both equal to the subspace of $(V_S \otimes V_{\overline{S}})^{\otimes n}$ isomorphic to $\SSS_\lambda(V_S) \otimes \SSS_\lambda(V_{\overline{S}})$.
\end{proof}

\cref{SSbar} says that we can unambiguously label isotypical projections on the symmetric subspace by a pair $(b, \lambda) \in B \times \prob_n$. For $b = \{S, \overline{S}\} \in B$ and a partition $\lambda\vdash n$, define $P_\lambda^{V_b}\coloneqq P_\lambda^{V_{\smash{S}}} P_{\smash{(n)}}^{V_{[k]}}$.

\begin{lemma}\label{comm}
If $b_1,b_2 \in B$ are noncrossing bipartitions and $\lambda^{(1)}, \lambda^{(2)}$ are partitions, then $P_{\lambda^{(1)}}^{V_{\smash{b_1}}}$ and $P_{\lambda^{(2)}}^{V_{\smash{b_2}}}$ commute.
\end{lemma}
\begin{proof}
Without loss of generality we may say $b_1 = \{S_1, \overline{S_1}\}$ and $b_2 = \{S_2, \overline{S_2}\}$ with $S_1 \cap S_2 = \emptyset$. Clearly  the projectors $P^{V_{\smash{S_1}}}_{\lambda^{(1)}}$ and $P^{V_{\smash{S_2}}}_{\lambda^{(2)}}$ commute, since the intersection $S_1 \cap S_2$ is empty. By statement 1 in  \cref{SSbar} we obtain
\[
P_{\lambda^{(1)}}^{V_{b_1}}P_{\lambda^{(2)}}^{V_{b_2}} = P_{\lambda^{(1)}}^{V_{S_1}} P_{(n)}^{V_{[k]}} P_{\lambda^{(2)}}^{V_{S_2}} P_{(n)}^{V_{[k]}} = P_{\lambda^{(2)}}^{V_{b_2}} P_{\lambda^{(1)}}^{V_{b_1}}.\qedhere
\]
\end{proof}

We can now define the upper quantum functional.

\begin{definition}\label{upperdef}
Let $t \in V_1 \otimes \cdots \otimes V_k$. Let $\theta \in \prob(B)$. %
Define
\[
E^{\theta}(t) = \sup_{(\lambda^{(b)})} \sum_{b\in \supp \theta} \theta(b)\, H(\overline{\lambda^{(b)}})
\]
with the supremum over all tuples $(\lambda^{(b)})_{b\in B}$ of partitions $\lambda^{(b)} \vdash n$ such that
\[
\prod_{b\in \supp\theta} P_{\lambda^{(b)}}^{V_b} \, t^{\otimes n} \neq 0.
\]
We define $F^{\theta}(t) = 2^{E^{\theta}(t)}$ if $t \neq 0$ and $F^{\theta}(0) = 0$. We call $F^{\theta}$ the \defin{upper quantum functional} and $E^{\theta}$ the \defin{logarithmic upper quantum functional}.
\end{definition}

\begin{remark}
The definition of $E^\theta(t)$ in \cref{upperdef} depends on the order of the product $\prod_{b\in \supp \theta} P_{\lambda^{(b)}}^{V_b}$, since the projectors in general do not commute. For $\theta \in \prob_{\nc}$, however, the projectors do commute and the order does not matter. Since we will mostly focus on $\theta \in \prob_{\nc}$ we will leave the order of the elements of~$B$ implicit.
\end{remark}

\begin{theorem}\label{upperkeyprops}
Let $\theta \in \prob_{\nc}(B)$. Let $s \in V_1 \otimes \cdots \otimes V_k$ and $t \in W_1 \otimes \cdots \otimes W_k$.
\begin{enumerate}
\item $F^{\theta}(\langle r \rangle) = r$ for $r \in \NN$.
\item $F^{\theta}(s \oplus t) \leq F^{\theta}(s) + F^{\theta}(t)$.
\item $F^{\theta}(s \otimes t) \leq F^{\theta}(s) F^{\theta}(t)$.
\item If $s \degengeq t$, then $F^{\theta}(s) \geq F^{\theta}(t)$.
\end{enumerate}
\end{theorem}

Statement 4 is a special case of the following lemma.

\begin{lemma}\label{degen}
Let $\theta \in \prob(B)$. %
If~$s \degengeq t$, then $F^{\theta}(s) \geq F^{\theta}(t)$.
\end{lemma}
\begin{proof}
We may assume $s$ and $t$ are not zero. Let $(\lambda^{(b)})_b$ be a tuple of partitions of $n$. 
It is sufficient to show
\[
\prod_{\mathclap{b \in \supp \theta}} P_{\lambda^{(b)}}^{W_b}\, t^{\otimes n} \neq 0 \implies \prod_{\mathclap{b \in \supp \theta}} P_{\lambda^{(b)}}^{V_b}\, s^{\otimes n} \neq 0,
\]
since then $E^\theta(s)$ is a supremum over a larger set than $E^{\theta}(t)$.
Suppose that 
\[
\prod_{\mathclap{b \in \supp\theta}} P_{\lambda^{(b)}}^{V_{b}}\, s^{\otimes n} = 0.
\]
Let $A_i : V_i \to W_i$ $(i \in [k])$ be linear maps. For any $S \subseteq [k]$ and $\lambda \vdash n$
\[
P_\lambda^{W_S}(A_1 \otimes \cdots \otimes A_k)^{\otimes n} = (A_1 \otimes \cdots \otimes A_k)^{\otimes n} P_\lambda^{V_S}.
\]
Therefore 
\[
\prod_{\mathclap{b\in \supp \theta}} P_{\lambda^{(b)}}^{W_b} (A_1 \otimes \cdots \otimes A_k)^{\otimes n} s^{\otimes n} = (A_1 \otimes \cdots \otimes A_k)^{\otimes n} \prod_{\mathclap{b\in \supp \theta}} P_{\lambda^{(b)}}^{V_b}\, s^{\otimes n} = 0,
\]
and by continuity $\prod_{b\in \supp \theta} P_{\lambda^{(b)}}^{W_b}\, t^{\otimes n} = 0$.
\end{proof}

In the rest of this section we prove statement 2 and 3, for which we need bounds on the dimension of irreducible $\GL(V)$- and $S_n$-representations, and the semigroup property of Kronecker coefficients and Littlewood--Richardson coefficients.

\begin{remark}[Dimension bounds]
The number of partitions of $n$ into at most~$d$ parts is upper bounded by~$(n+1)^d$. For $\lambda \vdash n$ the dimension of the irreducible $S_n$-module~$[\lambda]$ is given by the hook-length formula
\[
\dim {[\lambda]} = \frac{n!}{\prod_{(i,j)\in Y(\lambda)} \hook(i,j)}
\]
where $Y(\lambda)$ is the Young diagram of shape $\lambda$, and $\hook(i,j)$ equals the number of boxes in $Y(\lambda)$ directly below $(i,j)$ plus the number of boxes directly to the right of $(i,j)$ plus one, and the product is over all coordinates of $Y(\lambda)$. For $\lambda \vdash_d n$ the dimension of the irreducible $\GL(V)$-module $\SSS_\lambda(V)$ is given by the formula
\[
\dim \SSS_\lambda(V) = \prod_{1\leq i< j \leq d} \frac{\lambda_i - \lambda_j + j-i}{j -i}.
\]
We will make use of the following estimates:
\begin{equation}\label{sym}
\frac{n!}{\prod_{\ell=1}^d (\lambda_\ell + d- \ell)!} \leq \dim {[\lambda]} \leq \frac{n!}{\prod_{\ell=1}^d \lambda_\ell!}
\end{equation}
\begin{equation}\label{gl}
\dim \SSS_\lambda(V) \leq (n+1)^{d(d-1)/2}.
\end{equation}
\end{remark}

\begin{definition}
Let $\mu, \nu \vdash n$ be partitions. Restrict the irreducible $S_n \times S_n$-representation $[\mu]\otimes [\nu]$ to an $S_n$-representation via $S_n \to S_n \times S_n : \pi \mapsto (\pi, \pi)$ and consider the isotypic decomposition,
\begin{equation}\label{krondef}
[\mu] \otimes [\nu] \!\downarrow^{S_n \times S_n}_{S_n}\, \cong \bigoplus_{\lambda \vdash n} \CC^{g_{\lambda, \mu, \nu}} \otimes [\lambda]
\end{equation}
where $\CC^{g_{\lambda, \mu, \nu}}$ is the multiplicity space for the irreducible representation $[\lambda]$ in its isotypic component. The number $g_{\lambda, \mu, \nu}$ is called a \defin{Kronecker coefficient}. (Equivalently, $g_{\lambda, \mu, \nu} = \dim ([\lambda] \otimes [\mu] \otimes [\nu])^{S_n}$ which justifies the symmetric notation.) Let $\lambda \vdash_{a+b}$ be a partition. Restrict the irreducible $\GL_{a+b}$-module $\SSS_\lambda(\CC^{a+b})$ to a $\GL_a \times \GL_b$-module via the block diagonal embedding $\GL_a \times \GL_b \to \GL_{a + b}$, and consider the isotypic decomposition,
\begin{equation}\label{lrdef}
\SSS_\lambda(\CC^{a+b}) \!\downarrow^{\GL_{a+b}}_{\GL_a \times \GL_b}\, \cong  \bigoplus_{\substack{\mu \vdash_a\\ \nu \vdash_b}} \CC^{c^\lambda_{\mu, \nu}} \otimes \SSS_\mu(\CC^{a}) \otimes \SSS_\nu(\CC^{b})
\end{equation}
where $\CC^{c^{\lambda}_{\mu, \nu}}$ is the multiplicity space for the irreducible module $\SSS_\mu(\CC^{a}) \otimes \SSS_\nu(\CC^{b})$ in its isotypic component. The number $c^{\lambda}_{\mu, \nu}$ is called a \defin{Littlewood--Richardson coefficient}.
\end{definition}

\begin{remark}[Semigroup property]
If $\lambda$ and $\lambda'$ are partitions, then $\lambda + \lambda'$ is defined by elementwise addition.
A fundamental property of the Kronecker coefficients and the Littlewood--Richardson coefficients is the well-known \emph{semigroup property} (see e.g.~\cite{MR2276458}): 
\[
\textnormal{if } g_{\lambda, \mu, \nu} > 0 \textnormal{ and } g_{\alpha, \beta, \gamma} > 0 \textnormal{, then } g_{\lambda + \alpha, \mu + \beta, \nu + \gamma} >0;
\]
\[
\textnormal{if } c^{\lambda}_{\mu, \nu} > 0 \textnormal{ and } c^{\alpha}_{\beta, \gamma} > 0 \textnormal{, then } c^{\lambda + \alpha}_{\mu + \beta, \nu + \gamma} >0.
\]
In other words, the triples of partitions for which the Kronecker coefficients are nonzero form a semigroup under elementwise addition, and the same is true for Littlewood--Richardson coefficients.
\end{remark}

The semigroup properties can be used to prove the following lemma. Of this lemma, the first statement can be found in \cite{MR2197548}, while we do not know of any source that explicitly states the second statement. For the convenience of the reader we give the proofs of both statements. Let $H(P)$ denote the Shannon entropy and let $h(p)$ denote the binary entropy function.

\begin{lemma}\label{dim}
 Let $\lambda, \mu, \nu$ be integer partitions.
\begin{enumerate}
\item If $g_{\lambda, \mu, \nu}$ is nonzero, then $H(\overline{\lambda}) \leq H(\overline{\mu}) + H(\overline{\nu})$.
\item If $c^\lambda_{\mu, \nu}$ is nonzero, then $H(\overline{\lambda}) \leq \tfrac{\abs[0]{\mu}}{\abs[0]{\nu}+\abs[0]{\mu}} H(\overline{\mu}) + \tfrac{\abs[0]{\nu}}{\abs[0]{\nu}+\abs[0]{\mu}} H(\overline{\nu}) + h\bigl(\tfrac{\abs[0]{\mu}}{\abs[0]{\nu}+\abs[0]{\mu}}\bigr)$.
\end{enumerate}
\end{lemma}
\begin{proof}
We begin with the first statement. Suppose $g_{\lambda, \mu, \nu}$ is nonzero. The semigroup property of Kronecker coefficients implies $g_{N\lambda, N\mu, N\nu} \neq 0$ for any positive integer $N$, where $N\lambda$ is the partition with parts $N\lambda_i$. Then the irreducible representation $[N \lambda]$ is isomorphic to a subspace of $[N\mu] \otimes [N\nu]$ by~\eqref{krondef}. So obviously $\dim {[N \lambda]} \leq \dim {[N \mu]} \dim {[N \nu]}$. 
We use the dimension bounds \eqref{sym} to see that $Nn H(\overline{\lambda}) - o(N) \leq Nn H(\overline{\mu}) + Nn H(\overline{\nu})$ when $N \to \infty$. We thus obtain the required inequality $H(\overline{\lambda}) \leq H(\overline{\mu}) + H(\overline{\nu})$.

We prove the second statement. By Schur--Weyl duality, as an $S_n \times \GL_{a+b}$-representation $(\CC^{a+b})^{\otimes n}$ decomposes as
\[
(\CC^{a+b})^{\otimes n} \cong\; \bigoplus_{\mathclap{\lambda \vdash_{a+b} n}}\; [\lambda] \otimes \SSS_\lambda(\CC^{a+b})
\]
and for the restriction via $\GL_a \times \GL_b \to \GL_{a+b}$ (denoted by $\downarrow$) we obtain
\[
(\CC^{a+b})^{\otimes n} \!\downarrow\, \cong \bigoplus_{\lambda \vdash_{a+b} n} \bigoplus_{\substack{\mu \vdash_a\\\nu \vdash_b}}\; [\lambda] \otimes \CC^{c^\lambda_{\mu, \nu}} \otimes \SSS_\mu(\CC^a) \otimes \SSS_\nu(\CC^b).
\]
On the other hand, we may first expand $(\CC^a \oplus \CC^b)^{\otimes n}$ and then apply Schur--Weyl duality for $\GL_a$ and $\GL_b$ separately,
\begin{align*}
(\CC^{a+b})^{\otimes n} \cong (\CC^a \oplus \CC^b)^{\otimes n} &\cong\; \bigoplus_{k=0}^n\; \CC^{\binom{n}{k}} \otimes  (\CC^a)^{\otimes k} \otimes (\CC^b)^{\otimes n-k}\\
&\cong\; \bigoplus_{k=0}^n \bigoplus_{\substack{\mu\vdash_a k\\ \nu\vdash_b n-k}}\! \CC^{\binom{n}{k}} \otimes  [\mu] \otimes \SSS_\mu(\CC^a) \otimes [\nu] \otimes \SSS_\nu(\CC^b).
\end{align*}
Suppose $c^{\lambda}_{\mu, \nu}$ is nonzero.
Then the irreducible representation $[\lambda]$ is isomorphic to a subspace of $\CC^{\smash{\binom{n}{\abs[0]{\mu}}}} \otimes [\mu] \otimes [\nu]$. So obviously $\dim {[\lambda]} \leq \binom{n}{\abs[0]{\mu}} \dim {[\mu]} \dim {[\nu]}$. Because of the semigroup property we may repeat everything for $N\lambda, N\mu, N\nu$ instead of $\lambda, \mu, \nu$ and obtain $\dim {[N\lambda]} \leq \binom{Nn}{\smash{N\abs[0]{\mu}}} \dim {[N\mu]} \dim {[N\nu]}$. We use the dimension bounds \eqref{sym}   to see that $N H(\overline{\lambda}) - o(N) \leq N h(\tfrac{\abs[0]{\mu}}{n}) + \tfrac{\abs[0]{\mu}}{n} N H(\overline{\mu}) + \tfrac{\abs[0]{\nu}}{n} N H(\overline{\nu})$ when~$N \to \infty$, where $h(p)$ denotes the binary entropy function.
\end{proof}

Let $V$ and $W$ be vector spaces of dimension $d$ and $e$.
The tensor power $(V \oplus W)^{\otimes n}$ is both a $\GL(V) \times \GL(W)$-module and a $\GL(V \oplus W)$-module. Let $\lambda\vdash_{d+e} n+m$, $\mu \vdash_d n$ and $\nu \vdash_e m$. The projections $P_\lambda^{V\oplus W}$ and $P_{\smash{\mu}}^{V} \otimes P_\nu^W$ commute, and their product is nonzero if and only if the Littlewood--Richardson coefficient $c^\lambda_{\mu, \nu}$ is nonzero.

\begin{lemma}
Let $\theta \in \prob_{\nc}(B)$. Then
$F^{\theta}(s \oplus t) \leq F^{\theta}(s) + F^{\theta}(t)$.
\end{lemma}
\begin{proof}
Let $(\lambda^{(b)})_{b\in \supp \theta}$ be a tuple of partitions of $n$ such that
\begin{equation}\label{aa}
\prod_{\mathclap{b \in \supp \theta}} P_{\lambda^{(b)}}^{(V \oplus W)_b} (s \oplus t)^{\otimes n} \neq 0.
\end{equation}
Since $P_{\lambda^{(b)}}^{(V \oplus W)_b} = P_{\lambda^{(b)}}^{(V \oplus W)_b}P_{(n)}^{(V \oplus W)}$ statement \eqref{aa} is equivalent to
\begin{equation}\label{sum}
\sum_{m=0}^n \binom{n}{m} \prod_{b \in \supp \theta} P_{\lambda^{(b)}}^{(V \oplus W)_b} (s^{\otimes m} \otimes t^{\otimes(n-m)}) \neq 0.
\end{equation}
At least one of the summands in \eqref{sum} is nonzero, i.e.~for some $m$
\[
\prod_{\mathclap{b \in \supp \theta}} P_{\lambda^{(b)}}^{(V \oplus W)_b} (s^{\otimes m} \otimes t^{\otimes(n-m)}) \neq 0.
\]
We may write $s^{\otimes m} \otimes t^{\otimes (n-m)}$ as
\[
s^{\otimes m} \otimes t^{\otimes (n-m)} =
\sum_{\substack{(\mu^{(b)})\\ (\nu^{(b)})}} \prod_{b \in \supp \theta} P_{\mu^{(b)}}^{V_b}\, s^{\otimes m} \otimes \prod_{b \in \supp \theta} P_{\nu^{(b)}}^{W_b}\, t^{\otimes (n-m)}
\]
where the sum is over tuples of partitions $\mu^{(b)} \vdash m$ and $\nu^{(b)} \vdash n-m$.
There thus exist tuples $(\mu^{(b)})_b$ and $(\nu^{(b)})_b$ such that
\begin{align}
&P_{\lambda^{(b)}}^{(V \oplus W)_b} \bigl( P_{\mu^{(b)}}^{V_b} \otimes P_{\nu^{(b)}}^{W_b}\bigr) \neq 0 \quad\textnormal{for all $b\in \supp \theta$} \label{cc1}\\
&\prod_{\mathclap{b \in \supp \theta}} P_{\mu^{(b)}}^{V_b}\, s^{\otimes m} \neq 0 \label{cc2}\\
&\prod_{\mathclap{b \in \supp \theta}} P_{\nu^{(b)}}^{W_b}\, t^{\otimes n-m} \neq 0.\label{cc3}
\end{align}
\cref{cc1} implies that the Littlewood--Richardson coefficient $c^{\lambda^{(b)}}_{\mu^{(b)}, \nu^{(b)}}$ is nonzero and thus by \cref{dim},
$H(\lambda^{(b)}) \leq \tfrac{m}{n} H(\mu^{(b)}) + (1-\tfrac{m}{n}) H(\nu^{(b)}) + h(\tfrac{m}{n})$.
Setting $p = m/n$ we can therefore write
\[
\sum_{\mathclap{b\in \supp \theta}} \theta(b) H(\lambda^{(b)}) \leq p \sum_{\mathclap{b\in \supp \theta}} \theta(b) H(\mu^{(b)}) + (1-p) \sum_{\mathclap{b\in \supp \theta}} \theta(b) H(\nu^{(b)}) + h(p)
\]
which because of \eqref{cc2} and \eqref{cc3} is at most $p E^\theta(s) + (1-p) E^\theta(t) + h(p)$. Taking the supremum over the admissible tuples $(\lambda^{(b)})_b$ gives the inquality $E^\theta(s \oplus t) \leq p E^\theta(s) + (1-p) E^\theta(t) + h(p)$. We are done by the following \cref{entropytrick}.
\end{proof}

\begin{lemma}[{See e.g.~\cite[Eq.~2.13]{strassen1991degeneration}}]\label{entropytrick}
Let $x, y \in \RR_{\geq0}$. Then 
\[
\max_{0 \leq p \leq 1} 2^{p x + (1-p)y + h(p)} = 2^x + 2^y.
\]
\end{lemma}

The tensor power $(V\otimes W)^{\otimes n}$ is both a $\GL(V) \times \GL(W)$-module and a $\GL(V \otimes W)$-module. Let $\lambda \vdash_{de} n$, $\mu \vdash_d n$ and $\nu\vdash_e n$. The projections $P^{V\otimes W}_\lambda$ and $P_\mu^V \otimes P_\nu^W$ commute, and their product is nonzero if and only if the Kronecker coefficient $g_{\lambda, \mu, \nu}$ is nonzero. 

\begin{lemma}
Let $\theta \in \prob_{\nc}(B)$. Then
$F^{\theta}(s \otimes t) \leq F^\theta(s)F^\theta(t)$.
\end{lemma}
\begin{proof}
Let $(\lambda^{(b)})_b$ be a tuple of partitions of $n$ such that
\[
\prod_{\mathclap{b \in \supp \theta}} P_{\lambda^{(b)}}^{(V \otimes W)_b} (s \otimes t)^{\otimes n} \neq 0.
\]
We may write
\begin{align*}
s^{\otimes n} \otimes t^{\otimes n}= \sum_{\substack{(\mu^{(b)})_b: \mu^{(b)} \vdash n\\ (\nu^{(b)})_b: \nu^{(b)} \vdash n}}\; \prod_{b\in \supp \theta}\!\! P_{\mu^{(b)}}^{V_b} s^{\otimes n} \otimes \prod_{b\in \supp \theta}\!\! P_{\nu^{(b)}}^{W_b} t^{\otimes n}.
\end{align*}
There exist tuples $(\mu^{(b)})_b$ and $(\nu^{(b)})_b$ such that 
\begin{align}
&P_{\lambda^{(b)}}^{(V \otimes W)_b} \bigl( P_{\mu^{(b)}}^{V_b} \otimes P_{\nu^{(b)}}^{W_b}\bigr)  \neq 0 \quad\textnormal{for all $b\in \supp \theta$} \label{ee1}\\
&\prod_{\mathclap{b \in \supp \theta}} P_{\mu^{(b)}}^{V_b} s^{\otimes n} \neq 0 \label{ee2}\\
&\prod_{\mathclap{b \in \supp \theta}} P_{\nu^{(b)}}^{W_b} t^{\otimes n} \neq 0.\label{ee3}
\end{align}
By \eqref{ee1} the Kronecker coefficient $g_{\lambda^{(b)}, \mu^{(b)}, \nu^{(b)}}$ is nonzero and so by \cref{dim} we have $H(\overline{\lambda^{(b)}}) \leq H(\overline{\mu^{(b)}}) + H(\overline{\nu^{(b)}})$. Therefore
\[
\sum_{\mathclap{b \in \supp \theta}} \theta(b) H(\overline{\lambda^{(b)}}) \leq \sum_{\mathclap{b \in \supp \theta}} \theta(b) H(\overline{\mu^{(b)}}) + \sum_{\mathclap{b \in \supp \theta}} \theta(b) H(\overline{\nu^{(b)}})
\]
which because of \eqref{ee2} and \eqref{ee3} is at most $E^\theta(s) + E^\theta(t)$. Taking the supremum over the admissible tuples $(\lambda^{(b)})_b$ gives $E^\theta(s \otimes t) \leq E^\theta(s) + E^\theta(t)$.
\end{proof}

\subsection{Lower quantum functional}

To define the lower quantum functional we need concepts from quantum information theory.
We work with finite-dimensional Hilbert spaces.
A \defin{state} or \defin{density operator} on a Hilbert space $\HH$ is a positive semidefinite linear map $\rho : \HH \to \HH$ satisfying ${\Tr \rho = 1}$. Let $\states(\HH)$ be the set of states on $\HH$. The \defin{von Neumann entropy} or \defin{quantum entropy} of $\rho$ is defined as $H(\rho) = -\Tr \rho \log \rho$.  If $\HH = \HH_1 \otimes \cdots \otimes \HH_k$, then the $j$th \defin{marginal} of a state  $\rho \in \states(\HH)$ is $\rho_j = (\Tr_{\HH_1} \otimes \cdots \otimes \Tr_{\HH_{j-1}} \otimes \Id\otimes \Tr_{\HH_{j+1}} \otimes \cdots \otimes \Tr_{\HH_k}) \rho$. More generally, if $S\subseteq [k]$, then $\rho_S = (\bigotimes_{j \in S} \id_{\HH_j} \otimes \bigotimes_{j \in \overline{S}} \Tr_{\HH_{j}}) \rho \in \states(\HH_S)$.
A state $\rho$ is \defin{pure} if $\rho$ has rank one as a linear map. If $\rho \in \states(\HH_1 \otimes \cdots \otimes \HH_k)$ is pure and $S\subseteq [k]$, then $\rho_S$ and $\rho_{\overline{S}}$ are unitarily equivalent up to enlarging the underlying Hilbert spaces $\HH_1, \ldots, \HH_k$, and therefore the entropy of $\rho_S$ equals the entropy of~$\rho_{\overline{S}}$.

\begin{remark}
A probability distribution $P \in \prob(X)$ gives rise to a state $\rho \in \states(\CC^X)$ as $\rho = \sum_{x \in X} P(x) \ketbra{x}{x}$. This gives a bijection between probability distributions on $X$ and those states on $\CC^X$ which are diagonal in the standard basis. In this way $\prob(X)$ can be identified  with a subset of $\states(\CC^X)$, and this identification is compatible with the notations of marginal distribution/marginal state and Shannon entropy/von Neumann entropy.
\end{remark}

\begin{definition}
Let $\HH = \HH_1 \otimes \cdots \otimes \HH_k$ be a Hilbert space. Given $\theta \in \prob(B)$ and a nonzero vector $\psi \in \HH$, define
\[
H_\theta(\psi) \coloneqq \sum_{S \in B} \theta(S)\, H\Bigl( \Bigl( \tfrac{1}{\langle\psi|\psi\rangle}{\ketbra{\psi}{\psi}} \Bigr)_{\!\!S} \Bigr).
\]
\end{definition}

\begin{definition}\label{quantumfunc}
Let $t \in V_1 \otimes \cdots \otimes V_k$. Choose inner products on each $V_j$ so that they become inner product spaces. For any $\theta \in \prob(B)$ we define
\[
E_\theta(t) = \sup_{\mathclap{A_1, \ldots, A_k}} H_\theta\bigl( (A_1 \otimes \cdots \otimes A_k) t \bigr)
\]
where the supremum is over invertible linear maps $A_j \in \GL(V_j)$. For $t = 0$ we set $E_\theta(t) = -\infty$. We also define $F_\theta(t) = 2^{E_\theta(t)}$ if $t \neq 0$ and $F_\theta(0) = 0$. We call~$E_\theta$ the \defin{logarithmic lower quantum functional} and we call $F_\theta$ the \defin{lower quantum functional}.
\end{definition}

\begin{remark}\label{tricks}
The supremum in the definition of $E_\theta$ is independent of the chosen inner products, since different inner products on a vector space are related by invertible linear maps. We could equivalently define $E_\theta(t)$ as a supremum of~$H_\theta(t)$ over the choice of local inner products. We may as well restrict to those $k$-tuples of inner products for which $\langle t | t \rangle$ equals one.

The definition of $E_\theta$ is not sensitive to embedding each $V_j$ in some larger vector space. In fact, we get another equivalent definition of $E_\theta$ by choosing large enough inner product spaces $\HH_1, \ldots, \HH_k$ and taking the supremum of $H_\theta((A_1 \otimes \cdots \otimes A_k) t)$ over all injective linear maps $A_j : V_j \to \HH_j$.
\end{remark}

\begin{example}\label{graphtensex}
$F_\theta(\langle r\rangle) = r$.
The proof is as follows.
Let $\psi = \langle r \rangle$. If we make the local bases orthonormal, then $\ketbra{\psi}{\psi}_S$ is a projector of rank $r$, and the entropy of the normalised spectrum equals~$\log_2(r)$, for any $\emptyset\neq S \subsetneq [k]$.  %
This gives the lower bound. On the other hand, the quantum entropy of a density matrix $\rho$ is at most the matrix rank of $\rho$. Therefore, $H(\ketbra{\psi}{\psi}_S) \leq \rank(\ketbra{\psi}{\psi}_S) \leq \rank( \flatten_S(\psi))$. This gives the upper bound.
\end{example}

The key properties of the lower quantum functional are as follows.

\begin{theorem}\label{lowerkeyprops}
Let $\theta \in \prob(B)$. Let $s \in V_1 \otimes \cdots \otimes V_k$, $t\in W_1 \otimes \cdots \otimes W_k$.
\begin{enumerate}
\item $F_\theta(\langle r \rangle) = r$ for any $r\in \NN$.
\item $F_\theta(s \oplus t) \geq F_\theta(s) + F_\theta(t)$.
\item $F_\theta(s \otimes t) \geq F_\theta(s) F_\theta(t)$.
\item If $s \degengeq t$, then $F_\theta(s) \geq F_\theta(t)$.
\item $0 \leq F_\theta(s) \leq \prod_{\{S,\overline{S}\}\in B}\min\{\dim(V_S),\dim(V_{\overline{S}})\}^{\theta(\{S,\overline{S}\})}$. %
\end{enumerate}
\end{theorem}

Statement 1 is \cref{graphtensex}. Statement 5 is easy to prove. In the rest of this section we prove statement 2--4.

\begin{lemma}
If $s \degengeq t$, then $F_\theta(s) \geq F_\theta(t)$.
\end{lemma}

\begin{proof}
We may without loss of generality assume that $s$ and $t$ are elements of the same space $V_1 \otimes \cdots \otimes V_k$ (\cref{tricks}). Let $G = \GL(V_1) \times \cdots \times \GL(V_k)$. By assumption, there is a sequence $(A_{1,i}, \ldots, A_{k, i}) \in G$ such that $\lim_{i \to \infty} (A_{1, i} \otimes \cdots \otimes A_{k, i}) s = t$. For any $k$-tuple $(B_1, \ldots, B_k) \in G$ we have therefore
\begin{align*}
E_\theta(s) &= \sup_{C_1, \ldots, C_k} H_\theta\bigr( (C_1 \otimes \cdots \otimes C_k) s \bigl)\\
&\geq \lim_{i \to \infty} H_\theta\bigl( (B_1 A_{1, i} \otimes \cdots \otimes B_k A_{k,i}) s \bigr)\\
&= H_\theta\bigl( (B_1 \otimes \cdots \otimes B_k) t\bigr)
\end{align*}
by continuity of $H_\theta$. Now take the supremum over all $(B_1, \ldots, B_k)\in G$ to obtain $E_\theta(s) \geq E_\theta(t)$ and thus $F_\theta(s) \geq F_\theta(t)$.
\end{proof}

\begin{lemma}[{Recursion property of quantum entropy, e.g.~\cite[Eq.~12.19]{MR2230995}}]\label{quantrec}
Let $\rho_i$ be density matrices with support in orthogonal subspaces $\HH_i$ of a Hilbert space $\HH = \oplus_{i=1}^M \HH_i$. Then the density matrix $\rho = \sum_i p_i \rho_i$ has quantum entropy $H(\rho) = H(p) + \sum_{i=1}^M p_i H(\rho_i)$.
\end{lemma}

\begin{lemma}
$F_\theta(s \oplus t) \geq F_\theta(s) + F_\theta(t)$.
\end{lemma}
\begin{proof}
Choose inner products $\langle\cdot, \cdot\rangle_j$ on $V_j$ and $\langle \cdot, \cdot \rangle'_j$ on $W_j$ in such a way that~$s$ and $t$ have norm one, and consider $H_\theta(s)$ and $H_\theta(t)$ with respect to these inner products (cf.~\cref{tricks}). Equip $V_j \oplus W_j$ with the direct sum of the inner products. Let $u = \sqrt{p}\, s \oplus \sqrt{1-p}\, t \in (V_1 \oplus W_1) \otimes \cdots \otimes (V_k \oplus W_k)$. For any $p \in [0,1]$ the vector $u$ has norm one and $u \cong s \oplus t$. For any subset $S \subseteq [k]$ the reduced states are related as $\ketbra{u}{u}_S = p \ketbra{s}{s}_S \oplus (1-p) \ketbra{t}{t}_S$. Therefore, by recursivity of quantum entropy (\cref{quantrec}),
\[
H\bigl( \ketbra{u}{u}_S \bigr) = p H( \ketbra{s}{s}_S ) + (1-p) H( \ketbra{t}{t}_S) + h(p).
\]
Taking the $\theta$-weighted average of both sides gives
\[
H_\theta(u) = p H_\theta(s) + (1-p) H_\theta(t) + h(p).
\]
Since $E_\theta(s \oplus t) \geq H_\theta(u)$ we may take the supremum over the choices of inner products to get
\[
E_\theta(s \oplus t) \geq p E_\theta(s) + (1-p) E_\theta(t) + h(p).
\]
Apply \cref{entropytrick} to see that $F_\theta(s \oplus t) \geq F_\theta(s) + F_\theta(t)$.
\end{proof}

\begin{lemma}
$F_\theta(s \otimes t) \geq F_\theta(s) F_\theta(t)$.
\end{lemma}
\begin{proof}
Choose inner products $\langle\cdot,\cdot \rangle_j$ on $V_j$ and $\langle \cdot, \cdot \rangle'_j$ on $W_j$ in such a way that~$s$ and $t$ have norm one, and consider $H_\theta(s)$ and $H_\theta(t)$ with respect to these inner products. Equip $V_j \otimes W_j$ with the tensor product of the inner products. Then $s \otimes t \in (V_1 \otimes W_1) \otimes \cdots \otimes (V_k \otimes W_k)$ has norm one. For any subset $S \subseteq [k]$ the reduced states are related as $\ketbra{s \otimes t}{s \otimes t}_S = \ketbra{s}{s}_S \otimes \ketbra{t}{t}_S$. Therefore $H(\ketbra{s \otimes t}{s \otimes t}_S) = H(\ketbra{s}{s}_S) + H(\ketbra{t}{t}_S)$. Taking the $\theta$-weighted average of both sides gives $H_\theta(s \otimes t) = H_\theta(s) + H_\theta(t)$.
Since $E_\theta(s \otimes t) \geq H_\theta(s \otimes t)$, taking the supremum over the choice of inner products gives $E_\theta(s \otimes t) \geq E_\theta(s) + E_\theta(t)$ and thus $F_\theta(s\otimes t) \geq F_\theta(s) F_\theta(t)$.
\end{proof}

\subsection{Comparing the quantum functionals}

In this section we show that in the general regime $\theta \in \prob(B)$ the upper quantum functional is at least the lower quantum functional, justifying the names.

\begin{theorem}\label{lowerupper}
Let $\theta \in \prob(B)$. %
Then
\[
E^{\theta}(t) \geq E_\theta(t).
\]
\end{theorem}

To prove \cref{lowerupper} we use the \emph{gentle measurement lemma} and the \emph{spectrum estimation theorem}. A proof of the gentle measurement lemma can be found in~\cite{MR1725132}, see also \cite[Lemma 5]{ogawa2002new} for the specific bound that we use.  Keyl and Werner proved the spectrum estimation theorem \cite{MR1878924}, see also \cite{MR1955142} and \cite[Theorem~1]{MR2197548} for a more succinct proof. 
Let $\states_{\leq}(\HH)$ be the set of positive semidefinite operators $\rho : \HH \to \HH$ satisfying $\Tr \rho \leq 1$.

\begin{lemma}[Gentle measurement]\label{gentle}
Let $\rho \in \states_{\leq}(\HH)$. Let $X : \HH \to \HH$ be a linear map such that $0 \leq X \leq I$. Then
\[
\norm[0]{X \rho X - \rho}_1 \leq 2 \sqrt{1 - \Tr(X \rho X)}.
\]
\end{lemma}

\begin{corollary}\label{gentlecor}
Let $\rho_0 \in \states(\HH)$ and $0\leq X_1, \ldots, X_n \leq I$. Suppose that for each~$i\in [n]$ the inequality $\Tr (X_i\rho_0 X_i) \geq 1-\eps $ holds. Then
\begin{equation}
\norm[0]{X_n \cdots X_2 X_1 \rho_0 X_1 X_2 \cdots X_n - \rho_0}_1 \leq 2n \sqrt{\eps}.\label{toprove}
\end{equation}
In particular,
\[
\Tr(X_n \cdots X_2 X_1 \rho_0 X_1 X_2 \cdots X_n) \geq 1 - 2n \sqrt{\eps}.
\]
\end{corollary}
\begin{proof}
Let $\rho_i = X_i \rho_{i-1} X_i$ for $i \in [n]$. By the Hölder inequality and the gentle measurement lemma (\cref{gentle}), we have
\begin{align*}
\norm{\rho_i - \rho_0}_1 &\leq \norm[0]{\rho_i - X_i \rho_0 X_i}_1 + \norm[0]{X_i \rho_0 X_i - \rho_0}_1\\
&= \norm[0]{X_i(\rho_{i-1} - \rho_0)X_i}_1 + \norm[0]{X_i \rho_0 X_i - \rho_0}_1\\
&\leq \norm[0]{\rho_{i-1} - \rho_0}_1 \cdot \norm[0]{X_i}^2_\infty + 2 \sqrt{1 - \Tr( X_i \rho_0 X_i)}\\
&\leq \norm[0]{\rho_{i-1} - \rho_0}_1 + 2 \sqrt{\eps}.
\end{align*}
The statement follows by induction on $i$.
\end{proof}

\begin{theorem}[Spectrum estimation]\label{specest}
Let~$\HH$ have dimension $d$.
Let $\rho \in \states(\HH)$ be a density matrix.
Let~$r \in \RR^d$ be the sequence of eigenvalues of $\rho$ ordered nonincreasingly.
Let $\lambda \vdash n$.
Then
\[
\Tr(P_\lambda^{\HH} \rho^{\otimes n}) \leq (n+1)^{d(d-1)/2} 2^{-n D(\lambda \| r)}
\]
where $D(p\| q) = \sum_i p_i( \log_2 p_i - \log_2 q_i)$ is the Kullback--Leibler divergence of probability distributions $p$ and $q$.
\end{theorem}

\begin{corollary}\label{specestcor}
Let $\rho \in \states(\HH)$. Let $\eps > 0$. Let $X = \sum P_\lambda^\HH$ where the sum is over partitions $\lambda\vdash n$ such that $H(\overline{\lambda}) \geq H(\rho) - \eps$.
Then $\lim_{n \to \infty} \Tr\bigl(\rho^{\otimes n} X \bigr) = 1$.
\end{corollary}
\begin{proof}
We refer to Corollary 2 in \cite{MR2197548}.
\end{proof}

\begin{proof}[\bfseries\upshape Proof of \cref{lowerupper}]
Let $\ell = \abs[0]{\supp \theta}$.
Let $\eps > 0$ be arbitrary. Choose inner products on $V_1, \ldots, V_k$ such that $t$ is a unit vector and such that the $\theta$-weighted average of the marginal quantum entropies of $t$ is at least $E_\theta(t) - \eps$. Let $(r_b)_{b \in B}$ be the tuple of marginal quantum entropies of $t$.%
By \cref{specestcor} there is an $n_0$ such that for every $n \geq n_0$ and $i \in [\ell]$ the inequality $\Tr( \ketbra{t}{t}^{\otimes n} X_i) \geq 1 - \eps$ holds, where
\[
X_i = \sum_{\mathclap{\substack{\lambda \vdash n:\\H(\overline{\lambda}) \geq r_{\pi(i)} - \eps }}} P_\lambda^\HH.
\] 
According to \cref{gentlecor} this in turn implies
\[
\Tr\bigl(X_\ell \cdots X_2 X_1 (\ketbra{t}{t})^{\otimes n} X_1 X_2 \cdots X_\ell\bigr) \geq 1- 2 \ell \sqrt{\eps}.
\]
In particular, $X_\ell \cdots X_2 X_1 t^{\otimes n} \neq 0$ for $\eps$ small enough and $n$ large enough. By the definition of the operators $X_i$ this implies that there is a tuple $(\lambda^{(i)})_{i \in [\abs[0]{\supp \theta}]}$ of partitions $\lambda^{(i)} \vdash n$ such that
\[
P_{\lambda^{(\pi(\ell))}}^{V_{\pi(\ell)}} \cdots P_{\lambda^{(\pi(2))}}^{V_{\pi(2)}}P_{\lambda^{(\pi(1))}}^{V_{\pi(1)}} t^{\otimes n} \neq 0
\]
and such that $H(\overline{\lambda^{(i)}}) \geq r_{\pi(i)} - \eps$ for each $i \in [\ell]$. Therefore,
\[
E^{\theta}(t) \geq \sum_{\mathclap{b\in \supp \theta}} \theta(b) r_b - \eps \geq E_\theta(t) - 2 \eps.
\]
This holds for any $\eps > 0$, so $E^{\theta}(t) \geq E_\theta(t)$.
\end{proof}

\subsection{Universal spectral points}

We have the upper quantum functional $F^\theta$ which is sub-additive and sub-multiplicative, we have the lower quantum functional $F_\theta$ which is super-additive and super-multiplicative, both are $\degenleq$-monotone and normalised on $\langle r \rangle$, and we know that $F^\theta(t) \geq F_\theta(t)$. (See \cref{upperkeyprops}, \cref{lowerkeyprops} and \cref{lowerupper}.)
In this section we show that $F^\theta(t)$ and $F_\theta(t)$ in fact coincide for the singleton regime $\theta \in \prob_\s(B)$ and for any tensor $t$, which shows that $F^\theta=F_\theta$ is a universal spectral point.
The ingredient for the proof is an object called the entanglement polytope and two characterisations for this object.

Let $\HH = \HH_1 \otimes \cdots \otimes \HH_k$ be a Hilbert space. Let $G = \GL(\HH_1) \times \cdots \times \GL(\HH_k)$. Let $\ket{\psi} \in \HH$ be a unit vector. Then $\ketbra{\psi}{\psi} \in \states(\HH)$ and therefore we can consider the marginals $\ketbra{\psi}{\psi}_j$ of $\ketbra{\psi}{\psi}$ for $j \in [k]$. For each~$j \in [k]$, let~$r_j(\psi)$ be the nonincreasingly ordered tuple of eigenvalues of the $j$th marginal.  The \defin{entanglement polytope} of a nonzero tensor $t \in \HH_1 \otimes \cdots \otimes \HH_k$ is the subset $\Delta_t \subseteq \RR^{\dim \HH_1 + \cdots + \dim \HH_k}$ defined as
\[
\Delta_t = \bigl\{ \bigl(r_1(\psi), \ldots, r_k(\psi)\bigr) \,\big|\, \psi \in \overline{G \cdot t} \textnormal{ and } \norm{\psi} = 1 \bigr\}.
\]
Note that $\Delta_t$ does not depend on the choice of inner products on $\HH_j$. To relate~$F^\theta$ to $F_\theta$ we crucially use the following equivalent characterisation of $\Delta_t$, which is due to \cite[Theorem 1]{MR3087706} and \cite{MR3195184} and based on \cite{MR932055}.

\begin{theorem}
\label{characterization}
The entanglement polytope $\Delta_t$ is the Euclidean closure of the set
\[
\bigl\{ (\overline{\lambda^{(1)}}, \ldots, \overline{\lambda^{(k)}}) \,\big|\, (P_{\lambda^{(1)}}^{\HH_1} \otimes \cdots \otimes P_{\lambda^{(k)}}^{\HH_k}) t^{\otimes n} \neq 0 \bigr\}.
\]
\end{theorem}

We now return to the quantum functionals.

\begin{theorem}\label{upperequalslower}
Let $\theta \in \prob_\s(B)$. Then $E^\theta(t) = E_\theta(t)$.
\end{theorem}
\begin{proof}
Since $\theta \in \prob_\s(B)$, the value of $E_\theta(t)$ is related to the entanglement polytope~$\Delta_t$ as follows
\begin{equation}\label{polyt}
E_\theta(t) = \max \bigl\{ \theta(1) H(r_1) + \cdots + \theta(k) H(r_k) \,\big|\, (r_1, \ldots, r_k) \in \Delta_t \bigr\}.
\end{equation}
Using the characterisation in \cref{characterization} we can write \eqref{polyt} as 
\[
\sup \bigl\{ \theta(1) H(\overline{\lambda^{(1)}}) + \cdots + \theta(k) H(\overline{\lambda^{(k)}}) \,\big|\, (P_{\lambda^{(1)}}^{V_1} \otimes \cdots \otimes P_{\lambda^{(k)}}^{V_k}) t^{\otimes n} \neq 0 \bigr\}
\]
which equals $E^\theta(t)$.
\end{proof}

By the discussion at the beginning of this section the following is immediate.

\begin{corollary}\label{univ}
Let $\theta \in \prob_s(B)$. Then $F^\theta = F_\theta$ is a point in the asymptotic spectrum $\Delta(\efftens)$, a universal spectral point.
\end{corollary}

\begin{definition}
[Quantum simplex]
Analogous to the support simplex (\cref{supportsimplex}), we define for any family $\semiring{X}$ the \defin{quantum simplex} $F_{\prob([k])}(\semiring{X})$ as the image  of the map $\prob([k]) \to \Delta(\semiring{X}) : \theta \mapsto F^\theta$.
\end{definition}

We relate the Strassen upper support functional $\rho^\theta$ to the new upper quantum functional $E^\theta$. %
Recall that $\rho^\theta$ is defined for $\theta \in \prob_s(B) \subseteq \prob(B)$, that is, for probability distributions $\theta$ supported on bipartitions of the form $\{\{j\}, \overline{j}\}$.
Let $t \in V_1 \otimes \cdots \otimes V_k$.

\begin{remark}[Majorisation property]
Let $v_1, \ldots, v_d$ be a basis of $V$. Let $T \subseteq \GL(V)$ be the subgroup of linear maps having diagonal matrix in this basis, a maximal torus for $\GL(V)$. For a tuple $i  = (i_1, \ldots, i_n) \in [d]^n$ let the vector $v_i \in V^{\otimes n}$ be defined as
\[
v_i \coloneqq v_{i_1} \otimes \cdots \otimes v_{i_n}.
\]
Let $i \in [d]^n$ and let $Q$ be the type of $i$. Then the vector $v_i$ is a \defin{weight vector} for $T$ with weight $nQ$, meaning that for any $t  = \diag(t_1, \ldots, t_d) \in T$ we have $t \cdot v_i = \prod_{j=1}^d (t_j)^{nQ(j)} \,v_i$. 
The module $\SSS_\lambda(V)$ has so-called highest weight $\lambda$, and thus $P_\lambda^{\smash{V}} v_i \neq 0$ implies that $nQ\domleq\lambda$, where $\domleq$ is the dominance order or majorisation order on partitions. We call this the \defin{majorisation property}. In particular, $P_\lambda^V v_i \neq 0$ implies the inequality of Shannon entropies $H(Q) \geq H(\overline{\lambda})$, since the Shannon entropy is Schur-concave.
\end{remark}

\begin{theorem}\label{suppinv}
Let $\theta \in \prob_\s(B)$. Then $\rho^{\theta}(t) \geq E^{\theta}(t)$.
\end{theorem}
\begin{proof}
Choose a $k$-tuple of bases $C = ((v_{1,\alpha_1}), \ldots, (v_{k, \alpha_k}))$ in $\bases(t)$ such that $\rho^{\theta}(t) = H_\theta(\supp_C t)$. Suppose $(P_{\lambda^{(1)}}^{V_1} \otimes \cdots \otimes P_{\lambda^{(k)}}^{V_k})t^{\otimes N} \neq 0$ for some $k$-tuple of partitions $(\lambda^{(1)}, \ldots, \lambda^{(k)})$ and $N\in \NN$. 
Given an $N$-type $Q \in \prob_N([n_1] \times \cdots \times [n_k])$, define the vector $v_Q \in V_1^{\otimes N} \otimes \cdots \otimes V_k^{\otimes N}$ as
\[
v_Q \coloneqq \!\sum_{x} \bigotimes_{m=1}^N v_{1, x_{m,1}} \otimes v_{2, x_{m,2}} \otimes \cdots \otimes v_{k, x_{m,k}}
\]
where the sum is over all $N$-tuples $x = (x_1, \ldots, x_N) \in T_{\smash{Q}}^N$ of type $Q$,
and define the element $t_Q \in \CC$ as $t_Q \coloneqq \prod_{i \in \supp Q}\; (t_{i_1, \ldots, i_k})^{NQ(i)}$. 
Then the tensor power~$t^{\otimes N}$ can be written as
\begin{equation}\label{decomp}
t^{\otimes N} = \sum_{Q} t_Q \, v_Q
\end{equation}
where the sum is over $Q \in \prob_N(\supp_C t)$.
Then $(P_{\lambda^{(1)}}^{V_1} \otimes \cdots \otimes P_{\lambda^{(k)}}^{V_k}) v_Q \neq 0$ for some $Q \in \prob_N(\supp_C t)$. By the majorization property this implies for all $j \in [k]$ that $N Q_j \domleq \lambda^{(j)}$  and thus $H(Q_j) \geq H(\overline{\lambda^{(j)}})$. This shows that
\[
\rho^{\theta}(t) = H_\theta(\supp_C t) \geq H_\theta(Q) = \sum_{j=1}^k \theta(j) H(Q_j) \geq \sum_{j=1}^{\smash{k}} \theta(j) H(\overline{\lambda^{(j)}}).
\]
Take the supremum of the right hand side over the possible $(\lambda^{(1)}, \ldots, \lambda^{(k)})$ to get the required inequality $\rho^{\theta}(t) \geq E^{\theta}(t)$.
\end{proof}

By \cref{suppinv}, for $\theta \in \prob([k])$, $E^\theta(t) \leq \rho^\theta(t) \leq n_1^{\smash{\theta(1)}} \cdots n_k^{\smash{\theta(k)}}$. Therefore, if $E^\theta(t)=n_1^{\smash{\theta(1)}}\cdots n_k^{\smash{\theta(k)}}$, then $\rho^\theta(t) = n_1^{\smash{\theta(1)}}\cdots n_k^{\smash{\theta(k)}}$.
We will show the converse of this in \cref{hm}, that is, when $\theta_i >0$ for all $i$, if $E^\theta(t) < n_1^{\smash{\theta(1)}}\cdots n_k^{\smash{\theta(k)}}$, then $\rho^\theta(t) < n_1^{\smash{\theta(1)}}\cdots n_k^{\smash{\theta(k)}}$.

The final result of this section is that the regularisation of the upper support functional equals the upper quantum functional.

\begin{theorem}\label{regsupp}
Let $t\in V_1\otimes\cdots\otimes V_k$, $\theta\in\mathcal{P}([k])$. Then
\begin{equation*}
\lim_{n\to\infty}\frac{1}{n}\rho^\theta(t^{\otimes n})=E^\theta(t)
\end{equation*}
\end{theorem}
\begin{proof}
By \cref{suppinv}, $E^{\theta}(t) \leq \rho^\theta(t)$, so $\lim_{n \to \infty} \frac{1}{n}E^\theta(t^{\otimes n}) \leq \lim_{n\to\infty}\frac{1}{n}\rho^\theta(t^{\otimes n})$.
We know $E^\theta(t)=\lim_{n \to \infty} \frac{1}{n}E^\theta(t^{\otimes n})$ by \cref{univ}. Therefore
\begin{equation*}
\lim_{n\to\infty}\frac{1}{n}\rho^\theta(t^{\otimes n}) \ge E^\theta(t).
\end{equation*}
We now prove $\leq$. 
Let $n\in\mathbb{N}$ and consider the decomposition
\begin{equation*}
V_i^{\otimes n}=\bigoplus_{\lambda\vdash n}\mathbb{S}_\lambda(V_i)\otimes[\lambda].
\end{equation*}
In each direct summand on the right hand side choose a basis with (disjoint) index sets $I_{i,\lambda}$, $\lambda\vdash n$. Together they form a basis of $V_i^{\otimes n}$. Let $C\in \bases(t^{\otimes n})$ be this $k$-tuple of bases. Define the set
\begin{equation*}
\Phi=\bigcup \bigl\{ I_{1,\lambda^{(1)}}\times\cdots\times I_{k,\lambda^{(k)}} \,\big|\, \lambda^{(1)}\!,\ldots,\lambda^{(k)}:    \sum_{i=1}^k\theta(i)H(\overline{\lambda^{(i)}})\le E^\theta(t)  \bigr\}.
\end{equation*}
By definition, $\supp_C t^{\otimes n}\subseteq\Phi$, and therefore
\begin{equation*}
\rho^\theta(t^{\otimes n})\le H_\theta(\supp_C t^{\otimes n})\le H_\theta(\Phi).
\end{equation*}
To estimate $H_\theta(\Phi)$, note that by concavity of the entropy and invariance of $\Phi$ under permutations within each $I_{i,\lambda}$, the maximum is attained at a distribution which is uniform within each block $I_{1,\lambda^{(1)}}\times\cdots\times I_{k,\lambda^{(k)}}$. Let $P$ be such a distribution with $H_\theta(\Phi)=H_\theta(P)$, and let the total probability within block $I_{1,\lambda^{(1)}}\times\cdots\times I_{k,\lambda^{(k)}}$ be $Q(\lambda^{(1)},\ldots,\lambda^{(k)})$. Thus $Q$ describes a ``coarse'' structure of $P$, which is uniform conditioned on $Q$.
The marginals of $P$ inherit a similar structure: $P_i$ is uniform within each block $I_{i,\lambda^{(i)}}$ and the total probability of such a block is $Q_i(\lambda^{(i)})$. The recursion property allows us to express $H(P_i)$ as
\begin{equation*}
\begin{split}
H(P_i)
 & = H(Q_i)+\sum_{\lambda^{(i)}}Q_i(\lambda^{(i)})\log_2|I_{i,\lambda^{(i)}}|  \\
 & = H(Q_i)+\sum_{\mathclap{\lambda^{(1)}\!,\ldots,\lambda^{(k)}}}Q(\lambda^{(1)},\ldots,\lambda^{(k)})\log_2|I_{i,\lambda^{(i)}}|,
\end{split}
\end{equation*}
and therefore
\begin{equation}\label{sumoverlambda}
H_\theta(P)=H_\theta(Q)+\sum_{\mathclap{\lambda^{(1)}\!,\ldots,\lambda^{(k)}}}Q(\lambda^{(1)},\ldots,\lambda^{(k)})\sum_{i=1}^{\smash{k}}\theta(i)\log_2|I_{i,\lambda^{(i)}}|.
\end{equation}
We use the estimates
\begin{equation*}
|I_{i,\lambda}|=\dim (\mathbb{S}_\lambda(V_i)\otimes[\lambda] ) \le (n+1)^{d_i(d_i-1)/2}\binom{n}{\lambda_1,\ldots,\lambda_{d_i}}\le 2^{nH(\overline{\lambda})+o(n)}
\end{equation*}
and
\begin{equation*}
H_\theta(Q)\le H(Q)\le d_1d_2\cdots d_k\log_2(n+1)=o(n).
\end{equation*}
In the sum over $\lambda^{(1)},\ldots,\lambda^{(k)}$ in \eqref{sumoverlambda} only those terms can have nonzero probability which satisfy
\begin{equation*}
\sum_{i=1}^k\theta(i)H(\overline{\lambda^{(i)}})\le E^\theta(t).
\end{equation*}
For these terms we have the upper bound
\begin{equation*}
\sum_{i=1}^k\theta(i)\log_2|I_{i,\lambda^{(i)}}|\le n\sum_{i=1}^k\theta(i)H(\overline{\lambda^{(i)}})+o(n)\le nE^\theta(t)+o(n).
\end{equation*}
The sum of all $Q(\lambda^{(1)}, \ldots, \lambda^{(k)})$ equals $1$, and therefore
\begin{equation*}
\rho^\theta(t^{\otimes n})\le H_\theta(P)\le nE^\theta(t)+o(n),
\end{equation*}
which implies
$\lim_{n\to\infty}\frac{1}{n}\rho^\theta(t^{\otimes n})\le E^\theta(t)$.
\end{proof}

\begin{remark}
In the language of comment 7.(ii) in \cite{strassen1991degeneration}, our \cref{regsupp} shows that the set $\mathsf{Z}^\theta$ is a single point.
\end{remark}

Using \cref{regsupp} we can see that the quantum functionals are at least the lower support functional.

\begin{corollary}
Let $t \in V_1 \otimes \cdots \otimes V_k$, $\theta \in \prob_s(B)$. Then $E^\theta(t) \geq \rho_\theta(t)$.
\end{corollary}
\begin{proof}
By \cref{regsupp},
\[
F^\theta(t) = \lim_{n\to \infty} \zeta^\theta(t^{\otimes n})^{1/n}.
\]
We know $\zeta^\theta(t) \geq \zeta_\theta(t)$ by \cref{supportupperlower} and thus
\[
\lim_{n\to \infty} \zeta^\theta(t^{\otimes n})^{1/n} \geq \lim_{n \to \infty} \zeta_\theta(t^{\otimes n})^{1/n}.
\]
The lower support functional $\zeta_\theta$ is super-multiplicative under $\otimes$ (\cref{ls}), so 
\[
\lim_{n \to \infty} \zeta_\theta(t^{\otimes n})^{1/n} \geq \zeta_\theta(t)
\]
Combining these statements proves the theorem.
\end{proof}

\section{Several families of tensors}\label{tightfreegeneric}

With the definitions and key properties of the Strassen support functionals and the quantum functionals in place, we will now focus on various classes of tensors. We begin with the class of tight tensors that was introduced by Strassen. Tight tensors are oblique and Strassen showed that for tight tensors the support functionals are sufficient to compute the asymptotic subrank.

Next we discuss reduced polynomial multiplication and the cap set problem. Tightness is useful in this context.

Then we introduce free tensors. Oblique tensors are free. For complex free tensors we show that the upper support functional and the quantum functionals coincide in the singleton regime. %

Finally we compute generic values of the quantum functionals and compare this to the known results on the generic value of the Strassen support functionals. Related to this we apply the Hilbert--Mumford criterion to see that if the upper quantum functional is not maximal, then the upper support functional is not maximal.

\subsection{Tight tensors} Let $\FF$ be an arbitrary field. Strassen defined the family of tight tensors as follows.

\begin{definition}[Tight]
Let $\Phi \subseteq I_1 \times \cdots \times I_k$. We say the set $\Phi$ is \defin{tight} if there are injective maps $u_i : I_i \to \ZZ$ $(i\in [k])$ such that
\begin{equation}\label{tightnesscondition}
u_1(\alpha_1) + \cdots + u_k(\alpha_k) = 0\quad \textnormal{for each}\quad (\alpha_1, \ldots, \alpha_k) \in \Phi.
\end{equation}
We say a tensor $f$ is \defin{tight} if there is a $C\in \bases(f)$ such that the set $\supp_C f$ is tight.
\end{definition}

\begin{remark}Clearly, tight tensors are oblique. To summarise the families of tensors that we have defined up to now, we have
\[
\{\textnormal{tight}\} \subseteq \{ \textnormal{oblique} \} \subseteq \{ \textnormal{robust} \} \subseteq \{\textnormal{$\theta$-robust}\}.
\]
Recall that the families of oblique, robust and $\theta$-robust tensors each form a semiring under $\otimes$ and $\oplus$. Tight tensors have the same property \cite[Section~5]{strassen1991degeneration}. Another property is that any subset of a tight set is tight.
\end{remark}

\begin{example}\label{tightexample}
Let $k\geq 3$ be fixed.
For any integer $n\geq 1$ and $c\in [n]$ the following set is clearly tight:
\begin{align*}
\Phi_n(c) &= \{ \alpha \in \{0,\ldots, n-1\}^{\smash{k}} \mid \alpha_1 + \cdots + \alpha_k = c \}.\\
\intertext{For any integer $n\geq 2$ and any $c\in [n]$ the following set is \emph{not} tight (cf.~Exercise~15.20 in~\cite{burgisser1997algebraic}):}
\Psi_n(c) &= \{ \alpha \in \{0,\ldots, n-1\}^k \mid \alpha_1 + \cdots + \alpha_k = c \bmod n \}.
\end{align*}
The \emph{tensor} $f_n = \sum_{\alpha \in \Psi_n(n-1)} e_{\alpha_1} \otimes \cdots \otimes e_{\alpha_k} \in (\FF^n)^{\otimes k}$ is tight, however, when the ground field $\FF$ contains a primitive $n$th root of unity $\zeta$. Namely, the elements $v_j = \sum_{i=1}^n \zeta^{ij} e_i$ $(j\in [n])$ form a basis of $\FF^n$.  Let  $C \in \bases(f_n)$ be the corresponding product basis. We have $f_n = \sum_{j=1}^n v_j \otimes \cdots \otimes v_j$, and thus $\supp_C f_n = \{ \alpha \in [n]^{\smash{k}} \mid \alpha_1 = \cdots = \alpha_k  \}$ which is clearly tight. (See also \cite[Exercise 15.25]{burgisser1997algebraic}.)
When the characteristic of $\FF$ equals $n$ the tensor $f_n$ is also tight, as we will see in \cref{capset}.
\end{example}

We care about tight tensors because of the following remarkable theorem for tight 3-tensors, proved by Strassen in \cite[Lemma 5.1]{strassen1991degeneration} using the method of Coppersmith and Winograd~\cite{MR1056627}.  We need the concept of asymptotic subrank of a set (cf.~\cite[Section 5]{strassen1991degeneration}).
We say $D \subseteq I_1 \times \cdots \times I_k$ is a \defin{diagonal} when any two distinct $\alpha, \beta \in D$ are distinct in all $k$ coordinates. In other words, for elements in~$D$, the value at one coordinate uniquely determines the value at the other~$k-1$ coordinates.
Let $\Phi \subseteq I_1 \times \cdots \times I_k$.
We say a diagonal $D \subseteq I_1 \times \cdots \times I_k$ is \defin{free for $\Phi$} or simply $D \subseteq \Phi$ is a \defin{free diagonal} if $D = \Phi \cap (D_1 \times \cdots \times D_k)$, where $D_i = \{x_i \mid (x_1, \ldots, x_k) \in D\}$.
 Define the \defin{subrank} $\subrank(\Phi)$ as the size of the largest free diagonal $D \subseteq \Phi$.
For $\Phi \subseteq I_1 \times \cdots \times I_k$ and $\Psi \subseteq J_1 \times \cdots \times J_k$ we naturally define the product $\Phi \times \Psi \subseteq (I_1 \times J_1) \times \cdots \times (I_k \times J_k)$.
Define the \defin{asymptotic subrank} $\asympsubrank(\Phi) = \lim_{n \to \infty} \subrank(\Phi^{\times n})^{1/n}$. Let $t \in \FF^{n_1} \otimes \cdots \otimes \FF^{n_k}$ and let~$\Phi$ be the support of $t$ in the standard basis. Then $\subrank(\Phi) \leq \subrank(t)$ and $\asympsubrank(\Phi) \leq \asympsubrank(t)$. The number $\subrank(\Phi)$ may be interpreted as the largest number $n$ such that $\langle n\rangle$ can be obtained from $t$ using a restriction that consists of matrices that have at most one nonzero entry in each row and in each column. (This is called M-restriction in \cite[Section 6]{strassen1987relative} which stands for monomial restriction.)

Let $\Phi \subseteq [n_1] \times \cdots \times [n_k]$ and let $t \in \FF^{n_1} \otimes \cdots \otimes \FF^{n_k}$ be any tensor with support equal to $\Phi$. Then 
\[
\subrank(\Phi) \leq \subrank(t), \textnormal{ and } \asympsubrank(\Phi) \leq \asympsubrank(t).
\]
So we may upper bound $\asympsubrank(\Phi)$ by choosing the field $\FF$ and the tensor $t$ cleverly and upper bounding $\asympsubrank(t)$.

\begin{theorem}\label{cwconstr}
Let $\Phi \subseteq I_1 \times I_2 \times I_3$ be tight. Then
\begin{equation}\label{cwformula}
\asympsubrank(\Phi) = \max_{P \in \prob(\Phi)} \min \{ 2^{H(P_1)}\!,\, 2^{H(P_2)}\!,\, 2^{H(P_3)} \}.
\end{equation}
\end{theorem}

The amazing consequence of \cref{cwconstr} is that the Strassen support functionals are sufficiently powerful to compute the asymptotic subrank of tight 3-tensors.

\begin{corollary}[{\cite[Proposition 5.4]{strassen1991degeneration}}]\label{cwcor}
Let $f\in V_1 \otimes V_2 \otimes V_3$ be tight. Then
\[
\asympsubrank(f) = \min_{\theta \in \prob([3])} \zeta^\theta(f).
\]
Moreover, if $\Phi$ is a tight support for $f$, then $\asympsubrank(f) = \asympsubrank(\Phi)$.
\end{corollary}

\begin{remark}
Strassen conjectured in \cite[Conjecture~5.3]{MR1341854} that for the family $\semiring{Y}$ of tight 3-tensors the support functionals give all spectral points in the asymptotic spectrum~$\Delta(\semiring{Y})$, i.e.~the support simplex and the asymptotic spectrum coincide for $\semiring{Y}$. In \cite{strassen1991degeneration} numerous examples are given of subfamilies of $\semiring{Y}$ for which this is the case. We focus on one such example in the next section.
\end{remark}

\begin{remark}
The assumption $\Phi\subseteq I_1 \times I_2 \times I_3$ in \cref{cwconstr} is crucial. Namely, statement~\eqref{cwformula} becomes false when instead we let $\Phi\subseteq I_1 \times \cdots \times I_k$ with $k\geq4$ and we let the right-hand side of the equation be $\max_{P \in \prob(\Phi)} \min_i 2^{H(P_i)}$, see \cite[Example 1.1.38]{christandl2016asymptotic}. In~\cite{christandl2016asymptotic} the construction of \cref{cwconstr} is extended to obtain a lower bound for $\asympcombsubrank(\Phi)$ when $k\geq 4$. This lower bound is not known to be tight in general. We give an illustration in \cref{dicke} of a case in which the lower bound is tight.
\end{remark}

\begin{example}[Type tensors]\label{dicke}
Let $\lambda = (\lambda_1,\ldots,\lambda_n)$ be a partition of $k$. %
Define the set $\Phi_\lambda$ and the \defin{type tensor}~$D_\lambda$ as
\begin{align*}
\Phi_\lambda &\coloneqq \{ \alpha \in [n]^k \mid \type(\alpha) = \lambda \}\\
D_\lambda &\coloneqq \sum_{\mathclap{\alpha \in \Phi_\lambda}} e_{\alpha_1} \otimes \cdots \otimes e_{\alpha_k} \in (\FF^n)^{\otimes k}
\end{align*}
where $\type(\alpha)=\lambda$ means that the $k$-tuple $\alpha$ is a permutation of the $k$-tuple consisting of $\lambda_1$ ones, $\lambda_2$ twos, etc.
In quantum information theory (let ${\FF = \CC}$) each type tensor $D_\lambda$ encodes a so-called \defin{Dicke state of type $\lambda$}. The type tensor~$D_{(k-1, 1)}$ encodes what is called the  \defin{(generalised) W state}.
Clearly $D_\lambda$ is oblique, since the support in the standard basis $\Phi_\lambda$ is an antichain (in fact, $\Phi_\lambda$ is clearly tight), and therefore
\[
\rho^\theta(D_\lambda) = \rho_\theta(D_\lambda) =   H_\theta(\Phi_\lambda)
\]
(and over $\CC$ the quantum functionals $E^\theta(D_\lambda)$ and $E_\theta(D_\lambda)$ take the same value as $\rho^\theta(D_\lambda)$).
One verifies that $H_\theta( \Phi_\lambda ) \geq H((\tfrac{\lambda_1}{k}, \ldots, \tfrac{\lambda_n}{k}))$, with equality when~$\theta$ is uniform. Of course $H_\theta( \Phi_\lambda ) \leq n$, with equality when $\theta$ equals one of $(1,0,0,\ldots, 0)$, $(0,1,0, \ldots, 0)$, etc.
It is shown in \cite{vrana2015asymptotic} and \cite{christandl2016asymptotic}, and independently in~\cite{haviv}, that for all $k\geq3$ we have 
$\asympsubrank(D_{(1, k-1)}) = \asympcombsubrank(\Phi_{(1, k-1)}) = 2^{h(k^{-1})}$,
where $h(p)$ denotes the binary entropy function. In \cite{christandl2016asymptotic} it is shown that
$\asympsubrank(D_{(2, 2)}) = \asympcombsubrank(\Phi_{(2, 2)}) = 2$
which in \cite{srini} is extended to, for all even $k\geq4$,
$\asympsubrank(D_{(k/2, k/2)}) = \asympcombsubrank(\Phi_{(k/2, k/2)}) = 2$.
\end{example}

We are going to extend \cref{cwcor}.
Suppose $\Psi \subseteq I_1 \times I_2 \times I_3$ is not tight, but has a tight subset $\Phi \subseteq \Psi$. In the rest of this section we focus on obtaining a lower bound on $\asympcombsubrank(\Psi)$ via $\Phi$. This has an application in the context of tri-colored sum-free sets (\cref{capset}) for example. We begin with the following standard notion.

\begin{definition}\label{tighton}
Let $\Phi \subseteq \Psi \subseteq I_1 \times \cdots \times I_k$. We say that $\Phi$ is a \defin{combinatorial degeneration of $\Psi$}, and write  $\Psi \degengeq \Phi$, if there are maps $u_i : I_i \to \ZZ$ $(i\in [k])$ such that for all $\alpha \in I_1 \times \cdots \times I_k$, if $\alpha \in \Psi\setminus \Phi$, then $\sum_{i=1}^k u_i(\alpha_i) > 0$, and if~$\alpha \in \Phi$, then $\sum_{i=1}^k u_i(\alpha_i) = 0$. Note that the maps $u_i$ need not be injective.
\end{definition}

Combinatorial degeneration gets its name from the following standard proposition, see e.g.~\cite[Proposition 15.30]{burgisser1997algebraic}.

\begin{proposition}\label{combtodegen}
Let $f \in V_1 \otimes \cdots \otimes V_k$. Let $\Psi = \supp_C f$. Let $\Phi \subseteq \Psi$ such that $\Psi \degengeq \Phi$. Then $f \degengeq f|_\Phi$.
\end{proposition}

\cref{combtodegen} brings us only slightly closer to our goal. Namely, given $f \in V_1 \otimes \cdots \otimes V_k$ with $\Psi = \supp_C f$, and given $\Phi \subseteq \Psi$ such that $\Psi \degengeq \Phi$, it follows directly from \cref{combtodegen} that $f \degengeq f|_\Phi$ and thus $\asympsubrank(f) \geq \asympsubrank(f|_\Phi)$. This, however, does not give us a lower bound on $\asympcombsubrank(\Psi)$. The following theorem does. Our theorem extends the result in \cite{kleinberg2016growth}.

\begin{theorem}\label{cwwithoutbasistrans}
Let $\Phi\subseteq\Psi\subseteq I_1\times\cdots\times I_k$. If $\Psi\degengeq\Phi$, then
\[
\asympsubrank(\Psi)\ge\asympsubrank(\Phi).
\]
\end{theorem}

\begin{lemma}\label{trickslem}
Let $S\subseteq T\subseteq I_1\times\cdots\times I_k$, and $W_i\subseteq I_i$. Let $n \in \NN$. %
\begin{itemize}
\item If $T\degengeq S$, then $T^{\times n}\degengeq S^{\times n}$.
\item If $S$ is a free diagonal in $T$, then $S^{\times n}$ is a free diagonal in $T^{\times n}$.
\item If $S$ is a free diagonal in $T$, then $S\cap(W_1\times\cdots\times W_k)$ is a free diagonal in $T\cap(W_1\times\cdots\times W_k)$.
\item If $S$ is a free diagonal in $T\cap(W_1\times\cdots\times W_k)$, then $S$ is a free diagonal in~$T$.
\end{itemize}
\end{lemma}

The proof of \cref{trickslem} is straightforward and left to the reader.
First we lower bound $\asympsubrank(\Psi)$ by the (nonasymptotic) subrank $\subrank(\Phi)$.

\begin{lemma}\label{finite}
Let $\Phi\subseteq\Psi\subseteq I_1\times\cdots\times I_k$. If $\Psi\degengeq\Phi$, then $\asympsubrank(\Psi)\ge\subrank(\Phi)$.
\end{lemma}
\begin{proof}
Pick maps $u_i:I_i\to\mathbb{Z}$ such that
\begin{align*}
\sum_{i=1}^k u_i(\alpha_i) & = 0\quad\text{for $\alpha\in\Phi$}  \\
\sum_{i=1}^k u_i(\alpha_i) & > 0\quad\text{for $\alpha\in\Psi\setminus\Phi$.}
\end{align*}
Let $D$ be a free diagonal in $\Phi$ with $|D|=\subrank(\Phi)$ and let
\begin{equation*}
w_i=\sum_{x\in D_i}u_i(x).
\end{equation*}
Let $n\in\mathbb{N}$ and define
\begin{equation*}
W_i=\Bigl\{(x_1,\ldots,x_{n|D|})\in I_i^{\times n|D|} \,\Big|\, \sum_{\smash{j=1}}^{\mathclap{n|D|}}u_i(x_j)=nw_i \Bigr\}.
\end{equation*}
Then
\begin{equation*}
\Psi^{\times n|D|}\cap(W_1\times\cdots\times W_k)=\Phi^{\times n|D|}\cap(W_1\times\cdots\times W_k).
\end{equation*}
The inclusion $\supseteq$ is clear. To show $\subseteq$, let $(x_1,\ldots,x_k)\in\Psi^{\times n|D|}\cap(W_1\times\cdots\times W_k)$. Write $x_i=(x_{i,1},x_{i,2},\ldots,x_{i,n|D|})$ and consider the $n|D|\times k$ matrix of evaluations
\begin{equation*}
\begin{matrix}
u_1(x_{1,1}) & u_2(x_{2,1}) & \cdots & u_k(x_{k,1})  \\
u_1(x_{1,2}) & u_2(x_{2,2}) & \cdots & u_k(x_{k,2})  \\
\vdots & \vdots & \ddots & \vdots  \\
u_1(x_{1,n|D|}) & u_2(x_{2,n|D|}) & \cdots & u_k(x_{k,n|D|})
\end{matrix}
\end{equation*}
The sum of the $i$th column is $nw_i$ by definition of $W_i$, and $\sum_{i=1}^k nw_i=0$. The row sums are nonnegative by definition of the maps $u_1,\ldots,u_k$. We conclude that the row sums are zero. Therefore $(x_1,\ldots,x_k)$ is an element of $\Phi^{\times n|D|}$.

Since $D$ is a free diagonal in $\Phi$, $D^{\times n|D|}$ is a free diagonal in $\Phi^{\times n|D|}$, and also $D^{\times n|D|}\cap(W_1\times\cdots\times W_k)$ is a free diagonal in $\Phi^{\times n|D|}\cap(W_1\times\cdots\times W_k)$, which in turn is equal to $\Psi^{\times n|D|}\cap(W_1\times\cdots\times W_k)$. Therefore $D^{\times n|D|}\cap(W_1\times\cdots\times W_k)$ is also a free diagonal in $\Psi^{\times n|D|}$, i.e.
\begin{equation*}
\subrank(\Psi^{\times n|D|})\ge|D^{\times n|D|}\cap(W_1\times\cdots\times W_k)|.
\end{equation*}

In the set $D^{\times n|D|}$ consider the strings with uniform type, i.e.~where all $|D|$ elements of $D$ occur exactly $n$ times. These are clearly in $W_1\times\cdots\times W_k$, and their number is $\binom{n|D|}{n,\ldots,n}$. Therefore
\begin{equation*}
\subrank(\Psi^{\times n|D|})\ge\binom{n|D|}{n,\ldots,n}=|D|^{n|D|-o(n)},
\end{equation*}
which implies
$\asympsubrank(\Psi)=\lim_{n\to\infty}\subrank(\Psi^{\times n|D|})^{\frac{1}{n|D|}}\ge |D|$.
\end{proof}

\begin{proof}[\upshape\bfseries Proof of \cref{cwwithoutbasistrans}]
Clearly, we may write $\asympsubrank(\Psi) = \lim_{n\to\infty}\asympsubrank(\Psi^{\times n})^{1/n}$. By \cref{finite}, 
\[
 \lim_{n\to\infty}\asympsubrank(\Psi^{\times n})^{1/n} \geq \lim_{n\to\infty}\subrank(\Phi^{\times n})^{1/n}.
 \]
The right-hand side equals $\asympsubrank(\Phi)$ by definition.
\end{proof}

\subsection{The cap set problem}\label{capsetsec} %

In this section we go on a combinatorial excursion. A subset $A \subseteq (\ZZ/3\ZZ)^n$ is called a \defin{cap set} if any line in $A$ is a point, a line being a triple of points of the form $(u, u+v, u+2v)$. The \defin{cap set problem} is to decide whether the maximal size of a cap set in $(\ZZ/3\ZZ)^n$ grows like $3^{n - o(n)}$ or like $c^{n-o(n)}$ for some $c < 3$. Gijswijt and Ellenberg in \cite{MR3583358}, inspired by the work of Croot, Lev and Pach in~\cite{MR3583357}, settled this problem, showing that $c\leq 3 (207 + 33 \sqrt{33})^{1/3} /8 \approx 2.755$.
Tao realised in \cite{tao} that the cap set problem may naturally be phrased as the problem of computing the size of the largest main diagonal in powers of the \defin{cap set tensor}
\[
\sum_{\alpha} e_{\alpha_1} \otimes e_{\alpha_2} \otimes e_{\alpha_3}
\]
where the sum is over $\alpha_1, \alpha_2, \alpha_3 \in \FF_3$ with $\alpha_1 + \alpha_2 + \alpha_3 = 0$. Here \emph{main diagonal} refers to a subset $A$ of the basis elements such that restricting the cap set tensor to $A\times A \times A$ gives the tensor $\sum_{v \in A} v \otimes v \otimes v$.  We show (in hindsight!)\ that the cap set tensor is in the $\GL_3(\FF_3)^{\times 3}$-orbit of the \defin{reduced polynomial multiplication} tensor, the prime example in \cite{strassen1991degeneration}, and we show how all recent results follow from this connection, using \cref{cwwithoutbasistrans}. We first state Strassen's result on reduced polynomial multiplication and then the cap set problem.

\subsubsection{Reduced polynomial multiplication} %
\label{redpolmult}
Let $F \in \FF[x]$ be a polynomial of positive degree $n$. 
Let~$\FF[x]/(F)$ be the algebra of univariate polynomials modulo the ideal generated by the polynomial $F$. \emph{Reduced polynomial multiplication} refers to multiplying in this algebra. For any algebra~$A$ with multiplication map $m: A\times A \to A$ the structure tensor is defined as $\sum a_3^*( m(a_1,a_2) )\,a_1\otimes a_2 \otimes a_3$, where the sum goes over elements~$a_1, a_2, a_3$ in some basis of~$A$. Let~$\FF[x]/(F)$ denote the structure tensor of the algebra~$\FF[x]/(F)$. 
Let us take $F = x^n$. For example, $\FF[x]/(x^2)$ is the tensor
\[
e_0 \otimes e_0 \otimes e_0 + e_0 \otimes e_1 \otimes e_1 + e_1 \otimes e_0 \otimes e_1
\]
and $\FF[x]/(x^3)$ is the tensor
\begin{align*}
&e_0 \otimes e_0 \otimes e_0 + e_0 \otimes e_1 \otimes e_1 + e_1 \otimes e_0 \otimes e_1\\
+\, &e_0 \otimes e_2 \otimes e_2 + e_2 \otimes e_0 \otimes e_2 + e_1 \otimes e_1 \otimes e_2
\end{align*}
and $\FF[x]/(x^n)$ is the tensor $\sum_{\alpha} e_{\alpha_1} \otimes e_{\alpha_2} \otimes e_{\alpha_3}$
where the sum is over $(\alpha_1, \alpha_2, \alpha_3)$ in $\{0, 1, \ldots, n-1\}^3$ such that $\alpha_1 + \alpha_2 = \alpha_3$.
The support of $\FF[x]/(x^n)$ equals
\[
\bigl\{(\alpha_1, \alpha_2, \alpha_3) \in \{0, \ldots, n-1\}^3 \,\big|\, \alpha_1 + \alpha_2 = \alpha_3 \bigr\}.
\]
which via $\alpha_3 \mapsto n-1 - \alpha_3$ we may identify with the set
\begin{equation} \label{reducedpol}
\Phi_n = \bigl\{ (\alpha_1, \alpha_2, \alpha_3) \in \{0,\ldots, n-1\}^3 \,\big|\, \alpha_1 + \alpha_2 + \alpha_3 = n-1 \bigr\}.
\end{equation}
The support $\Phi_n$ is tight (cf.~\cref{tightexample}).
We thus already know that the support simplex $\zeta^{\prob([3])}(\FF[x]/(x^n)) = \{\zeta^\theta(\FF[x]/(x^n)) \mid \theta \in \prob([3])\}$ is a subset of $\Delta(\FF[x]/(x^n))$, and we know that the minimum value of $\zeta^\theta(\FF[x]/(x^n))$ over $\theta\in  \prob([3])$ equals the asymptotic subrank $\asympsubrank(\Phi_n)$ (\cref{cwcor}).
Strassen proves in \cite[Theorem~6.7]{strassen1991degeneration} the remarkable result that for $\FF[x]/(x^n)$ the asymptotic spectrum and the support simplex in fact coincide and are equal to the interval
\[
\Delta\bigl(\FF[x]/(x^n)\bigr) = \zeta^{ \prob([3])}\bigl(\FF[x]/(x^n)\bigr) =  [z(n),\, n]
\]
where $z(n)$ is defined as
\begin{equation}\label{defz}
z(n) \coloneqq \frac{\gamma^n - 1}{\gamma - 1} \gamma^{-2(n-1)/3}
\end{equation}
with $\gamma$ equal to the unique real positive solution of the equation $\tfrac{1}{\gamma-1} - \frac{n}{\gamma^n-1} = \frac{n-1}{3}$.
In particular, $\asympsubrank(\Phi_n) = z(n)$.
We collect some (rounded) values of $z(n)$ for small~$n$ in the following table. For~$z(2)$ and~$z(3)$ one can write down the exact value. See also \cite[Table 1]{strassen1991degeneration}.
\begin{center}
\begin{tabular}{lll}
\toprule
$n$ 	& $z(n)$ &\\
\cmidrule(l){2-3}
	& rounded & exact\\
\midrule
2  & 1.88988 & $3/2^{2/3} = 2^{h(1/3)}$\\
3  & 2.75510 & $3 (207 + 33 \sqrt{33})^{1/3} /8$\\
4  & 3.61072 \\
5  & 4.46158 \\
6  & 5.30973 \\
7  & 6.15620 \\
8  & 7.00155 \\
9  & 7.84612 \\
10 & 8.69012 \\
\bottomrule
\end{tabular}
\end{center}

Strassen in \cite[Theorem~6.7]{strassen1991degeneration} in fact computes the asymptotic spectrum of $\FF[x]/(F)$ for arbitrary~$F$. The result is as follows. Let $F = \beta \prod_{i=1}^r (x - \alpha_i)^{n_i}$ be the prime factorisation of~$F$ over the algebraic closure of $\FF$ such that the $\alpha_i$ are pairwise distinct (i.e.~the~$n_i$ are the multiplicities of the roots of $F$). Then the asymptotic spectrum and support simplex of $\FF[x]/(F)$ coincide and equal the following interval,
\[
\Delta\bigl(\FF[x]/(F)\bigr) = \zeta^{\prob([3])}\bigl(\FF[x]/(F)\bigr) =  \Bigl[ \sum_{i=1}^r z(n_i),\, n \Bigr].
\]
where $z(n_i)$ is the function defined in \eqref{defz}.

Franz Mauch in \cite{mauch} computes the support simplex for a natural higher-order generalisation of the support $\Phi_n$ and the tensor $\FF[x]/(x^n)$ which he calls~a~\emph{Hang}.

\subsubsection{The cap set problem}\label{capset}

With the asymptotic spectrum of reduced polynomial multiplication at our disposal, we turn to the cap set problem. We begin by reintroducing the problem in the terminology that is commonly used in the literature.

\begin{definition}
A \defin{three-term progression-free set} is a set $A\subseteq (\ZZ/m\ZZ)^n$ satisfying the following. For all $(x_1, x_2, x_3) \in A^{\times 3}$: there are $u,v \in (\ZZ/m\ZZ)^n$ such that $(x_1, x_2, x_3) = (u, u+v, u+2v)$ if and only if $x_1 = x_2 = x_3$. 
Let~$r_3((\ZZ/m\ZZ)^n)$ be the size of the largest three-term progression-free set in $(\ZZ/m\ZZ)^n$  and define the regularisation $\underaccent{\wtilde}{r}_3(\ZZ/m\ZZ) = \lim_{n\to \infty} r_3((\ZZ/m\ZZ)^n)^{1/n}$.
\end{definition}

A three-term progression-free set in $(\ZZ/3\ZZ)^n$ is called a \defin{cap} or \defin{cap set}.
The \defin{cap set problem} is to decide whether $\underaccent{\wtilde}{r}_3(\ZZ/3\ZZ)$ equals 3 or is strictly smaller than~3. We next introduce an asymmetric variation on three-term progression free sets, called tri-colored sum-free sets, which are potentially larger. They are interesting since all known upper bound techniques for the size of three-term progression-free set turn out to be upper bounds on the size of tri-colored sum-free sets.

\begin{definition}
Let $G$ be an abelian group. Let $\Gamma \subseteq G\times G \times G$. For~$i \in [3]$ we define the marginal sets $\Gamma_i = \{x \in G \mid \exists \alpha \in \Gamma : \alpha_i = x\}$. We say $\Gamma$ is \defin{tricolored sum-free} if the following holds. The set $\Gamma$ is a diagonal, and for any $\alpha \in \Gamma_1 \times \Gamma_2 \times \Gamma_3$: $\alpha_1 +\alpha_2 + \alpha_3 = 0$ if and only if $\alpha \in \Gamma$. 
(Recall that $\Gamma\subseteq I_1 \times I_2 \times I_3$ is a \defin{diagonal} when any two distinct $\alpha, \beta \in \Gamma$ are distinct in all coordinates.)
Let~$s_3(G)$ be the size of the largest tricolored sum-free set in $G\times G \times G$ and define the regularisation $\regularize{s}_3(G) = \lim_{n\to \infty} s_3(G^{\times n})^{1/n}$.
\end{definition}

Equivalently, $\Gamma \subseteq G\times G\times G$ is a tricolored sum-free set if and only if $\Gamma$ is a free diagonal in $\{\alpha \in G\times G\times G \mid \alpha_1 + \alpha_2 + \alpha_3 = 0\}$.

Clearly, if the set $A\subseteq G = (\ZZ/m\ZZ)^n$ is three-term progression-free, then the set $\Gamma = \{(a, a, -2a) : a\in A\}\subseteq G \times G \times G$ is tri-colored sum-free. Therefore, we have %
$\regularize{r}_3(\ZZ/m\ZZ) \leq \regularize{s}_3(\ZZ/m\ZZ)$. 

Let us briefly summarise the recent history on the cap set problem. For clarity we focus on $m=3$; we refer the reader to the references for the general results.
Edel in \cite{MR2031694} proved the lower bound $2.21739 \leq \underaccent{\wtilde}{r}_3(\ZZ/3\ZZ)$.
Ellenberg and Gijswijt in~\cite{MR3583358} proved the upper bound 
\begin{align*}
\underaccent{\wtilde}{r}_3(\ZZ/3\ZZ) &\leq  3 (207 + 33 \sqrt{33})^{1/3} /8 \approx 2.755,
\intertext{settling the cap set problem, but leaving open the problem of computing $\underaccent{\wtilde}{r}_3(\ZZ/3\ZZ)$.
Blasiak et~al.~\cite{MR3631613} proved that in fact}
\regularize{s}_3(\ZZ/3\ZZ) &\leq 3 (207 + 33 \sqrt{33})^{1/3} /8. \label{blasiakbound} %
\end{align*}
This upper bound was shown to be an equality in the three papers \cite{kleinberg2016growth, norin2016distribution, pebody2016proof}.

\begin{theorem}\label{capsetth}
$\regularize{s}_3(\ZZ/3\ZZ) = 3 (207 + 33 \sqrt{33})^{1/3} /8$.
\end{theorem}

We reprove \cref{capsetth} by proving that $\regularize{s}_3(\ZZ/m\ZZ)$ equals the asymptotic subrank $z(m)$ of $\FF_m[x]/(x^m)$ discussed in \cref{redpolmult}, when $m$ is a prime power. We emphasise that the significance of our proof lies in the explicit connection to the framework of asymptotic spectra and not in the obtained value, which also for prime powers~$m$ was already computed in \cite{MR3631613,kleinberg2016growth, norin2016distribution, pebody2016proof}. %

\begin{proof}
We will prove $\regularize{s}_3(\ZZ/m\ZZ) = z(m)$ when $m$ is a prime power.
By definition $\regularize{s}_3(\ZZ/m\ZZ)$ is equal to the asymptotic subrank of the set
\[
\{\alpha \in \{0, \ldots, m-1\}^3 \mid \alpha_1 + \alpha_2 + \alpha_3 = 0 \bmod m\}
\]
which via $\alpha_3 \mapsto \alpha_3 - (m-1)$ we may identify with the set
\begin{align*}
\Psi_m &= \{ \alpha \in \{0,\ldots, m-1\}^3 \mid \alpha_1 + \alpha_2 + \alpha_3 = m-1 \bmod m \}
\intertext{and so $\regularize{s}_3(\ZZ/m\ZZ) = \asympsubrank(\Psi_m)$. Define the tight set (cf.~\cref{tightexample})}
\Phi_m &= \{ \alpha \in \{0,\ldots, m-1\}^{\smash{3}} \mid \alpha_1 + \alpha_2 + \alpha_3 = m-1 \}.
\end{align*}
We know $\asympsubrank(\Phi_m) = z(m)$ (\cref{redpolmult}). We will show that $\asympsubrank(\Phi_m) = \asympsubrank(\Psi_m)$ when $m$ is a prime power. This proves the theorem.

We first prove $\asympsubrank(\Phi_m) \leq \asympsubrank(\Psi_m)$.
There is %
a combinatorial degeneration $\Psi_m \degengeq \Phi_m$. 
Indeed, let $u_i : \{0,\ldots, m-1\} \to \{0,\ldots, m-1\}$ be the identity map. If $\alpha \in \Phi_m$, then $\sum_{i=1}^3 u_i(\alpha_i) = m-1$, and if $\alpha \in \Psi_m \setminus \Phi_m$, then $\sum_{i=1}^3 u_i(\alpha_i)$ equals $m-1$ plus a positive multiple of~$m$.
This means \cref{cwwithoutbasistrans} applies, and we thus obtain
$\asympsubrank(\Phi_m) \leq \asympsubrank(\Psi_m)$. %
This proves the claim.

We show $\asympsubrank(\Psi_m) \leq \asympsubrank(\Phi_m)$ when $m$ is a power of the prime $p$.
Let $\FF = \FF_p$.
Let $f_m \in \FF^m \otimes \FF^m \otimes \FF^m$ have support~$\Psi_m$ with all nonzero coefficients equal to~1. Obviously, $\asympsubrank(\Psi_m) \leq \asympsubrank(f_m)$.
To compute $\asympsubrank(f_m)$ we show that there is a basis in which the support of $f_m$ equals the tight set $\Phi_m$. Then $\asympsubrank(f_m) = \asympsubrank(\Phi_m)$ (\cref{cwcor}). This implies the claim.
We prepare to give the basis (which is the same basis as used in \cite{MR3631613}).
First observe that the rule $x \mapsto \binom{x}{a}$ gives a well-defined map $\ZZ/m\ZZ \to \ZZ/p\ZZ$, since for $a\in \{0,1,\ldots, m-1\}$, if $x = y \bmod m$ then $\binom{x}{a} = \binom{y}{a} \bmod p$ by Lucas' theorem. 
Let~$(e_x)_x$ be the standard basis of~$\FF^m$.
The elements $(\sum_{x\in \ZZ/m\ZZ} \binom{x}{a} e_x)_{a \in \ZZ/m\ZZ}$ form a basis of $\FF^m$ since the matrix $(\binom{x}{a})_{a,x}$ is upper triangular with ones on the diagonal. We will now rewrite $f_m$ in the basis $((\sum_x \binom{x}{a} e_x)_a, (\sum_y \binom{y}{b} e_y)_b, (\sum_z \binom{z}{c} e_z)_c)$.
Observe that $\binom{x}{m-1}$ equals 1 if and only if $x$ equals $m-1$, and hence
\begin{align*}
f_m &= \hspace{-0.5em}\sum_{\substack{x, y, z \in \ZZ/m\ZZ:\\ x+y+z = m-1}} \hspace{-1em} e_x \otimes e_y \otimes e_z
\hspace{0.5em} = \hspace{-0.5em} \sum_{x, y, z \in \ZZ/m\ZZ}\! \binom{x+y+z}{m-1} e_x \otimes e_y \otimes e_z.
\end{align*}
The identity $\binom{x+y+z}{w} = \sum \binom{x}{a} \binom{y}{b} \binom{z}{c}$ with sum over $a,b,c \in \{0,1, \ldots, m-1\}$ such that $a+ b+ c= w$ is true and thus
\begin{align}
&\sum_{x, y, z \in \ZZ/m\ZZ} \binom{x+y+z}{m-1} e_x \otimes e_y \otimes e_z \nonumber\\
 &\quad= \sum_{x, y, z \in \ZZ/m\ZZ}\; \sum_{\substack{a,b,c\in \{0, 1, \ldots, m-1\}:\\a+b+c=m-1}} \binom{x}{a}\binom{y}{b}\binom{z}{c} e_x \otimes e_y \otimes e_z.\label{xxx}
\end{align}
We may simply rewrite \eqref{xxx} as
\[
\sum_{\substack{a,b,c\in \{0, 1, \ldots, m-1\}:\\a+b+c=m-1}}\hspace{0.5em}  \sum_{x\in \ZZ/m\ZZ}\!\! \binom{x}{a} e_x \,\otimes \!\sum_{y\in \ZZ/m\ZZ}\!\! \binom{y}{b} e_b\, \otimes \!\sum_{z\in \ZZ/m\ZZ} \!\!\binom{z}{c} e_z.
\]
Therefore, with respect to the basis $((\sum_x \binom{x}{a} e_x)_a, (\sum_y \binom{y}{b} e_y)_b, (\sum_z \binom{z}{c} e_z)_c)$, the support of $f_m$ equals the tight set $\Phi_m$. 
(And even stronger, $f_m$ is isomorphic to the tensor $\FF[x]/(x^m)$ of \cref{redpolmult}.)  %
\end{proof}

\subsection{Free tensors}

Let $\FF$ be the complex numbers $\CC$.
We consider a family of tensors called free tensors. %
These were introduced by Franz in \cite{MR1923785}.
Recall that for any $\theta \in \prob_\s(B)$ and any tensor $t$
\begin{align*}
E^\theta(t) &= E_\theta(t) \textnormal{\quad(\cref{upperequalslower})}
\shortintertext{and}
\rho^\theta(t) &\geq E^\theta(t) \textnormal{\quad(\cref{suppinv})}.
\end{align*}
We care about free tensors, because for any free tensor $t$ we can show that the three values $\rho^\theta(t)$, $E^\theta(t)$, $E_\theta(t)$ coincide for $\theta \in \prob_s(B)$. 

\begin{definition}[Free]
Let $\Phi \subseteq I_1 \times \cdots \times I_k$. We say $\Phi$ is \defin{free} if the following holds. For any $x,y \in \Phi$, if $x\neq y$, then the tuples $x$ and $y$ differ in at least two positions.
Let $t \in V_1 \otimes \cdots \otimes V_k$. We say $t$ is \defin{free} if $\supp_C t$ is free for some choice of bases~$C \in \bases(t)$.
\end{definition}

\begin{remark}
Clearly, oblique tensors are free. We thus have $\{\textnormal{tight}\} \subseteq \{ \textnormal{oblique} \} \subseteq \{ \textnormal{free} \}$.
Free tensors from a semigroup under $\otimes$ and $\oplus$, like the tight tensors and oblique tensors.
\end{remark}

\begin{remark}
We prove there exist tensors that are not free in $\CC^n \otimes \CC^n \otimes \CC^n$ for $n \geq5$.
First we upper bound the maximal size of a free support.
Let $\Phi \subseteq [n] \times [n] \times [n]$ be free.
Then $\abs[0]{\Phi} = \abs[0]{\{(\alpha_1, \alpha_2) : \alpha \in \Phi\}} \leq n^2$.
Second we apply a lemma of Bürgisser~\cite{burg}, which is as follows. Let 
\[
Z_n = \{t \in \CC^n \otimes \CC^n \otimes \CC^n : \exists C \in \bases(t)\, \abs[0]{\supp_C t} < n^3 - 3n^2\}.
\]
Let $Y_n = \CC^n\otimes \CC^n \otimes \CC^n \setminus \overline{Z_n}$. Then the set $Y_n$ is Zariski open and nonempty.
Now let $n \geq 5$ and let $t \in Y_n$. Then $\forall C \in \bases(t)$\, $\abs[0]{\supp_C t} \geq n^3 - 3n^2 > n^2$. We conclude~$t$ is not free.
\end{remark}

\begin{theorem}\label{freetheorem}
Let $t$ be free. Let $\theta \in \prob_\s(B)$. Then $\rho^\theta(t) = E^\theta(t) = E_\theta(t)$.
\end{theorem}
\begin{proof}
By \cref{suppinv} and \cref{upperequalslower} we have $\rho^\theta(t) \geq E^\theta(t) = E_\theta(t)$. We prove $\rho^\theta(t) \leq E_\theta(t)$. 
Let $C \in \bases(t)$ such that $\supp_C t$ is free.
It is sufficient to prove the following claim. For any $P \in \prob(\supp_C t)$ the tuple of ordered marginals of $P$ is in the entanglement polytope~$\Delta_t$. This result can be found in \cite{MR1645052, MR1923785, MR2193441}, see also \cite{wernli}. Then $\rho^\theta(t) \leq \max_{P \in \prob(\supp_C t)} H_\theta(P) \leq E_\theta(t)$.
We give a proof of the above claim.
Write $C = ((e_{\alpha_1}), \ldots, (e_{\alpha_k}))$. Let $T \subseteq G = \SL_{n_1} \times \cdots \times \SL_{n_k}$ be the corresponding torus.
Let $\mu_G$ map any density matrix $\rho$ to its reduced density matrices,
\[
\mu_G : \rho \mapsto (\rho_1, \ldots, \rho_k).
\]
Let $\mu_T$ map any density matrix to its reduced density matrices projected onto diagonal matrices by setting the off-diagonal entries to zero,
\[
\mu_T : \rho \mapsto ( P_{\mathrm{diag}}\, \rho_1, \ldots, P_{\mathrm{diag}}\, \rho_k ).
\]
The maps $\mu_G$ and $\mu_T$ can be identified with the moment map for $K = \SU_{n_1} \times \cdots \times \SU_{n_k}$ and abelian moment map for $T \subseteq K$ in the sense of \cite{MR932055, MR2193441}. We write $\mu_G(\psi)$ for $\mu_G(\ketbra{\psi}{\psi})$ and we write $\mu_T(\psi)$ for $\mu_T(\ketbra{\psi}{\psi})$.
Let $D_{\downarrow}$ be the set of $k$-tuples of diagonal marginals with ordered diagonal entries.
It is a nontrivial fact that $\mu_T(\overline{T\cdot t})$ is a convex polytope, and that $\mu_G(\overline{G\cdot t})\cap D_{\downarrow}$ is a convex polytope, namely the entanglement polytope $\Delta_t$ \cite{MR932055}.
The elements of $\mu_T(\overline{T\cdot t})$ and $\mu_G(\overline{G\cdot t})\cap D_{\downarrow}$ are tuples of diagonal matrices whose diagonal entries are a probability distribution. We will identify these elements with these tuples of probability distributions.
Let $\Phi = \supp_C t$ be the support of $t$. Let $\alpha \in \Phi$. Then
\[
e_{\alpha_1} \otimes \cdots \otimes e_{\alpha_k} \in \overline{T \cdot t}.
\]
Therefore, the $k$-tuple of point distributions $(\delta_{\alpha_1}, \ldots, \delta_{\alpha_k})$ is in $\mu_T(\overline{T\cdot t})$. 
Since the set $\mu_T(\overline{T\cdot t})$ is convex, the convex hull $\conv \{(\delta_{\alpha_1}, \ldots, \delta_{\alpha_k}) \mid \alpha \in \Phi\}$ is a subset of $\mu_T(\overline{T\cdot t})$.
Observe that $\conv \{(\delta_{\alpha_1}, \ldots, \delta_{\alpha_k}) \mid \alpha \in \Phi\}$ equals the set $M = \{(P_1, \ldots, P_k) \mid P \in \prob(\supp_C t) \}$. So 
\[
M \subseteq \mu_T(\overline{T\cdot t}).
\]
Let $M_{\downarrow} \subseteq M$ consist of the tuples of ordered marginal distributions, then
\[
M_{\downarrow} \subseteq \mu_T(\overline{T\cdot t}) \cap D_{\downarrow}.
\]
Now suppose that $\supp_C t$ is free. 
The freeness implies that the marginal density matrices of $t$ are diagonal matrices and thus $\mu_G(\overline{T\cdot t})$ consists of diagonal matrices, i.e.~$\mu_T(\overline{T\cdot t}) = \mu_G(\overline{T\cdot t})$. Clearly, $\mu_G(\overline{T\cdot t}) \subseteq \mu_G(\overline{G\cdot t})$. Intersecting with $D_{\downarrow}$ gives
\[
M_{\downarrow} \subseteq \mu_T(\overline{T\cdot t}) \cap D_{\downarrow} \subseteq \mu_G(\overline{G\cdot t}) \cap D_{\downarrow} = \Delta_t.
\]
We conclude that for any $P \in \prob(\supp_C t)$ the tuple of ordered marginals of $P$ is in the entanglement polytope $\Delta_t$. 
\end{proof}

\begin{remark}
In the special case of $t$ being oblique, the statement in our \cref{freetheorem} corresponds to Corollary 12 in the paper \cite{MR2138544} of Strassen.
\end{remark}

\begin{example}\label{CW}
Coppersmith and Winograd in \cite{MR1056627} obtained the upper bound $\omega \leq 2.41$ on the matrix multiplication exponent $\omega$ by analysing the \defin{easy Coppersmith--Winograd tensor}, which for $q\geq 1$, is defined as
\[
\CW_q = \frac{1}{\sqrt{3q}}\Bigl( \sum_{i=1}^q \ket{0 i i} + \ket{i 0 i} + \ket{i i 0}\Bigr)  \in \CC^{q+1} \otimes \CC^{q+1} \otimes \CC^{q+1}.
\]
This tensor is free, so $E_\theta(\CW_q)$ equals $\rho^\theta(\CW_q)$. We compute for any $\theta \in \prob([3])$ the lower bound
\begin{align*}
E_\theta(\CW_q) &\geq H\Bigl( \frac{q}{3q} \ketbra{0}{0} + \frac{2}{3q} \sum_{j=1}^q \ketbra{j}{j} \Bigr) \\
&= H\bigl(\tfrac13, \underbrace{\tfrac23 \tfrac1q, \ldots, \tfrac23 \tfrac1q}_{q}\bigr)\\
&= \tfrac23 \log_2 q + h(\tfrac13).
\end{align*}
On the other hand, let $\theta = (\tfrac13,\tfrac13,\tfrac13)$ be uniform. Then using the KKT-conditions (see \cite[Proposition 2.1]{strassen1991degeneration}) and the recursion property of Shannon entropy
we compute the upper bound
\begin{align*}
\rho^{\bigl(\tfrac13,\tfrac13,\tfrac13\bigr)}(\CW_q) &\leq H_{\bigl(\tfrac13,\tfrac13,\tfrac13\bigr)}\bigl( \{(0,i,i), (i,0,i), (i,i,0) : i \in [q]\}\bigr)\\ 
&= H\bigl(\tfrac13, \underbrace{\tfrac23 \tfrac1q, \ldots, \tfrac23 \tfrac1q}_{q}\bigr)\\
&= \tfrac23 \log_2 q + h(\tfrac13).
\end{align*}
We conclude that $\rho^\theta(\CW_q) \geq \tfrac23 \log_2 q + h(\tfrac13)$ with equality when~$\theta$ is uniform. Obviously, $\rho^\theta(\CW_q) \leq \log_2 (q+1)$ with equality when $\theta$ equals $(1,0,0)$ or $(0,1,0)$ or $(0,0,1)$.

Let $q = 1$. Then the tensor $\CW_1$ is tight and therefore oblique, which implies that $\rho_\theta(\CW_1) = \rho^\theta(\CW_1) \geq h(\tfrac13)$ with equality when $\theta$ is uniform. In fact $\CW_1$ equals the tensor $D_{(2,1)}$ (the W state) from \cref{dicke}.
\end{example}

\begin{remark}
As shown in \cref{CW}, the support and quantum functionals give for all $q\geq 2$ the lower bound $\asymprank(\CW_q) \geq q+1$.
If %
the asymptotic rank $\asymprank(\CW_2)$ equals $3$, then the exponent of matrix multiplication $\omega$ equals~2 
(see e.g.~\cite[Remark 15.44]{burgisser1997algebraic} and \cite[Section 9]{blaser2013fast}). Bläser and Lysikov in \cite{blser_et_al:LIPIcs:2016:6434} obtained the border rank lower bound $\borderrank(\CW_q^{\otimes n}) \geq (q+1)^n + 2^n - 1$ which asymptotically also implies $\asymprank(\CW_q) \geq q+1$, and thus leaves open the possibility that $\asymprank(\CW_q) = q+1$.
\end{remark}

\begin{remark}
We summarise the known relationships among the support and quantum functionals.
\begin{enumerate}
\item $F^{\theta}(t) \geq F_\theta(t)$\, if $\theta \in \prob(B)$\,\hfill (\cref{lowerupper})%
\item $\zeta^\theta(t) \geq \lim_{n\to \infty} \zeta^{\theta}(t^{\otimes n})^{1/n} = F^\theta(t) =  F_\theta(t) \geq \lim_{n\to \infty} \zeta_{\theta}(t^{\otimes n})^{1/n} \geq \zeta_\theta(t)$\\[1ex]\, if $\theta \in \prob_\s(B)$\,\\[1ex]
\null\hfill (\cref{suppinv,upperequalslower,regsupp})
\item $\zeta^\theta(t) = \zeta_\theta(t)$\, if $\theta \in \prob_\s(B)$ and $t$ is oblique\, \hfill(\cref{obliquetheorem})
\item $\zeta^\theta(t) = F^\theta(t) = F_\theta(t) \geq \zeta_\theta(t)$\, if $\theta \in \prob_s(B)$ and $t$ is free\,\hfill (\cref{freetheorem})
\end{enumerate}
\end{remark}

\subsection{Generic tensors}\label{generic}

The quantum functionals over $\CC$ and the support functionals over algebraically closed fields, when restricted to tensors of a specific format, have a generic value, i.e.~there is a Zariski open set on which the value of the functional is constant. Recall that in an irreducible affine variety (e.g.~$\CC^{n_1 \times n_2 \times n_3}$) the intersection of any two nonempty Zariski open sets is nonempty, and therefore in an irreducible affine variety any nonempty Zariski open set is Zariski dense.

\subsubsection{Generic value of the Strassen support functionals}
Let $\FF$ be algebraically closed.
For any format $(n_1, n_2, n_3)$ there is a nonempty Zariski open subset of $\FF^{n_1 \times n_2 \times n_3}$ on which $\zeta^\theta$ and $\zeta_\theta$ each have a constant value. (Indeed, a constructible set $X$ in an affine variety~$Y$ contains a subset that is open and dense in the closure of $X$ in $Y$, see e.g.~\cite{MR1102012}. There are only finitely many possible supports for tensors in $V = \FF^{n_1 \times n_2 \times n_3}$. The map~$\zeta^\theta$ therefore attains only finitely many values on $V$ and thus has finitely many fibres, each of which is a constructible set. At least one of these fibres is dense in $V$ by irreducibility of~$V$. This fibre contains a subset that is open and dense in $V$. The same argument holds for $\zeta_\theta$.) 

\begin{definition}
The above value of $\zeta^\theta$ and $\zeta_\theta$ is called \defin{typical} or \defin{generic}, and is denoted by $\zeta^\theta(n_1, n_2, n_3)$ and $\zeta_\theta(n_1, n_2, n_3)$ respectively.
\end{definition}

Obviously,  $0 \leq \zeta_\theta(n_1, n_2, n_3) \leq \zeta^\theta(n_1, n_2, n_3) \leq n_1^{\theta(1)} n_2^{\theta(2)} n_3^{\theta(3)}$.

\begin{definition}
Let $\Psi\subseteq [n_1] \times [n_2] \times [n_3]$. We call $\Psi$ \defin{comfortable} if there is a subset $\Phi\subseteq \Psi$ and a probability distribution $P \in \prob(\Phi)$ such that $\Phi$ is an antichain and each marginal $P_i$ is uniform on $[n_i]$. We call a format $(n_1, n_2, n_3)$ comfortable if $[n_1] \times [n_2] \times [n_3]$ itself is comfortable. (In \cite{tobler} comfortable is called \emph{bequem}.)
\end{definition}

\begin{example}
The format $(n,n,n)$ is comfortable. Namely, let $\Phi$ be the set $\{(\alpha, \alpha, n+1-\alpha) \mid \alpha \in [n]\}$ and let $P \in \prob(\Phi)$ be the uniform probability distribution. Clearly $\Phi$ is an antichain and each marginal $P_i$ is uniform on~$[n_i]$. The format $(ab,bc,ca)$ is comfortable. Namely let $\Phi$ be the support of the matrix multiplication tensor $\langle a,b,c \rangle$ and let $P \in \prob(\Phi)$ be the uniform probability distribution. The set $\Phi$ is tight and hence an antichain with respect to some product order, and $P_1$ is uniform on $[ab]$, $P_2$ is uniform on $[bc]$ and $P_3$ is uniform on $[ca]$.
\end{example}

\begin{theorem}[\cite{tobler, burg}]\label{gen} Let $\theta \in \prob([3])$. \hfill
\begin{enumerate}
\item $\zeta^\theta(n_1,n_2,n_3) = n_1^{\theta(1)} n_2^{\theta(2)} n_3^{\theta(3)}$ when $(n_1, n_2, n_3)$ is comfortable.
\item $\zeta_\theta(n,n,n) = n^{1- \min_i \theta(i) + o(1)}$ when $n \to \infty$.
\end{enumerate}
\end{theorem}

Recall that the upper support functional is additive. An interesting byproduct of Bürgisser's proof of statement 2 (see \cite[Theorem 3.1]{burg}) is that the lower support functional is not additive.

\begin{proposition}
Let $\theta \in \prob([3])$ with $\theta_i > 0$ for all $i$. Then $\zeta_\theta$ is not additive.
\end{proposition}

\subsubsection{Generic value of the quantum functionals} Let $\FF$ be the complex numbers~$\CC$. We consider the singleton regime $\theta \in \prob([k])$, so that we can apply the theory of entanglement polytopes. Namely, recall that by \cref{upperequalslower}, in this regime, for $t \in \CC^{n_1} \otimes \cdots \otimes \CC^{n_k}$,  the value of $E_\theta(t)$ is obtained by the following maximisation over the entanglement polytope $\Delta_t$,
\[
E_\theta(t) = E^\theta(t) = \max \bigl\{ \theta(1) H(r_1) + \cdots + \theta(k) H(r_k) \,\big|\, (r_1, \ldots, r_k) \in \Delta_t \bigr\}.
\]
Fixing the tensor format $(n_1, \ldots, n_k)$, there exist finitely many covariants $P_1, \ldots, P_m$ such that, with $(\lambda^{(j,1)}, \ldots, \lambda^{(j, k)})$ being the type of $P_j$, the entanglement polytope of any $t \in \CC^{n_1} \otimes \cdots \otimes \CC^{n_k}$ equals the convex hull
\[
\Delta_t = \conv \{ (\overline{\lambda^{(j,1)}}, \ldots, \overline{\lambda^{(j, k)}}) \mid P_j(t) \neq 0 \},
\]
see \cite[Theorem 2]{MR3087706}. This implies that there are only finitely many possible entanglement polytopes~$\Delta_t$ associated to $t \in \CC^{n_1 \times \cdots \times n_k}$ and thus $F_\theta$ takes only finitely many values. Moreover there is a nonempty Zariski open subset $U$ of $\CC^{n_1 \times \cdots \times n_k}$ on which $F_\theta$ has a constant value. Namely let $U$ be the Zariski open set of tensors for which all generators are nonzero, 
\[
U = \bigcap_{j=1}^m \{ t \in \CC^{n_1\times \cdots \times n_k} \mid P_j(t) \neq 0\}.
\]
For every $t \in U$ the entanglement polytope $\Delta_t$ is maximal, i.e.~equal to the convex hull $\conv \{ (\overline{\lambda^{(j,1)}}, \ldots, \overline{\lambda^{(j, k)}}) \mid j \in [m] \}.$

\begin{definition}
The above value of $F_\theta$ is called \defin{typical} or \defin{generic}, and is denoted by $F_\theta(n_1, \ldots n_k)$.
\end{definition}

Obviously, $0 \leq F_\theta(n_1, \ldots, n_k) \leq n_1^{\theta(1)}\cdots n_k^{\theta(k)}$. 

We say that $\ket{\psi} \in \CC^{n_1} \otimes \cdots \otimes \CC^{n_k}$ has \defin{completely mixed marginals} if for each $i \in [k]$ the marginal quantum entropy $H(\ketbra{\psi}{\psi}_i)$ equals the maximal value~$\log_2 n_i$.
Bryan et al.~in \cite{bryan2017existence} give a characterisation of the formats $(n_1, \ldots, n_k)$ for which a state with completely mixed marginals exists. (The earlier characterisation by Klyachko in \cite[Corollary 2.5.2.2]{klyachko2002coherent} appears to be incorrect.)

\begin{theorem}\label{quantumgeneric}
Let $\theta \in \prob([k])$ such that $\theta_i > 0$ for all $i$.
We have 
\[
F_\theta(n_1, \ldots, n_k) = n_1^{\theta(1)} \cdots n_k^{\theta(k)}
\]
if and only if $\CC^{n_1} \otimes \cdots \otimes \CC^{n_k}$ contains a pure quantum state with completely mixed marginals. %
\end{theorem}

\begin{proof}
Let $(n_1, \ldots, n_k)$ be a format such that $\CC^{n_1} \otimes \cdots \otimes \CC^{n_k}$ contains a pure quantum state with completely mixed marginals.
Let $U$ be the Zariski open set defined above. Let $t \in U$. Then the entanglement polytope $\Delta_t$ is the maximal entanglement polytope and thus the maximisation
\[
E_\theta(t) = E^\theta(t) = \max \bigl\{ \theta(1) H(r_1) + \cdots + \theta(k) H(r_k) \,\big|\, (r_1, \ldots, r_k) \in \Delta_t \}
\]
achieves the value $\theta(1)\log_2 n_1 + \cdots + \theta(k) \log_2 n_k$.
We conclude that $F_{\theta}(n_1, \ldots, n_k)$ equals the upper bound $n_1^{\smash{\theta(1)}} \cdots n_k^{\smash{\theta(k)}}$. On the other hand, if $\CC^{n_1} \otimes \cdots \otimes \CC^{n_k}$ does not contain a pure quantum state with completely mixed marginals, then $((\tfrac{1}{n_1},\ldots, \tfrac{1}{n_1}), \ldots, (\tfrac{1}{n_k}, \ldots, \tfrac{1}{n_k}))$ is not in the maximal entanglement polytope and thus $\max \bigl\{ \theta(1) H(r_1) + \cdots + \theta(k) H(r_k) \,\big|\, (r_1, \ldots, r_k) \in \Delta_t \}$ is strictly smaller than $\theta(1)\log_2 n_1 + \cdots + \theta(k) \log_2 n_k$ for $t \in U$ (and in fact for any $t$).
\end{proof}

\begin{remark}
In particular, $F_\theta(n_1, \ldots, n_k) = n_1^{\theta(1)} \cdots n_k^{\theta(k)}$ for any comfortable format $(n_1, \ldots, n_k)$. Indeed, if $(n_1, \ldots, n_k)$ is comfortable, then by definition there is a set $\Phi \subseteq [n_1] \times \cdots \times [n_k]$ and a probability distribution $P \in \prob(\Phi)$ such that $\Phi$ is an antichain and each marginal $P_i$ is uniform on $[n_i]$. Define the tensor
\[
t = \sum_{\alpha \in \Phi} \sqrt{P(\alpha)}\, e_{\alpha_1} \otimes \cdots \otimes e_{\alpha_k} \in \CC^{n_1} \otimes \cdots \otimes \CC^{n_k}.
\]
Since $\Phi$ is an antichain and thus free, $t$ has completely mixed marginals. Following the proof of \cref{quantumgeneric} the value of $F_\theta(n_1, \ldots, n_k)$ equals $n_1^{\smash{\theta(1)}} \cdots n_k^{\smash{\theta(k)}}$.
\end{remark}

\subsubsection{Hilbert--Mumford criterion and instability}

We have seen that generically for certain formats $E^\theta$ and $\rho^\theta$ are maximal.
We give a related application of the Hilbert--Mumford criterion.
Let $t \in \CC^{n_1} \otimes \cdots \otimes \CC^{n_k}$.
Let $\theta \in \prob([k])$. We know that 
\[
E^\theta(t) \leq \rho^\theta(t) \leq n_1^{\smash{\theta(1)}} \cdots n_k^{\smash{\theta(k)}}.
\]
Therefore, 
\[
E^\theta(t)=n_1^{\smash{\theta(1)}}\cdots n_k^{\smash{\theta(k)}} \implies \rho^\theta(t) = n_1^{\smash{\theta(1)}}\cdots n_k^{\smash{\theta(k)}}.
\]
We show the following converse.

\begin{theorem}\label{hm}
Let $t \in \CC^{n_1} \otimes \cdots \otimes \CC^{n_k}$.
Let $\theta \in \prob([k])$ such that $\theta_i > 0$ for all $i$.
If $E^\theta(t) < n_1^{\smash{\theta(1)}}\cdots n_k^{\smash{\theta(k)}}$\!, then $\rho^\theta(t) < n_1^{\smash{\theta(1)}}\cdots n_k^{\smash{\theta(k)}}$.
\end{theorem}

We will use terminology and results from \cite{MR765581}.  The point in the entanglement polytopes corresponding to quantum states that are locally maximally mixed will be called the \defin{origin} and denoted by $O$.

\begin{lemma}\label{withlemma}
$O \not \in \mu_G(\overline{G\cdot t})$ if and only if $t$ is $G$-unstable i.e.~$0\in \overline{G\cdot t}$.
\end{lemma}
\begin{proof}
According to \cite[Lemma 2.3]{MR765581}, if $t$ is $G$-semistable, then $\overline{G\cdot t}$ contains a closed orbit. This implies according to \cite[Theorem 2.2(ii)]{MR765581} and \cite[Definition 2.1(iii)]{MR765581} that $O \in \mu_G(\overline{G\cdot t})$. If $t$ is $G$-unstable, then all $G$-invariants vanish, which implies $O \not\in \mu_G(\overline{G\cdot t})$.
\end{proof}

\begin{proof}[\bfseries\upshape Proof of \cref{hm}]
Let $G = \SL(\CC^{n_1}) \times \cdots \times \SL(\CC^{n_k})$ act naturally on $V = \CC^{n_1} \otimes \cdots \otimes \CC^{n_k}$.
Let $t \in V$.
If $E^\theta(t) < n_1^{\smash{\theta(1)}}\cdots n_k^{\smash{\theta(k)}}$, then $O \not\in \Delta_t$. This means $t$ is $G$-unstable (\cref{withlemma}) %
i.e.~$0 \in\overline{G\cdot t}$. %
By the Hilbert--Mumford criterion \cite[Theorem 3.1]{MR765581} this implies that there exists a one-parameter subgroup $\lambda$ such that $t$ is $\lambda$-unstable. This implies that there exists a maximal torus $T \subseteq G$ such that $t$ is $T$-unstable.
 By \cref{withlemma} this implies that $O \not\in \mu_T(\overline{T \cdot t})$.
Let $C\in \bases(t)$ be a choice of bases compatible with $T$.
As shown in the proof of \cref{freetheorem},
\begin{equation}\label{freeeq}
M = \{(P_1, \ldots, P_k) \mid P \in \prob(\supp_C t) \} \subseteq \mu_T(\overline{T\cdot t}).
\end{equation}
Therefore $O \not \in M$. Noting that $H_\theta(\supp_C t)$ equals $n_1^{\smash{\theta(1)}}\cdots n_k^{\smash{\theta(k)}}$ if and only if $O \in M$, we see that $H_\theta(\supp_C t)< n_1^{\smash{\theta(1)}}\cdots n_k^{\smash{\theta(k)}}$.
Therefore we conclude $\rho^{\theta}(t) \leq \max_{P \in \supp_C t} H_\theta(P) < n_1^{\smash{\theta(1)}}\cdots n_k^{\smash{\theta(k)}}$.
\end{proof}

We prove a quantitative version of \cref{hm} in terms of instability, which moreover holds over all fields. Our result improves a result of Blasiak et al.~in \cite{MR3631613} when the tensor format $(n_1, \ldots, n_k)$ is nonuniform enough.

Instability is a standard notion in geometric invariant theory. Let $\FF$ be algebraically closed. Let $t \in \FF^{n_1} \otimes \cdots \otimes \FF^{n_k}$. Let $G = \SL_{n_1}(\FF) \times \cdots \times \SL_{n_k}(\FF)$. The tensor $t$ is called unstable if $0$ is in the orbit closure $\overline{G\cdot t}$, with the closure taken in the Zariski topology, and otherwise $t$ is called semistable.

We use the quantitative notion of instability for a tensor $t \in \FF^{n_1} \otimes \cdots \otimes \FF^{n_k}$ defined in \cite{MR3631613},
with $\FF$ an arbitrary field,
\begin{multline*}
\instab(t) = \max_{C \in \bases(t)} \max_{w_1, \ldots, w_k} \sup \Bigl\{ \eps \geq 0 \,\Big|\, \forall a \in \supp_C t :\\ \sum_{i=1}^k w_i(a_i) \leq \sum_{i=1}^k \Bigl( \frac{1}{\abs[0]{I_i}} \sum_{\mathclap{x \in I_i}} w_i(x) - \eps \max_x w_i(x) \Bigr) \Bigr\}
\end{multline*}
where the first maximum is over the choice of bases of $\FF^{n_1}, \ldots, \FF^{n_k}$, with index sets $I_1, \ldots, I_k$, and the second maximum is over weight functions $w_i : I_i \to \RR_{\geq 0}$ that are not identically zero. When $\FF$ is algebraically closed, the Hilbert--Mumford criterion says that $t$ is unstable if and only if $\instab(t) > 0$.

For probability distributions $P, Q \in \prob(X)$, define the total variation distance $V(P,Q) = \tfrac12 \sum_{x \in X} \abs[0]{P(x) - Q(x)}$ and define the Kullback-Leibler divergence $D(P||Q)= \sum_{x \in X} P(x) \log_2 (P(x)/Q(x))$. We will use Pinsker's inequality, which says
\begin{equation}\label{pinsker}
V(P,Q) \leq \sqrt{\tfrac{\ln 2}{2} D(P||Q)},
\end{equation}
see e.g.~\cite[Lemma 2.5]{MR2724359}.

\begin{theorem}
For $t \in \FF^{n_1} \otimes \cdots \otimes \FF^{n_k}$ and $\theta \in \prob([k])$,
\begin{equation}\label{instabeq}
\rho^{\theta}(t) \leq \sum_{i=1}^k \theta(i) \log_2 n_i - \bigl(\tfrac{2}{\ln2} \min_i \theta(i) \bigr) \instab(t)^2.
\end{equation}
\end{theorem}
\begin{proof}
The two sides of \eqref{instabeq} are continuous in $\theta$, so we may without loss of generality assume that $\theta(i) > 0$ for all $i \in [k]$. Choose bases $C \in \bases(t)$, weights $w_1, \ldots, w_k$, and $\eps \geq 0$, such that for every $a \in \supp_C t$
\[
\sum_{i=1}^k w_i(a_i) \leq \sum_{i=1}^k \Bigl( \frac{1}{\abs[0]{I_i}} \sum_{\mathclap{x \in I_i}} w_i(x) - \eps \max_x w_i(x)\Bigr)
\]
i.e.~for every $a \in \supp_C t$
\[
\eps \sum_{i=1}^k \max_x w_i(x) \leq \sum_{i=1}^k \frac{1}{\abs[0]{I_i}} \sum_{\mathclap{x \in I_i}} w_i(x) - \sum_{i=1}^k w_i(a_i).
\]
Let $P \in \prob(\supp_C t)$ be a probability distribution on $\supp_C t$. Then 
\begin{align*}
\eps \sum_{i=1}^k \max_x w_i(x) &\leq \sum_{i=1}^k \frac{1}{\abs[0]{I_i}} \sum_{\mathclap{x \in I_i}} w_i(x) - \sum_{\mathclap{a \in \supp_C t}} P(a) \sum_{i=1}^k w_i(a_i)\\
&= \sum_{i=1}^k \sum_{x \in I_i} \Bigl( \frac{1}{\abs[0]{I_i}} - P_i(x) \Bigr)\, w_i(x)\\
&\leq \tfrac12 \sum_{i=1}^k \sum_{x \in I_i} \Bigl| \frac{1}{\abs[0]{I_i}} - P_i(x) \Bigr|\, \max_x w_i(x).
\end{align*}
Let $U_i$ be the uniform probability distribution on $I_i$. We apply Pinsker's inequality \eqref{pinsker} to get
\begin{align*}
\tfrac12\sum_{i=1}^k \sum_{x\in I_i}\, {\abs[2]{\frac{1}{\abs[0]{I_i}} - P_i(x)}}\, \max_x w_i(x) &\leq  \sum_{i=1}^k \sqrt{\tfrac{\ln2}{2} D(P_i || U_i)} \max_x w_i(x)
\end{align*}
Next we apply the Cauchy-Schwarz inequality to obtain
\begin{align*}
&\hspace{-1em}\sum_{i=1}^k \sqrt{\tfrac{\ln2}{2} D(P_i || U_i)} \max_x w_i(x)\\ &\leq
\sqrt{\tfrac{\ln2}{2}} \Bigl(\sum_{i=1}^k D(P_i || U_i)\Bigr)^{1/2}  \Bigl(\sum_{i=1}^k \max_x w_i(x)^2 \Bigr)^{1/2}\\
&\leq
\sqrt{\tfrac{\ln2}{2}} \Bigl(\sum_{i=1}^k D(P_i || U_i)\Bigr)^{1/2}  \sum_{i=1}^k \max_x w_i(x)\\
&\leq
\sqrt{\tfrac{\ln2}{2} \tfrac{1}{\min_i \theta(i)} } \Bigl(\sum_{i=1}^k \theta(i) D(P_i || U_i)\Bigr)^{1/2}  \sum_{i=1}^k \max_x w_i(x)\\
&= \sqrt{\tfrac{\ln2}{2} \tfrac{1}{\min_i \theta(i)} } \Bigl(\sum_{i=1}^k \theta(i) \log_2 \abs[0]{I_i} - \theta(i) H(P_i) \Bigr)^{1/2}  \sum_{i=1}^k \max_x w_i(x).
\end{align*}
We conclude
\[
\eps^2 \leq \frac{\ln 2}{2} \frac{1}{\min_i \theta(i)} \Bigl( \sum_{i=1}^k \theta(i) \log_2 \abs[0]{I_i} - \theta(i)H(P_i) \Bigr)
\]
and thus
\[
\sum_{i=1}^k \theta(i) H(P_i) \leq \sum_{i=1}^k \theta(i) \log_2 \abs[0]{I_i} - \frac{2}{\ln 2} \min_i {\theta(i)}\, \eps^2.
\]
Now take the supremum over $P$, and the infimum over $\eps$, $w_i$ and $C$.
\end{proof}

\section{Subrank, slice rank and multi-slice rank}\label{slicesec}

This section is about two variations on tensor rank, namely slice rank, which was introduced by Tao \cite{tao}, and multi-slice rank, which was introduced by Naslund \cite{naslund2017multi}. %
We know that subrank is asymptotically upper bounded by the quantum functionals. It is known that subrank is at most slice rank and multi-slice rank.
The result of this section is that, remarkably, even slice rank and multi-slice rank are asymptotically upper bounded by the quantum functionals. In fact, we show that the asymptotic slice rank $\lim_{n\to \infty} \slicerank(t^{\otimes n})^{1/n}$ exists and equals the minimum $\min_{\theta \in \prob([k])}F^\theta(t)$.

\subsection{Definition} Recall that a simple tensor is a tensor of the form $v_1 \otimes \cdots \otimes v_k \in V_1 \otimes \cdots \otimes V_k$ with $v_i \in V_i$ for $i\in [k]$, and that the rank $\rank(t)$ of a tensor $t \in V_1 \otimes \cdots \otimes V_k$ is the smallest number $r$ such that $t$ can be written as a sum of $r$ simple tensors. 

\begin{definition}
A \defin{slice} is a tensor of the form $v\otimes w \in V_1 \otimes \cdots \otimes V_k$ where $v \in V_j$ and $w \in V_{\overline j}$ for some $j \in [k]$. A \defin{multi-slice} is a tensor of the form $v \otimes w \in V_1 \otimes \cdots \otimes V_k$ where $v \in V_S$ and $w \in V_{\overline S}$ for some subset $S \subseteq [k]$. Let $t \in V_1 \otimes \cdots \otimes V_k$. The \defin{slice rank} of $t$, denoted $\slicerank(t)$, is the smallest number~$r$ such that $t$ can be written as a sum of $r$ slices. The \defin{multi-slice rank} of~$t$, denoted $\multislicerank(t)$, is the smallest number $r$ such that $t$ can be written as a sum of $r$ multi-slices.
\end{definition}

\subsection{Key properties}
Slice rank and multi-slice rank are clearly $\leq$-monotones. The slice rank and multi-slice rank of $\langle r \rangle$ equals $r$ \cite{tao, naslund2017multi}. It follows that subrank is at most multi-slice rank which is at most slice rank,
\[
\subrank(t) \leq \multislicerank(t) \leq \slicerank(t).
\]
Computing upper bounds on subrank $\subrank(t)$ and asymptotic subrank $\asympsubrank(t)$ was the main motivation for introducing slice rank in \cite{tao}.

\begin{example}
While tensor rank is easily seen to be sub-multiplicative with respect to tensor products, and subrank to be super-multiplicative, slice rank and multi-slice rank are neither sub-multiplicative nor super-multiplicative. 
For example, the tensors $
\sum_{i=1}^n e_i \otimes e_i \otimes 1$, $\sum_{i=1}^n e_i \otimes 1 \otimes e_i$, $\sum_{i=1}^n 1 \otimes e_i \otimes e_i$
have slice rank one, while their tensor product equals the matrix multiplication tensor $\langle n,n,n \rangle$ which has slice rank $n^2$ (see \cite[Remark~4.6]{MR3631613}). This shows slice rank is not sub-multiplicative.
To see slice rank is not super-multiplicative, take for example $W$ to be the tensor ${e_1 \otimes e_1 \otimes e_2 + e_1 \otimes e_2 \otimes e_1 + e_2 \otimes e_1 \otimes e_1}$. The slice rank of $W$ equals two. The value of the logarithmic upper support functional $\rho^{(1/3, 1/3, 1/3)}(W)$ equals $h(1/3)\approx0.918296$ (see~\cref{dicke}). In \cref{slicerankzeta} we will show that for any $\theta \in \prob([k])$ the value of $\slicerank(t^{\otimes n})$ is at most $2^{\rho^\theta(t)n + o(n)}$ when $n\to \infty$. We thus obtain the upper bound $\slicerank(W^{\otimes n}) \leq  2^{h(1/3)n + o(n)}$, which proves the claim.
Now it is also clear that slice rank is not equal to subrank or border subrank in general, since subrank and border subrank are super-multiplicative.
\end{example}

Since slice rank and multi-slice rank are not sub-multiplicative and not super-multiplicative, the limit $\lim_{n\to \infty} \slicerank(t^{\otimes n})^{1/n}$ and the analogous limit for multi-slice rank might not exist.
We will show that $\lim_{n\to \infty} \slicerank(t^{\otimes n})^{1/n}$ in fact does exist.
For now define
\begin{align*}
\mathrm{SR}^{\sim} &= \limsup_{n\to \infty} \slicerank(t^{\otimes n})^{1/n}\\
\msr(t) &=  \limsup_{n\to \infty} \multislicerank(t^{\otimes n})^{1/n}.
\end{align*}
We have
\[
\asympsubrank(t) \leq \msr(t)
 \leq
\sr(t).
\]

\subsection{Upper bound by support and quantum functionals}
We show that slice rank and multi-slice rank are asymptotically upper bounded by the quantum functionals.

\begin{theorem}\label{slicerankF}
Let $\FF$ be the complex numbers $\CC$.
Let $t \in V_1 \otimes \cdots \otimes V_k$.\\
If $\theta \in \prob(B)$,
then
$\msr(t) 
\leq F^{\theta}(t)$. %
If $\theta \in \prob_\s(B)$, then
$\sr(t)
\leq F^{\theta}(t)$.
\end{theorem}
\begin{proof}
Let $c = E^{\theta}(t)$. %
Let $\ell = \abs[0]{\supp \theta}$.
Let $n \in \NN$. Let $t_0 = t^{\otimes n}$. Define a sequence of tensors $t_1, \ldots, t_\ell$ by 
\[
t_i = \Bigl(\id -\sum_{\mathclap{\substack{\lambda^{(i)} \vdash n:\\H(\overline{\lambda^{(i)}}) \leq c}}}P_{\lambda^{(i)}}^{V_{i}} \Bigr) \,t_{i-1}. %
\]
Let
\[
s_i = \sum_{\mathclap{\substack{\lambda^{(i)} \vdash n:\\H(\overline{\lambda^{(i)}}) \leq c}}}P_{\lambda^{(i)}}^{V_{i}} \,t_{i-1}.
\]
We can write $t_\ell$ in two ways. On the one hand,
\[
t_\ell = t^{\otimes n} - s_1 - s_2 - \cdots - s_{\ell}.
\]
On the other hand,
\begin{align*}
t_\ell &= \Bigl(\id - \sum_{\mathclap{\substack{\lambda^{(\ell)}\vdash n:\\H(\overline{\lambda^{(\ell)}}) \leq c}}} P_{\lambda^{(\ell)}}^{V_{\ell}}\Bigr) \cdots \Bigl(\id - \sum_{\mathclap{\substack{\lambda^{(1)}\vdash n:\\ H(\overline{\lambda^{(1)}}) \leq c}}} P_{\lambda^{(1)}}^{V_{1}} \Bigr)\, t^{\otimes n}\\
&= \;\sum_{\mathclap{\substack{\lambda^{(\ell)} \vdash n:\\ H(\overline{\lambda^{(\ell)}}) > c}}} \;\;\cdots\;\; \sum_{\mathclap{\substack{\lambda^{(1)} \vdash n:\\ H(\overline{\lambda^{(1)}}) > c}}}\; P_{\lambda^{(\ell)}}^{V_{\ell}} \cdots P_{\lambda^{(1)}}^{V_{1}}\, t^{\otimes n}
\end{align*}
which is 0 by definition of $E^{\theta}$.
Therefore $t^{\otimes n} = \sum_{i=1}^\ell s_i$. 

The multi-slice rank of an element in the image of $P_\lambda^{V_i}$ is at most $2^{\smash{nH(\overline{\lambda})}}$.
Each $s_i$ is in the image of $\sum P_\lambda^{\smash{V_{i}}}$ where the sum is over partitions $\lambda \vdash n$ with $H(\overline{\lambda}) \leq c$ and with at most $d_1 d_2\cdots d_k$ parts. There are at most $(n+1)^{d_1 d_2 \cdots d_k}$ such partitions.
Therefore $\multislicerank(s_i) \leq (n+1)^{d_1 d_2\cdots d_k} 2^{n c}$. This implies
\begin{align*}
\msr(t) \leq \limsup_{n \to \infty} \bigl( \ell (n+1)^{d_1 d_2 \cdots d_k}\, 2^{n c}\bigr)^{1/n} = F^{\theta}(t).
\end{align*}
This proves the first statement. The proof of the second statement is the same, except that it uses the upper bound $\slicerank(s_i) \leq (n+1)^{d_i} 2^{n E^\theta(t)}$. %
\end{proof}

Let us focus on the singleton regime $\theta \in \prob([k])$. We show that the asymptotic slice rank exists and equals the minimum value over $F^\theta$ with $\theta \in \prob([k])$.

\begin{lemma}\label{slice_lem1}
Let $t \in \CC^{d_1} \otimes \cdots \otimes \CC^{d_k}$ be nonzero. For any $\eps > 0$ there is an $n_0 \in \NN$ such that for all $n \geq n_0$ there are partitions $\lambda^{(1)}, \ldots, \lambda^{(k)} \vdash n$ satisfying $(P_{\lambda^{(1)}} \otimes \cdots \otimes P_{\lambda^{(k)}}) t^{\otimes n} \neq 0$ and
\[
\min \{ H(\tfrac1n \lambda^{(1)}), \ldots, H(\tfrac1n \lambda^{(k)}) \} \geq \min_{\theta \in \prob([k])} E^{\theta}(t) - \eps.
\]
\end{lemma}
\begin{proof}
Let $\Delta_t$ be the moment polytope of $t$. By the minimax theorem
\begin{align*}
\min_{\theta \in \prob([k])} E^{\theta}(t) &= \min_{\theta \in \prob([k])} \max_{(r_1, \ldots, r_k) \in \Delta_t} \sum_{j=1}^k \theta(j) H(r_j)\\
&=  \max_{(r_1, \ldots, r_k) \in \Delta_t}  \min_{\theta \in \prob([k])} \sum_{j=1}^k \theta(j) H(r_j)\\
&= \max_{(r_1, \ldots, r_k) \in \Delta_t} \min_{j \in [k]} H(r_j).
\end{align*}
Let $\mu^{(1)}, \ldots, \mu^{(k)} \vdash m$ be partitions such that
\[
\min_{j \in [k]} H(\tfrac1m \mu^{(j)}) \geq \min_{\theta \in \prob([k])} E^\theta(t) - \eps/2
\]
and $(P_{\mu^{(1)}} \otimes \cdots \otimes P_{\mu^{(k)}}) t^{\otimes m} \neq 0$. This is possible since $\Delta_t$ is the closure of such rescaled partitions and the entropy is continuous. We use that $(P_{(1)} \otimes \cdots \otimes P_{(1)}) t = t \neq 0$ and that the tuples of partitions $\lambda^{(1)}, \ldots, \lambda^{(k)}$ satisfying $(P_{\lambda^{(1)}} \otimes \cdots \otimes P_{\lambda^{(k)}}) t^{\otimes n} \neq 0$ form a semigroup.

Any $n \in \NN$ can be written as $n = mq + r$ where $q,r \in \NN$, $0\leq r < m$. Let $\lambda^{(j)} = q \mu^{(j)} + (r)$.
By the semigroup property, $(P_{\lambda^{(1)}} \otimes \cdots \otimes P_{\lambda^{(k)}}) t^{\otimes n} \neq 0$. The entropies can be estimated using the concavity of entropy as
\begin{align*}
H\bigl( \tfrac1n (q\mu^{(j)} + (r)) \bigr) &= H\bigl( \tfrac{qm}{n} \tfrac1m \mu^{(j)} + \tfrac{r}{n} \tfrac{1}{r} (r) \bigr) \geq \tfrac{qm}{n} H\bigl( \tfrac1m \mu^{(j)} \bigr)\\
&= \tfrac{m}{n} \floor{\tfrac{n}{m}} H\bigl( \tfrac1m \mu^{(j)} \bigr) \geq (1 - \tfrac{m}{n}) H\bigl( \tfrac{1}{m}\mu^{(j)} \bigr).
\end{align*}
When $n$ is large enough, this is greater than $H(\tfrac1m \mu^{(j)}) - \eps/2$. Choose $n_0$ such that this is true for all $j \in [k]$.
\end{proof}

\begin{lemma}\label{slice_lem2}
Let $t \in \CC^{d_1} \otimes \cdots \otimes \CC^{d_k}$ and suppose that $(P_{\lambda^{(1)}} \otimes \cdots \otimes P_{\lambda^{(k)}}) t^{\otimes n} \neq 0$ for some partitions $\lambda^{(1)}, \ldots, \lambda^{(k)} \vdash n$. Then
\[
\slicerank(t^{\otimes n}) \geq \min \{ \dim [\lambda_1], \ldots, \dim [\lambda_k]\}.
\]
\end{lemma}
\begin{proof}
$t^{\otimes n}$ restricts to $(P_{\lambda^{(1)}} \otimes \cdots \otimes P_{\lambda^{(k)}}) t^{\otimes n}$, which by assumption is nonzero. Choose rank one projections $A_j$ in the vector spaces $\SSS_{\lambda^{(j)}}(\CC^{d_j})$ such that
\[
\bigl( (\id_{[\lambda^{(1)}]} \otimes A_1) \otimes \cdots \otimes (\id_{[\lambda^{(k)}]} \otimes A_k) \bigr) ( P_{\lambda^{(1)}} \otimes \cdots \otimes P_{\lambda^{(k)}}) t^{\otimes n} \neq 0.
\]
The resulting tensor is invariant under the diagonal $S_n$-action, therefore the marginals are maximally mixed. In the GIT language, this means they are unstable, therefore the slice rank is the smallest local dimension by \cite[Theorem 4.6]{MR3631613}.
\end{proof}

\begin{theorem}
For any tensor $t \in \CC^{d_1} \otimes \cdots \otimes \CC^{d_k}$,
\[
\liminf_{n \to \infty} \slicerank(t^{\otimes n})^{1/n} \geq \min_{\theta \in \prob([k])} 2^{E^\theta(t)}.
\]
\end{theorem}
\begin{proof}
Choose $\eps > 0$ and for all large enough $n$ choose $\lambda^{(1)}, \ldots, \lambda^{(k)}\vdash n$ as in \cref{slice_lem1}. By \cref{slice_lem2},
\begin{align*}
\slicerank(t^{\otimes n}) &\geq \min\{ \dim [\lambda^{(1)}], \ldots, \dim [\lambda^{(k)}] \}\\
&\geq \min \{2^{n H(\tfrac1n \lambda^{(1)})}, \ldots, 2^{n H(\tfrac1n \lambda^{(k)})}\} 2^{-o(n)}\\
&\geq 2^{n(\min_{\theta \in \prob([k])} E^\theta(t) - \eps)} 2^{-o(n)},
\end{align*}
therefore
\[
\liminf_{n\to\infty} \slicerank(t^{\otimes n})^{1/n} \geq 2^{\min_{\theta \in \prob([k])} E^\theta(t) - \eps }.
\]
This is true for any $\eps > 0$, therefore
\[
\liminf_{n \to\infty} \slicerank(t^{\otimes n})^{1/n} \geq 2^{\min_{\theta \in \prob([k])} E^\theta(t)}.
\]
\end{proof}

\begin{corollary}\label{minpoint}
For complex tensors, asymptotic slice rank exists and equals the minimum value of the support functionals over $\theta \in \prob([k])$,
\[
\regularize{\mathrm{SR}}(t)\coloneqq \lim_{n \to \infty} \slicerank(t^{\otimes n})^{1/n} = \min_{\theta \in \prob([k])} 2^{E^\theta(t)}.
\]
\end{corollary}

\begin{remark}
Connections between asymptotic slice rank and moment polytopes have also been observed in \cite[Corollary 6.5]{burgisser2017alternating} and \cite{CGN}.
\end{remark}
For tensors over arbitrary fields we have the following result. (See also \cite{sawin} for a similar statement.)

\begin{theorem}\label{slicerankzeta}
Let $t \in V_1 \otimes \cdots \otimes V_k$. Let $\theta \in \prob([k])$. Then
\[
\sr(t) \leq \zeta^{\theta}(t).
\]
\end{theorem}

\begin{proof}
Choose a $k$-tuple of bases $C = ((v_{1, 1}, \ldots, v_{1, d_1}), \ldots, (v_{k, 1}, \ldots, v_{k, d_k}))$ in~$\bases(f)$. Let $n \in \NN$. As in the proof of \cref{suppinv}, given a distribution $Q \in \prob_n([d_1] \times \cdots \times [d_k])$, define the vector $v_Q \in V_1^{\otimes n} \otimes \cdots \otimes V_k^{\otimes n}$ as
\[
v_Q \coloneqq \!\sum_{x} \bigotimes_{m=1}^n v_{1, x_{m,1}} \otimes v_{2, x_{m,2}} \otimes \cdots \otimes v_{k, x_{m,k}}
\]
where the sum is over all $n$-tuples $x = (x_1, \ldots, x_n) \in T_Q^n$,
and define the element $t_Q \in \CC$ as $t_Q \coloneqq \prod_{i \in \supp Q}\; (t_{i_1, \ldots, i_k})^{nQ(i)}$. 
Then the tensor power $t^{\otimes n}$ can be written as
\begin{equation}\label{decomp}
t^{\otimes n} = \sum_{Q} t_Q \, v_Q
\end{equation}
where the sum is over $Q \in \prob_n(\supp_C t)$.
Let $W_1 \otimes \cdots \otimes W_k \subseteq V_1^{\otimes n} \otimes \cdots \otimes V_k^{\otimes n}$ be the subspace where $W_j$ is spanned by all vectors $v_{j,i_1} \otimes v_{j, i_2} \otimes \cdots \otimes v_{j, i_k}$ such that the empirical distribution of $(i_1, \ldots, i_k) \in [d_j]^n$ is equal to the marginal distribution $Q_j$. Then $v_Q \in W_1 \otimes \cdots \otimes W_k$. This implies that
\[
\slicerank(v_Q) \leq \min_{j \in [k]} 2^{n H(Q_j)} \leq 2^{n \sum_{j=1}^k \theta(j) H(Q_j)}.
\]
The number of terms in \eqref{decomp} is at most $(n+1)^{\abs[0]{\supp_C t}}$. Therefore,
\begin{align*}
\log_2 \sr(t)
&\leq \limsup_{n\to \infty} \frac1n \Bigl( {\abs[0]{\supp_C t}}\, \log_2 (n+1) + n \max_{P} \sum_{\smash{j=1}}^{\smash{k}} \theta(j) H(P_j)  \Bigr)\\
&= \max_{P} \sum_{j=1}^{\smash{k}} \theta(j) H(P_j)
\end{align*}
with both maximisations over $P \in \prob(\supp_C t)$.
This inequality holds for any basis choice $C\in \bases(f)$ and therefore we may minimise the right side over all basis choices. This proves the theorem.
\end{proof}

\begin{corollary}\label{slicetight}
Let $t$ be a tight 3-tensor. Then $\regularize{\mathrm{SR}}(t) = \lim_{n\to\infty} \slicerank(t^{\otimes n})^{1/n}$ exists and equals the asymptotic subrank $\asympsubrank(t) = \lim_{n\to \infty} \subrank(t^{\otimes n})^{1/n}$. 
\end{corollary}
\begin{proof}
We have
\[
\asympsubrank(t) \leq \sr(t) \leq \min_\theta \zeta^\theta(t)
\]
and $\min_\theta \zeta^\theta(t)$ equals $\asympsubrank(t)$ since $t$ is a tight 3-tensor (\cref{cwcor}).
\end{proof}

\begin{remark}
We summarise the relationships among the functionals and asymptotic subrank, rank and slice-rank. 
\begin{enumerate}
\item $\regularize{\mathrm{SR}}(t) \coloneqq \lim_{n\to\infty} \slicerank(t^{\otimes n})^{1/n} = \min_{\theta \in \prob_\s(B)}F^\theta(t)$\,\quad (\cref{minpoint})
\item $\asympsubrank(t) \leq \msr(t) \leq  \regularize{\mathrm{SR}}(t)  \leq F^\theta(t) \leq \zeta^\theta(t)$\, if $\theta \in \prob_\s(B)$\\[1ex] (\cref{slicerankzeta,slicerankF})
\item $\asympsubrank(t) = \msr(t) =  \regularize{\mathrm{SR}}(t) = \min_{\theta \in \prob_\s(B)}F^\theta(t) = \min_{\theta \in \prob_\s(B)} \zeta^\theta(t)$\\[1ex] if $t$ tight order-3 (\cref{slicetight})
\item $\asympsubrank(t) \leq \msr(t) \leq F^\theta(t)$\, if $\theta \in \prob(B)$\,\quad (\cref{slicerankF})
\item $\zeta_\theta(t) \leq \asymprank(t)$\, if $\theta \in \prob_\s(B)$
\item $F_\theta(t)\leq \asymprank(t)$\, if $\theta \in \prob_\nc(B)$.
\end{enumerate}
\end{remark}

\vspace{2em}
\parag{Acknowledgements} 
The authors thank the members of the QMATH Tensor reading group and Peter Bürgisser for much valuable discussion.
Part of this work was carried out while Jeroen Zuiddam and Péter Vrana were visiting QMATH.
We acknowledge financial support from the European Research Council (ERC Grant Agreement no.~337603), the Danish Council for Independent Research (Sapere Aude), and VILLUM FONDEN via the QMATH Centre of Excellence (Grant no.~10059). JZ is supported by~NWO (617.023.116).

\raggedright
\addcontentsline{toc}{section}{References}
\bibliographystyle{alphaurlpp}
\bibliography{all}

\newcommand{\etalchar}[1]{$^{#1}$}
\begin{thebibliography}{VDDMV02}

\bibitem[AS81]{ALDER1981201}
Alexander Alder and Volker Strassen.
\newblock \href
  {http://dx.doi.org/https://doi.org/10.1016/0304-3975(81)90070-0} {On the
  algorithmic complexity of associative algebras}.
\newblock {\em Theoret. Comput.~Sci.}, 15(2):201 -- 211, 1981.

\bibitem[ASU13]{Alon2013}
Noga Alon, Amir Shpilka, and Christopher Umans.
\newblock \href {http://dx.doi.org/10.1007/s00037-013-0060-1} {On sunflowers
  and matrix multiplication}.
\newblock {\em Comput. Complexity}, 22(2):219--243, Jun 2013.

\bibitem[AVZ]{srini}
Srinivasan Arunachalam, P{\'e}ter Vrana, and Jeroen Zuiddam.
\newblock Optimal distillation from balanced type tensors.
\newblock Manuscript.

\bibitem[BCC{\etalchar{+}}17]{MR3631613}
Jonah Blasiak, Thomas Church, Henry Cohn, Joshua~A. Grochow, Eric Naslund,
  William~F. Sawin, and Chris Umans.
\newblock \href {http://dx.doi.org/10.19086/da.1245} {On cap sets and the
  group-theoretic approach to matrix multiplication}.
\newblock {\em Discrete Anal.}, 2017.
\newblock \href {http://arxiv.org/abs/1605.06702} {\path{arXiv:1605.06702}}.

\bibitem[BCS97]{burgisser1997algebraic}
Peter B{\"u}rgisser, Michael Clausen, and M.~Amin Shokrollahi.
\newblock \href {http://dx.doi.org/10.1007/978-3-662-03338-8} {{\em Algebraic
  complexity theory}}, volume 315 of {\em Grundlehren der Mathematischen
  Wissenschaften}.
\newblock Springer-Verlag, Berlin, 1997.

\bibitem[BCSX10]{bhattacharyya2010testing}
Arnab Bhattacharyya, Victor Chen, Madhu Sudan, and Ning Xie.
\newblock \href {http://dx.doi.org/10.1007/978-3-642-16367-8_18} {{\em Testing
  Linear-Invariant Non-linear Properties: A Short Report}}, pages 260--268.
\newblock Springer Berlin Heidelberg, Berlin, Heidelberg, 2010.

\bibitem[BGO{\etalchar{+}}17]{burgisser2017alternating}
Peter B{\"u}rgisser, Ankit Garg, Rafael Oliveira, Michael Walter, and Avi
  Wigderson.
\newblock \href {} {Alternating minimization, scaling algorithms, and the
  null-cone problem from invariant theory}.
\newblock {\em arXiv}, 2017.
\newblock \href {http://arxiv.org/abs/1711.08039} {\path{arXiv:1711.08039}}.

\bibitem[BI11]{Burgisser:2011:GCT:1993636.1993704}
Peter B\"{u}rgisser and Christian Ikenmeyer.
\newblock \href {http://dx.doi.org/10.1145/1993636.1993704} {Geometric
  Complexity Theory and Tensor Rank}.
\newblock In {\em Proceedings of the Forty-third Annual ACM Symposium on Theory
  of Computing}, STOC '11, pages 509--518, New York, NY, USA, 2011. ACM.

\bibitem[BL16]{blser_et_al:LIPIcs:2016:6434}
Markus Bl{\"a}ser and Vladimir Lysikov.
\newblock \href {http://dx.doi.org/10.4230/LIPIcs.MFCS.2016.19} {{On
  Degeneration of Tensors and Algebras}}.
\newblock In {\em 41st International Symposium on Mathematical Foundations of
  Computer Science (MFCS 2016)}, pages 19:1--19:11, 2016.
\newblock \href {http://arxiv.org/abs/1606.04253} {\path{arXiv:1606.04253}}.

\bibitem[Bl{\"a}01]{blaser20015}
Markus Bl{\"a}ser.
\newblock \href {http://dx.doi.org/10.1007/3-540-44693-1_9} {A 5/2 $n^2$-Lower
  Bound for the Multiplicative Complexity of $n\times n$-Matrix
  Multiplication}.
\newblock {\em STACS 2001}, pages 99--109, 2001.

\bibitem[Bl{\"a}13]{blaser2013fast}
Markus Bl{\"a}ser.
\newblock \href {http://dx.doi.org/10.4086/toc.gs.2013.005} {{\em Fast Matrix
  Multiplication}}.
\newblock Number~5 in Graduate Surveys. Theory of Computing Library, 2013.

\bibitem[Bor91]{MR1102012}
Armand Borel.
\newblock \href {http://dx.doi.org/10.1007/978-1-4612-0941-6} {{\em Linear
  algebraic groups}}, volume 126 of {\em Graduate Texts in Mathematics}.
\newblock Springer-Verlag, New York, second edition, 1991.

\bibitem[BPR{\etalchar{+}}00]{bennett2000exact}
Charles~H. Bennett, Sandu Popescu, Daniel Rohrlich, John~A. Smolin, and
  Ashish~V. Thapliyal.
\newblock \href {http://dx.doi.org/10.1103/PhysRevA.63.012307} {Exact and
  asymptotic measures of multipartite pure-state entanglement}.
\newblock {\em Phys. Rev.~A}, 63(1):012307, 2000.

\bibitem[Bri87]{MR932055}
Michel Brion.
\newblock \href {http://dx.doi.org/10.1007/BFb0078526} {Sur l'image de
  l'application moment}.
\newblock In {\em S\'eminaire d'alg\`ebre {P}aul {D}ubreil et {M}arie-{P}aule
  {M}alliavin ({P}aris, 1986)}, volume 1296 of {\em Lecture Notes in Math.},
  pages 177--192. Springer, Berlin, 1987.

\bibitem[Bri05]{MR2143072}
Michel Brion.
\newblock \href {http://dx.doi.org/10.1007/3-7643-7342-3_2} {Lectures on the
  geometry of flag varieties}.
\newblock In {\em Topics in cohomological studies of algebraic varieties},
  Trends Math., pages 33--85. Birkh\"auser, Basel, 2005.
\newblock \href {http://arxiv.org/abs/math/0410240}
  {\path{arXiv:math/0410240}}.

\bibitem[BRVR17]{bryan2017existence}
Jim Bryan, Zinovy Reichstein, and Mark Van~Raamsdonk.
\newblock \href {} {Existence of locally maximally entangled quantum states via
  geometric invariant theory}.
\newblock {\em arXiv}, 2017.
\newblock \href {http://arxiv.org/abs/1708.01645} {\path{arXiv:1708.01645}}.

\bibitem[BS83]{MR707730}
Eberhard Becker and Niels~and Schwartz.
\newblock \href {http://dx.doi.org/10.1007/BF01192806} {Zum {D}arstellungssatz
  von {K}adison-{D}ubois}.
\newblock {\em Arch. Math. (Basel)}, 40(5):421--428, 1983.

\bibitem[BX15]{Bhattacharyya2015}
Arnab Bhattacharyya and Ning Xie.
\newblock \href {http://dx.doi.org/10.1007/s00037-014-0092-1} {Lower bounds for
  testing triangle-freeness in Boolean functions}.
\newblock {\em Comput. Complexity}, 24(1):65--101, Mar 2015.

\bibitem[BZ06]{MR2230995}
Ingemar Bengtsson and Karol \.Zyczkowski.
\newblock \href {http://dx.doi.org/10.1017/CBO9780511535048} {{\em Geometry of
  quantum states}}.
\newblock Cambridge University Press, Cambridge, 2006.
\newblock An introduction to quantum entanglement.

\bibitem[Bü90]{burg}
Peter Bürgisser.
\newblock {\em Degenerationsordnung und Trägerfunktional bilinearer
  Abbildungen}.
\newblock PhD thesis, Universität Konstanz, 1990.
\newblock \url{http://nbn-resolving.de/urn:nbn:de:bsz:352-opus-20311}.

\bibitem[CGN{\etalchar{+}}]{CGN}
Henry Cohn, Ankit Garg, Eric Naslund, Rafael Oliveira, and William~F. Sawin.
\newblock Personal communication.

\bibitem[CHM07]{MR2276458}
Matthias Christandl, Aram~W. Harrow, and Graeme Mitchison.
\newblock \href {http://dx.doi.org/10.1007/s00220-006-0157-3} {Nonzero
  {K}ronecker coefficients and what they tell us about spectra}.
\newblock {\em Comm. Math. Phys.}, 270(3):575--585, 2007.

\bibitem[CKSU05]{cohn2005group}
Henry Cohn, Robert Kleinberg, Balazs Szegedy, and Christopher Umans.
\newblock \href {http://dx.doi.org/10.1109/SFCS.2005.39} {Group-theoretic
  algorithms for matrix multiplication}.
\newblock In {\em Foundations of Computer Science, 2005. FOCS 2005. 46th Annual
  IEEE Symposium on}, pages 379--388. IEEE, 2005.

\bibitem[CLP17]{MR3583357}
Ernie Croot, Vsevolod~F. Lev, and P\'eter~P\'al Pach.
\newblock \href {http://dx.doi.org/10.4007/annals.2017.185.1.7}
  {Progression-free sets in {$\Bbb Z^n_4$} are exponentially small}.
\newblock {\em Ann. of Math. (2)}, 185(1):331--337, 2017.

\bibitem[CM06]{MR2197548}
Matthias Christandl and Graeme Mitchison.
\newblock \href {http://dx.doi.org/10.1007/s00220-005-1435-1} {The spectra of
  quantum states and the {K}ronecker coefficients of the symmetric group}.
\newblock {\em Comm. Math. Phys.}, 261(3):789--797, 2006.

\bibitem[CU03]{cohn2003group}
Henry Cohn and Christopher Umans.
\newblock \href {http://dx.doi.org/10.1109/SFCS.2003.1238217} {A
  group-theoretic approach to fast matrix multiplication}.
\newblock In {\em Foundations of Computer Science, 2003. Proceedings. 44th
  Annual IEEE Symposium on}, pages 438--449. IEEE, 2003.

\bibitem[CVZ16]{christandl2016asymptotic}
Matthias Christandl, P{\'e}ter Vrana, and Jeroen Zuiddam.
\newblock \href {https://arxiv.org/abs/1609.07476} {Asymptotic tensor rank of
  graph tensors: beyond matrix multiplication}.
\newblock {\em arXiv}, 2016.
\newblock \href {http://arxiv.org/abs/1609.07476} {\path{arXiv:1609.07476}}.

\bibitem[CW90]{MR1056627}
Don Coppersmith and Shmuel Winograd.
\newblock \href {http://dx.doi.org/10.1016/S0747-7171(08)80013-2} {Matrix
  multiplication via arithmetic progressions}.
\newblock {\em J. Symbolic Comput.}, 9(3):251--280, 1990.

\bibitem[DVC00]{MR1804183}
Wolfgang D\"ur, Guivre Vidal, and Juan~Ignacio Cirac.
\newblock \href {http://dx.doi.org/10.1103/PhysRevA.62.062314} {Three qubits
  can be entangled in two inequivalent ways}.
\newblock {\em Phys. Rev. A (3)}, 62(6):062314, 12, 2000.

\bibitem[Ede04]{MR2031694}
Yves Edel.
\newblock \href {http://dx.doi.org/10.1023/A:1027365901231} {Extensions of
  generalized product caps}.
\newblock {\em Des. Codes Cryptogr.}, 31(1):5--14, 2004.

\bibitem[EG17]{MR3583358}
Jordan~S. Ellenberg and Dion Gijswijt.
\newblock \href {http://dx.doi.org/10.4007/annals.2017.185.1.8} {On large
  subsets of {$\Bbb F^n_q$} with no three-term arithmetic progression}.
\newblock {\em Ann. of Math. (2)}, 185(1):339--343, 2017.

\bibitem[FK14]{fu_et_al:LIPIcs:2014:4730}
Hu~Fu and Robert Kleinberg.
\newblock \href {http://dx.doi.org/10.4230/LIPIcs.APPROX-RANDOM.2014.669}
  {{Improved Lower Bounds for Testing Triangle-freeness in Boolean Functions
  via Fast Matrix Multiplication}}.
\newblock In {\em Approximation, Randomization, and Combinatorial Optimization.
  Algorithms and Techniques (APPROX/RANDOM 2014)}, pages 669--676, 2014.

\bibitem[Fra02]{MR1923785}
Matthias Franz.
\newblock \href {http://emis.ams.org/journals/JLT/vol.12_no.2/16.html} {Moment
  polytopes of projective {$G$}-varieties and tensor products of symmetric
  group representations}.
\newblock {\em J.~Lie Theory}, 12(2):539--549, 2002.

\bibitem[HHHH09]{MR2515619}
Ryszard Horodecki, Pawe\l{} Horodecki, Micha\l{} Horodecki, and Karol
  Horodecki.
\newblock \href {http://dx.doi.org/10.1103/RevModPhys.81.865} {Quantum
  entanglement}.
\newblock {\em Rev. Modern Phys.}, 81(2):865--942, 2009.

\bibitem[HM02]{MR1955142}
Masahito Hayashi and Keiji Matsumoto.
\newblock \href {http://dx.doi.org/10.1103/PhysRevA.66.022311} {Quantum
  universal variable-length source coding}.
\newblock {\em Phys. Rev. A (3)}, 66(2):022311, 13, 2002.
\newblock \href {http://arxiv.org/abs/quant-ph/0209124}
  {\path{arXiv:quant-ph/0209124}}.

\bibitem[HX15]{haviv}
Ishay Haviv and Ning Xie.
\newblock \href {http://dx.doi.org/10.1145/2688073.2688084} {Sunflowers and
  testing triangle-freeness of functions}.
\newblock In {\em I{TCS}'15---{P}roceedings of the 6th {I}nnovations in
  {T}heoretical {C}omputer {S}cience}, pages 357--366. ACM, New York, 2015.
\newblock \href {http://arxiv.org/abs/1411.4692} {\path{arXiv:1411.4692}}.

\bibitem[HX17]{Haviv2017}
Ishay Haviv and Ning Xie.
\newblock \href {http://dx.doi.org/10.1007/s00037-016-0138-7} {Sunflowers and
  Testing Triangle-Freeness of Functions}.
\newblock {\em Comput. Complexity}, 26(2):497--530, Jun 2017.

\bibitem[Kly02]{klyachko2002coherent}
Alexander Klyachko.
\newblock \href {} {Coherent states, entanglement, and geometric invariant
  theory}.
\newblock {\em arXiv}, 2002.
\newblock \href {http://arxiv.org/abs/quant-ph/0206012}
  {\path{arXiv:quant-ph/0206012}}.

\bibitem[Kra84]{kraft1984geometrische}
Hanspeter Kraft.
\newblock \href {http://dx.doi.org/10.1007/978-3-663-10143-7} {{\em
  Geometrische {M}ethoden in der {I}nvariantentheorie}}.
\newblock Springer, 1984.

\bibitem[KS08]{Kaufman:2008:APT:1374376.1374434}
Tali Kaufman and Madhu Sudan.
\newblock \href {http://dx.doi.org/10.1145/1374376.1374434} {Algebraic Property
  Testing: The Role of Invariance}.
\newblock In {\em Proceedings of the Fortieth Annual ACM Symposium on Theory of
  Computing}, STOC '08, pages 403--412, New York, NY, USA, 2008. ACM.

\bibitem[KSS16]{kleinberg2016growth}
Robert Kleinberg, William~F. Sawin, and David~E. Speyer.
\newblock \href {} {The growth rate of tri-colored sum-free sets}.
\newblock {\em arXiv}, 2016.
\newblock \href {http://arxiv.org/abs/1607.00047} {\path{arXiv:1607.00047}}.

\bibitem[KW01]{MR1878924}
Michael Keyl and Reinhard~F. Werner.
\newblock \href {http://dx.doi.org/10.1103/PhysRevA.64.052311} {Estimating the
  spectrum of a density operator}.
\newblock {\em Phys. Rev. A (3)}, 64(5):052311, 5, 2001.
\newblock \href {http://arxiv.org/abs/quant-ph/0102027v1}
  {\path{arXiv:quant-ph/0102027v1}}.

\bibitem[Lan12]{landsberg2012tensors}
Joseph~M. Landsberg.
\newblock {\em Tensors: geometry and applications}, volume 128 of {\em Graduate
  Studies in Mathematics}.
\newblock American Mathematical Society, Providence, RI, 2012.

\bibitem[Lan14]{landsberg2014new}
Joseph~M. Landsberg.
\newblock \href {http://dx.doi.org/10.1137/120880276} {New lower bounds for the
  rank of matrix multiplication}.
\newblock {\em SIAM J.~Comput.}, 43(1):144--149, 2014.
\newblock \href {http://arxiv.org/abs/1206.1530v2} {\path{arXiv:1206.1530v2}}.

\bibitem[Lan17]{landsberg2017}
Joseph~M. Landsberg.
\newblock {\em Geometry and Complexity Theory}.
\newblock Cambridge Studies in Advanced Mathematics. Cambridge University
  Press, 2017.

\bibitem[LG14]{le2014powers}
Fran{\c{c}}ois Le~Gall.
\newblock \href {http://dx.doi.org/10.1145/2608628.2608664} {Powers of tensors
  and fast matrix multiplication}.
\newblock In {\em I{SSAC} 2014---{P}roceedings of the 39th {I}nternational
  {S}ymposium on {S}ymbolic and {A}lgebraic {C}omputation}, pages 296--303.
  ACM, New York, 2014.

\bibitem[Mau98]{mauch}
Franz Mauch.
\newblock {\em Ein Randverteilungsproblem und seine Anwendung auf das
  asymptotische Spektrum bilinearer Abbildungen}.
\newblock PhD thesis, Universität Konstanz, 1998.
\newblock \url{http://nbn-resolving.de/urn:nbn:de:bsz:352-opus-20543}.

\bibitem[Nas17]{naslund2017multi}
Eric Naslund.
\newblock \href {} {The multi-slice rank method and polynomial bounds for
  orthogonal systems in $\mathbb{F}_{q}^{n}$}.
\newblock {\em arXiv}, 2017.
\newblock \href {http://arxiv.org/abs/1701.04475} {\path{arXiv:1701.04475}}.

\bibitem[Nes84]{MR765581}
Linda Ness.
\newblock \href {http://dx.doi.org/10.2307/2374395} {A stratification of the
  null cone via the moment map}.
\newblock {\em Amer. J. Math.}, 106(6):1281--1329, 1984.
\newblock With an appendix by David Mumford.

\bibitem[Nor16]{norin2016distribution}
Sergey Norin.
\newblock \href {} {A distribution on triples with maximum entropy marginal}.
\newblock {\em arXiv}, 2016.
\newblock \href {http://arxiv.org/abs/1608.00243} {\path{arXiv:1608.00243}}.

\bibitem[ON02]{ogawa2002new}
Tomohiro Ogawa and Hiroshi Nagaoka.
\newblock \href {http://dx.doi.org/https://doi.org/10.1109/ISIT.2002.1023345}
  {A new proof of the channel coding theorem via hypothesis testing in quantum
  information theory}.
\newblock In {\em Information Theory, 2002. Proceedings. 2002 IEEE
  International Symposium on}, page~73. IEEE, 2002.
\newblock \href {http://arxiv.org/abs/quant-ph/0208139}
  {\path{arXiv:quant-ph/0208139}}.

\bibitem[Peb16]{pebody2016proof}
Luke Pebody.
\newblock \href {} {Proof of a Conjecture of Kleinberg-Sawin-Speyer}.
\newblock {\em arXiv}, 2016.
\newblock \href {http://arxiv.org/abs/1608.05740} {\path{arXiv:1608.05740}}.

\bibitem[Sha09]{Shapira:2009:GCT:1536414.1536438}
Asaf Shapira.
\newblock \href {http://dx.doi.org/10.1145/1536414.1536438} {Green's Conjecture
  and Testing Linear-invariant Properties}.
\newblock In {\em Proceedings of the Forty-first Annual ACM Symposium on Theory
  of Computing}, STOC '09, pages 159--166, New York, NY, USA, 2009. ACM.

\bibitem[Sja98]{MR1645052}
Reyer Sjamaar.
\newblock \href {http://dx.doi.org/10.1006/aima.1998.1739} {Convexity
  properties of the moment mapping re-examined}.
\newblock {\em Adv. Math.}, 138(1):46--91, 1998.

\bibitem[Smi04]{MR2193441}
A.~V. Smirnov.
\newblock \href {http://dx.doi.org/10.1090/S0077-1554-04-00143-8}
  {Decomposition of symmetric powers of irreducible representations of
  semisimple {L}ie algebras, and the {B}rion polytope}.
\newblock {\em Tr. Mosk. Mat. Obs.}, 65:230--252, 2004.

\bibitem[SOK14]{MR3195184}
Adam Sawicki, Micha\l{} Oszmaniec, and Marek Ku\'s.
\newblock \href {http://dx.doi.org/10.1142/S0129055X14500044} {Convexity of
  momentum map, {M}orse index, and quantum entanglement}.
\newblock {\em Rev. Math. Phys.}, 26(3):1450004, 39, 2014.

\bibitem[Sto10]{stothers2010complexity}
Andrew~James Stothers.
\newblock {\em On the complexity of matrix multiplication}.
\newblock PhD thesis, University of Edinburgh, 2010.
\newblock \url{http://hdl.handle.net/1842/4734}.

\bibitem[Str69]{strassen1969gaussian}
Volker Strassen.
\newblock \href {http://dx.doi.org/10.1007/BF02165411} {Gaussian elimination is
  not optimal}.
\newblock {\em Numer. Math.}, 13(4):354--356, 1969.

\bibitem[Str86]{Strassen:1986:AST:1382439.1382931}
Volker Strassen.
\newblock \href {http://dx.doi.org/10.1109/SFCS.1986.52} {The Asymptotic
  Spectrum of Tensors and the Exponent of Matrix Multiplication}.
\newblock In {\em Proceedings of the 27th Annual Symposium on Foundations of
  Computer Science}, SFCS '86, pages 49--54, Washington, DC, USA, 1986. IEEE
  Computer Society.

\bibitem[Str87]{strassen1987relative}
Volker Strassen.
\newblock \href {http://dx.doi.org/10.1515/crll.1987.375-376.406} {Relative
  bilinear complexity and matrix multiplication}.
\newblock {\em J. Reine Angew. Math.}, 375/376:406--443, 1987.

\bibitem[Str88]{strassen1988asymptotic}
Volker Strassen.
\newblock \href {http://dx.doi.org/10.1515/crll.1988.384.102} {The asymptotic
  spectrum of tensors}.
\newblock {\em J. Reine Angew. Math.}, 384:102--152, 1988.

\bibitem[Str91]{strassen1991degeneration}
Volker Strassen.
\newblock \href {http://dx.doi.org/10.1515/crll.1991.413.127} {Degeneration and
  complexity of bilinear maps: some asymptotic spectra}.
\newblock {\em J. Reine Angew. Math.}, 413:127--180, 1991.

\bibitem[Str94]{MR1341854}
Volker Strassen.
\newblock \href {http://dx.doi.org/10.1007/s10107-008-0221-1} {Algebra and
  complexity}.
\newblock In {\em First {E}uropean {C}ongress of {M}athematics, {V}ol.\ {II}
  ({P}aris, 1992)}, volume 120 of {\em Progr. Math.}, pages 429--446.
  Birkh\"auser, Basel, 1994.

\bibitem[Str05]{MR2138544}
Volker Strassen.
\newblock Komplexit\"at und {G}eometrie bilinearer {A}bbildungen.
\newblock {\em Jahresber. Deutsch. Math.-Verein.}, 107(1):3--31, 2005.

\bibitem[Str12]{strassentalk}
Volker Strassen.
\newblock Asymptotic Spectrum and Matrix Multiplication.
\newblock
  \url{https://www.math.uni-konstanz.de/~strassen/pdf/grenoble_long2012.pdf},
  2012.
\newblock ISSAC 2012, Grenoble.

\bibitem[Tao08]{tao2008structure}
Terence Tao.
\newblock {\em Structure and randomness: pages from year one of a mathematical
  blog}.
\newblock American Mathematical Soc., 2008.

\bibitem[Tao16]{tao}
Terence Tao.
\newblock \href
  {https://terrytao.wordpress.com/2016/05/18/a-symmetric-formulation-of-the-croot-lev-pach-ellenberg-gijswijt-capset-bound}
  {A symmetric formulation of the {C}root-{L}ev-{P}ach-{E}llenberg-{G}ijswijt
  capset bound}.
\newblock \url{https://terrytao.wordpress.com}, 2016.

\bibitem[Tob91]{tobler}
Verena Tobler.
\newblock {\em Spezialisierung und Degeneration von Tensoren}.
\newblock PhD thesis, Universität Konstanz, 1991.
\newblock \url{http://nbn-resolving.de/urn:nbn:de:bsz:352-opus-20324}.

\bibitem[TS16]{sawin}
Terence Tao and Will Sawin.
\newblock \href
  {https://terrytao.wordpress.com/2016/08/24/notes-on-the-slice-rank-of-tensors/}
  {Notes on the ``slice rank'' of tensors}.
\newblock \url{https://terrytao.wordpress.com}, 2016.

\bibitem[Tsy09]{MR2724359}
Alexandre~B. Tsybakov.
\newblock \href {https://doi.org/10.1007/b13794} {{\em Introduction to
  nonparametric estimation}}.
\newblock Springer Series in Statistics. Springer, New York, 2009.
\newblock Revised and extended from the 2004 French original, Translated by
  Vladimir Zaiats.

\bibitem[VC15]{vrana2015asymptotic}
P\'eter Vrana and Matthias Christandl.
\newblock \href {http://dx.doi.org/10.1063/1.4908106} {Asymptotic entanglement
  transformation between {W} and {GHZ} states}.
\newblock {\em J. Math. Phys.}, 56(2):022204, 12, 2015.
\newblock \href {http://arxiv.org/abs/1310.3244} {\path{arXiv:1310.3244}}.

\bibitem[VDDMV02]{MR1910235}
F.~Verstraete, J.~Dehaene, B.~De~Moor, and H.~Verschelde.
\newblock \href {http://dx.doi.org/10.1103/PhysRevA.65.052112} {Four qubits can
  be entangled in nine different ways}.
\newblock {\em Phys. Rev. A (3)}, 65(5, part A):052112, 5, 2002.

\bibitem[WDGC13]{MR3087706}
Michael Walter, Brent Doran, David Gross, and Matthias Christandl.
\newblock \href {http://dx.doi.org/10.1126/science.1232957} {Entanglement
  polytopes: multiparticle entanglement from single-particle information}.
\newblock {\em Science}, 340(6137):1205--1208, 2013.
\newblock \href {http://arxiv.org/abs/1208.0365} {\path{arXiv:1208.0365}}.

\bibitem[Wer13]{wernli}
Konstantin Wernli.
\newblock Computing entanglement polytopes, 2013.
\newblock Master thesis.

\bibitem[Wil12]{MR2961552}
Virginia~Vassilevska Williams.
\newblock \href {http://dx.doi.org/10.1145/2213977.2214056} {Multiplying
  matrices faster than {C}oppersmith-{W}inograd [extended abstract]}.
\newblock In {\em S{TOC}'12---{P}roceedings of the 2012 {ACM} {S}ymposium on
  {T}heory of {C}omputing}, pages 887--898. ACM, New York, 2012.

\bibitem[Win99]{MR1725132}
Andreas Winter.
\newblock \href {http://dx.doi.org/10.1109/18.796385} {Coding theorem and
  strong converse for quantum channels}.
\newblock {\em IEEE Trans. Inform. Theory}, 45(7):2481--2485, 1999.
\newblock \href {http://arxiv.org/abs/1409.2536} {\path{arXiv:1409.2536}}.

\end{thebibliography}
\vspace{1em}
\textbf{Matthias Christandl}\\
Department of Mathematical Sciences, University of Copenhagen, Universitetsparken 5, 2100 Copenhagen Ø, Denmark.\\
Email: \href{mailto:christandl@math.ku.dk}{christandl@math.ku.dk}\\[1em]
\textbf{Péter Vrana}\\
Department of Geometry, Budapest University of Technology and Economics, Egry József~u.~1., 1111 Budapest, Hungary.\\
Email: \href{mailto:vranap@math.bme.hu}{vranap@math.bme.hu}\\[1em]
\textbf{Jeroen Zuiddam}\\
Centrum Wiskunde \& Informatica and Institute for Logic, Language and Computation, University of Amsterdam,\\ Science Park~123, 1098~XG Amsterdam, Netherlands. \\
Email: \href{mailto:j.zuiddam@cwi.nl}{j.zuiddam@cwi.nl}
\end{document}